\documentclass[reqno,11pt]{amsart}
\usepackage{amsmath,amssymb,amsthm, verbatim}
\usepackage{url}
\usepackage[usenames, dvipsnames]{color}
\usepackage{mathrsfs}
\usepackage[all, 2cell]{xy}

\usepackage[letterpaper,hmargin=1.5in,vmargin=1in]{geometry}

\usepackage{graphicx}
\usepackage{enumerate,pinlabel}
\usepackage{mathrsfs,graphicx,url}
\usepackage[usenames,dvipsnames]{xcolor}
\usepackage{tikz,tikz-cd}
\usetikzlibrary{arrows.meta}
\usepackage[colorlinks=true,linkcolor=Black,citecolor=Black, urlcolor=Black]{hyperref}

\numberwithin{equation}{section}
\numberwithin{figure}{section}

\theoremstyle{plain}
\newtheorem*{theorem*}{Theorem}
\newtheorem*{lemma*}{Lemma}
\newtheorem{lemma}{Lemma}[section]

\newtheorem{proposition}[lemma]{Proposition}
\newtheorem{prop}[lemma]{Proposition}
\newtheorem{theorem}[lemma]{Theorem}

\newtheorem{corollary}[lemma]{Corollary}

\theoremstyle{definition}
\newtheorem{defn}[lemma]{Definition}

\theoremstyle{remark}

\newtheorem{remark}{Remark}[section]
\newtheorem{rem}{Remark}[section]

\newcommand{\newP}{L}

\newcommand{\ep}{\varepsilon}
\newcommand{\NN}{\mathbb{N}}
\newcommand{\ZZ}{\mathbb{Z}}

\newcommand{\RR}{\mathbb{R}}
\newcommand{\R}{\mathbb{R}}

\newcommand{\SC}{\mathcal{S}}

\newcommand{\Psisc}{\Psi_{\mathrm{sc}}}
\newcommand{\lra}{\longrightarrow}

\newcommand\Ell{\mathrm{ell}}
\newcommand\Char{\mathrm{char}}
\DeclareMathOperator{\WF}{WF}
\newcommand\supp{\mathop{\rm supp}}

\newcommand\coker{\mathop{\rm coker}}
\newcommand\real{\mathop{\rm Re}}

\newcommand{\Rp}{\mathcal{R}_+}

\newcommand\Id{\operatorname{Id}}
\newcommand{\sw}{\mathsf{r}}   

\newcommand\Y{\mathcal{Y}}
\newcommand\ang[1]{\langle #1 \rangle}

\usepackage{color}
\definecolor{darkgreen}{cmyk}{1,0,1,.2}
\definecolor{m}{rgb}{1,0.1,1}
\definecolor{b}{rgb}{0,0.1,1}

\renewcommand{\SC}{\mathcal C}
\newcommand{\SF}{\mathcal F}
\newcommand{\SG}{\mathcal G}

\newcommand{\SM}{\mathcal M}
\newcommand{\SN}{\mathcal N}
\newcommand{\SR}{\mathcal R}
\newcommand{\SU}{\mathcal U}
\newcommand{\SX}{\mathcal X}
\newcommand{\SY}{\mathcal Y}

\renewcommand{\Re}{\mathrm{Re}}
\renewcommand{\Im}{\mathrm{Im}}
\newcommand{\rhob}{\rho_{\mathrm{base}}}
\newcommand{\rhof}{\rho_{\mathrm{fib}}}
\newcommand{\rhor}{\rho_{\mathrm{\mathcal{R}}}}
\newcommand{\Hpar}[2]{H_\mathrm{par}^{#1,#2}}
\newcommand{\Psip}[2]{\Psi^{#1,#2}_\mathrm{par}}
\newcommand{\Psipcl}[2]{\Psi^{#1,#2}_\mathrm{par, cl}}
\newcommand{\Psipd}[2]{\Psi^{#1,#2}_{\delta, \mathrm{par}}}

\newcommand{\br}[1]{\left(#1\right)}  
\newcommand{\Lap}{\Delta}
\newcommand{\Diff}{\mathrm{Diff}}

\newcommand{\Op}{\mathrm{Op}}
\renewcommand{\epsilon}{\varepsilon}
\newcommand{\totz}{Z}
\DeclareMathOperator{\spt}{\mathrm{spt}}
\newcommand{\NP}{\chi_{\mathrm{pol,+}}}
\newcommand{\SP}{\chi_{\mathrm{pol,-}}}
\newcommand{\EQ}{\chi_{\mathrm{eq}}}
\newcommand{\TP}{\chi_{\tau,+}}
\newcommand{\TN}{\chi_{\tau,-}}
\newcommand{\bface}{\partial_{\mathrm{base}} \overline{T}^*_{\mathrm{par}}\mathbb{R}^{n + 1}}
\newcommand{\fface}{\partial_{\mathrm{fiber}} \overline{T}^*_{\mathrm{par}}\mathbb{R}^{n + 1}}

\newcommand\Hdata{\mathcal{W}}
\newcommand\Poi{\mathcal{P}}
\newcommand\Poim{\mathcal{P}_-}
\newcommand\Poip{\mathcal{P}_+}
\newcommand\Poipm{\mathcal{P}_\pm}
\newcommand\swmin{\mathsf{r}_{min}}
\newcommand\swmax{\mathsf{r}_{max}}
\newcommand\Nhat{\widehat{\mathcal{N}}}
\newcommand\Nhatgen{\widehat{\mathcal{N}}_{gen}}
\newcommand\Radm{\mathcal{R}_-}
\newcommand\Radp{\mathcal{R}_+}

\newcommand\Zdual{\Phi}
\newcommand\Rout{{P_+^{-1}}}
\newcommand\Rin{{P_-^{-1}}}

\begin{document}
\title[Fredholm analysis for the
  time-dependent Schr\"odinger equation]{Propagation estimates and Fredholm analysis for the
  time-dependent Schr\"odinger equation}

\author{Jesse Gell-Redman}
\address{School of Mathematics and Statistics, University of Melbourne, Melbourne, Victoria, Australia}
\email{jgell@unimelb.edu.au}
\author{Sean Gomes}
\address{Department of Mathematics and Statistics, University of Helsinki, Helsinki, Finland}
\email{sean.p.gomes@gmail.com}
\author{Andrew Hassell}
\address{Mathematical Sciences Institute, Australian National University, Acton, ACT, Australia}
\email{Andrew.Hassell@anu.edu.au}

\thanks{The authors were supported in part by the Australian Research
  Council through grant DP180100589.  The first author is supported in
part by ARC grant DP210103242}

\begin{abstract}
We study the time-dependent Schr\"odinger operator $P = D_t + \Delta_g
+ V$ acting on functions defined on $\RR^{n+1}$, where, using
coordinates $z \in \RR^n$ and $t \in \RR$, $D_t$ denotes $-i \partial_t$,
$\Delta_g$ is the positive Laplacian with respect to a time dependent
family of non-trapping metrics $g_{ij}(z, t) dz^i dz^j$ on
$\mathbb{R}^n$ which is equal to the Euclidean metric outside of a
compact set in spacetime, and $V = V(z, t)$ is a potential function which is also compactly supported in spacetime. 
In this paper we introduce a new approach to studying $P$, by finding pairs of Hilbert spaces between which the operator acts invertibly. 

Using this invertibility it is straightforward to solve the `final state problem' for the time-dependent Schr\"odinger equation, that is, find a global solution $u(z, t)$ of $Pu = 0$ having prescribed asymptotics as $t \to  +\infty$. These asymptotics are of the form 
$$
u(z, t) \sim t^{-n/2} e^{i|z|^2/4t} f_+\big( \frac{z}{2t} \big), \quad t \to +\infty
$$
where $f_+$, the `final state' or outgoing data, is an arbitrary element of a suitable function space $\Hdata^k(\mathbb{R}^n)$; here $k$ is a regularity parameter simultaneously measuring smoothness and decay at infinity. We can of course equally well prescribe asymptotics as $t \to -\infty$; this leads to incoming data $f_-$. We consider the `Poisson operators' $\Poipm : f_\pm \to u$ and precisely characterise the range of these operators on $\Hdata^k(\mathbb{R}^n)$ spaces. Finally we show that the scattering map, mapping $f_-$ to $f_+$, preserves these spaces. 


 \end{abstract}

\maketitle

\tableofcontents

\section{Introduction and statement of results}

\subsection{Introduction}\label{subsec:intro}
In this article we develop Fredholm theory for the time-depend\-ent Schr\"odinger equation on
$\mathbb{R}^{n + 1} = \mathbb{R}^n_z \times \mathbb{R}_t$, with time-dependent coefficients.  We begin
by studying the inhomogeneous problem
\begin{equation}
  \label{eq:inhomogeneous equation}
  P u(z, t) := \left( \frac{1}{i} \frac{\partial}{\partial t}  + \Lap_{g(t)} +
    V(z, t)\right)u(z, t) = v(z, t)
\end{equation}
where $\Lap_{g(t)}$ denotes the positive Laplacian with respect to the metric $g(t)$, and we assume that 
\begin{multline}\label{eq:Pconditions}
\text{ $g(t) - \delta$ and $V$ are compactly supported in spacetime, and }\\
\text{$g(t)$ is a nontrapping metric on $\RR^n$ for every time $t$.}
\end{multline}
 The condition of compact support is chosen for convenience; it could be weakened to symbolic-type decay estimates (in spacetime) of $g_{ij} - \delta_{ij}$ and $V$. The stronger assumption is made in order to introduce the Fredholm approach to the time-dependent Schr\"odinger operator in a relatively simple (but still variable-coefficient) setting.

It was proven by Lascar in \cite{lascar} that solutions $u$ to
\eqref{eq:inhomogeneous equation} satisfy a propagation of singularities
result which is analogous to the classical theorem of H\"ormander
\cite{Hormander:Existence} but adapted to the parabolic nature of the
equation, i.e. where time derivatives are first order but spatial derivatives are second order, so both $D_t$ and $\Delta_g$ contribute to the principal symbol of $P$. Lascar's result implies in particular that if $Pu$ is smooth,
then singularities of $u$ propagate along $g(t)$-geodesics in space, at a fixed time $t$, 
i.e.\ with `infinite speed'. 

In this paper we prove microlocal propagation estimates for $P$ that are valid uniformly out to spacetime infinity. This is done by adapting propagation estimates, including so-called  `radial point estimates', of Melrose \cite{RBMSpec} and Vasy \cite{VD2013} (which are themselves a microlocal version of the classical Mourre estimate) to this setting. 
In particular, we prove estimates for $u$ in
terms of $v = Pu$ in weighted parabolic Sobolev spaces $H^{s,l}_{par}(\RR^{n+1})$, defined in Section~\ref{sec:scatcalc}, thus taking into account both
the (parabolic) regularity, as measured by $s$, and spacetime growth or decay, as measured by $l$. One essential feature is that we need to work with variable spatial orders, which will vary `microlocally', that is, vary in phase space not just physical spacetime, as will be explained shortly. We will always use a sans-serif font such as $\sw$ to denote variable orders. 
We define, for arbitrary fixed differential order $s$ and suitable variable orders $\sw_\pm$, Hilbert spaces $\mathcal{Y}^{s,\sw_\pm} = H_{\mathrm{par}}^{s,\sw_\pm}(\RR^{n+1})$ and $\mathcal{X}^{s,\sw_\pm}$, given by 
\begin{equation}\label{eq:Xvardef}
\mathcal{X}^{s,\sw_\pm} = \{ u \in H_{\mathrm{par}}^{s,\sw_\pm}(\RR^{n+1}) \mid Pu \in H_{\mathrm{par}}^{s-1,\sw_\pm +1}(\RR^{n+1}) \},
\end{equation}
with the corresponding inner products and norms. 
Our first main result is then the following mapping property for $P$:

\begin{theorem}
	\label{thm:sob.invertible} Assume that $g(t)$ and $V$ satisfy the conditions above. 
	For all $s\in \RR$, for each choice of sign $\pm$ and all
        weight functions $\sw_\pm$ satisfying the conditions in
        Section \ref{subsec:global.fredholm}, the map
	\begin{equation}
		\label{eq:P.mapping}
		P:\mathcal{X}^{s,\sw_\pm}\to \mathcal{Y}^{s-1,\sw_\pm +1}	
	\end{equation}
	is invertible.
\end{theorem}

Theorem~\ref{thm:sob.invertible}, proved in Section \ref{sec:inv proof sec} below,  implies the existence of two inverses (which we will call \emph{propagators}) of $P$, namely
$$
P_+^{-1} : \mathcal{Y}^{s-1,\sw_+ +1} \to \mathcal{X}^{s,\sw_+}
$$
and 
$$
P_-^{-1} : \mathcal{Y}^{s-1,\sw_- +1} \to \mathcal{X}^{s,\sw_-}.
$$
These are in fact `forward' and `backward' propagators. To explain this, consider 
$v \in C_{c}^\infty(\mathbb{R}^{n + 1})$. If one denotes by $T_\pm$ 
the initial and final times of the support of $v$,
$$
T_+ = \sup\{ t : \exists (z, t) \in \supp v \},\ T_- = \inf\{ t : \exists (z, t) \in \supp v \},
$$
there are two special solutions to \eqref{eq:inhomogeneous equation}, the
forward solution $u_+$ and the backward solution $u_-$, which are the
unique solutions satisfying (respectively) $\supp u_+ \subset \{ t \ge
T_- \}$ and $\supp u_- \subset \{ t \le T_+ \}$.  The inverse mappings 
$P_\pm^{-1}$ of \eqref{eq:P.mapping}, which we refer to as the forward 
($+$)/ backward ($+$) propagators, take
$v$ to $\Rout v = u_+, \Rin v = u_-$, in this instance lying in $\mathcal{X}^{s,
  \sw_\pm}$ for arbitrary $s$ (since $v$ is assumed smooth).
The asymptotic behavior of $u_\pm$ as $t \to \pm \infty$ in the
regions $|z| / t < C$ is (see Section~\ref{sec:final state}) 
\begin{equation}
u_\pm(z, t) \sim (4\pi it)^{-n/2} e^{i |z|^2 / t} f_\pm(z/t)
\label{upm}\end{equation}
where $f_\pm$ are Schwartz. The role of the
spacetime decay/growth weight function $\sw_\pm$ is precisely to allow only one of these behaviours; namely, for
example choosing $+$, the weight function $\sw_+$ is subject
to a threshold condition which allow expansions such as \eqref{upm} for $t \to +\infty$, but not such expansions as $t \to -\infty$. For $\sw_-$, the reverse is true. See Section~\ref{subsec:global.fredholm} for the precise conditions on $\sw_\pm$. 

Let us elaborate on how the variable spatial orders $\sw_\pm$ allow or disallow expansions such as \eqref{upm}. Given $f_\pm \in \mathcal{S}(\RR^n)$, we note that \eqref{upm} is in the weighted space $\ang{z, t}^{-l} L^2(\RR^{n+1})$ for $l < -1/2$, but not for $l = -1/2$. The value $-1/2$ is thus a threshold value; whether the spatial weight is greater or less than this threshold value determines whether \eqref{upm} for $f \neq 0$ is possible for a function in the corresponding weighted space. Microlocally the expression \eqref{upm} is concentrated at the `outgoing radial set' $\mathcal{R}_+$ for $t \to +\infty$ and the `incoming radial set' $\mathcal{R}_-$ for $t \to -\infty$; these are the limiting points of bicharacteristics\footnote{By a \emph{bicharacteristic} of $P$ we mean an integral curve of the Hamilton vector field of $p = \sigma(P)$ contained within $\{ p = 0 \}$.} of $P$ at spacetime infinity, that is, the initial/final points of bicharacteristics of $P$ on the compactified phase space. The key property of the weight $\sw_+$ is therefore that $\sw_+$ is less than $-1/2$ on $\mathcal{R}_+$ but greater than $-1/2$ on $\mathcal{R}_-$, thereby allowing elements of $\mathcal{X}^{s, \sw_+}$ to have asymptotics  \eqref{upm} as $t \to +\infty$ but not for $t \to -\infty$.  For $\sw_-$, the reverse is true. 

The above discussion may give the impression that the $\sw_\pm$ could be chosen to depend only on $t$, and therefore, these weights do not need to vary `microlocally'. 
In fact, this is not the case. Along bicharacteristics at fibre-infinity (that is, where the frequency variables are infinite), the time $t$ is fixed and can be any finite value, so the value of $t$ cannot be used to distinguish between these two sets. What determines whether we are at $\mathcal{R}_\pm$ is the relative orientation of the spatial variable $z$ and its dual variable $\zeta$. Let $\hat z = z/|z|$ and $\hat \zeta = \zeta/|\zeta|$. At the incoming radial set, we have $\hat z \cdot \hat \zeta = -1$ while at the outgoing radial set, we have $\hat z \cdot \hat \zeta = +1$. Thus $\sw_\pm$ need to be functions of both $z$ and $\zeta$ (at least), and in particular, the weights must vary nontrivially in phase space, not just in spacetime. 

Theorem~\ref{thm:sob.invertible} immediately implies that we can solve the `final state problem' with prescribed outgoing data $f_+$ in the sense of \eqref{upm}. For now we only consider Schwartz $f_+$; in Theorem~\ref{thm:Poissonbounds-intro} we will treat  distributional $f_+$. 

\begin{theorem}\label{thm:finalstate-intro} Given $f_+ \in \mathcal{S}(\RR^n)$ there is a unique solution to the equation $Pu = 0$ with asymptotics \eqref{upm} as $t \to +\infty$. Moreover, for every real $s$, and variable orders $\sw_\pm$ satisfying the conditions in Section~\ref{subsec:global.fredholm},  $u$ lies in the space 
$$
u \in \mathcal{X}^{s, \sw_+} + \mathcal{X}^{s, \sw_-}.
$$
We write $u = \Poip f_+$ and refer to $\Poip$ as the outgoing `Poisson operator'. 
\end{theorem}

See Section \ref{sec:final state} for the simple proof in the case of Schwartz data. Thus, global solutions to $Pu = 0$ do not lie in one or other of our function spaces, but in the sum of the two. Given a (suitable) global solution, this decomposition is easily effected using a microlocal partition of unity, $\Id = Q_- + Q_+$, where $Q_-$ is microlocally equal to the identity near $\mathcal{R}_-$ and microlocally trivial near $\mathcal{R}_+$. Then $u = Q_- u + Q_+ u$ gives a decomposition such that $u_\pm \in \mathcal{X}^{s, \sw_\pm}$.

\vskip 5pt

One of the main goals of this work is to establish foundational theory for a microlocal/Fredholm approach to the \emph{nonlinear} Schr\"odinger equation along the lines of \cite{NLSM, NEH} for the nonlinear Helmholtz equation, which in turn was inspired by works \cite{HVsemi}, \cite{GHV2016} on nonlinear wave equations. The first step in this direction has been taken in \cite{NLS1}. To this end,  we include \emph{module regularity estimates} which typically arise in 
microlocal approaches to nonlinear analysis. The notion of module regularity was formalized in \cite{HMV2004} although it goes back much further; for example, the definition of Lagrangian distribution given by H\"ormander in \cite{Ho4} (which he credits to Melrose) is in terms of module regularity. In the present context, it is closely related to Klainerman's vector field method \cite{klainerman1985}.   Thus we also prove
refined mapping properties for $P$ in which the spaces
$\mathcal{X}^{s,\sw_\pm}, \mathcal{Y}^{s-1,\sw_\pm +1}$
are replaced by spaces in which regularity is measured with respect to
iterated application of elements in the module of operators which are
characteristic on $\mathcal{R}_\pm$.  
We distinguish between the module $\mathcal{N}$ of operators which are
characteristic at both radial sets
$\mathcal{R}_\pm$ simultaneously and the larger modules of operators
$\mathcal{M}_+$ and $\mathcal{M}_-$ vanishing at $\mathcal{R}_+$ and
$\mathcal{R}_-$, respectively.  The module $\mathcal{N}$ is generated
(up to precomposition by globally elliptic operators)
by a finite collection of operators which
correspond directly to the natural invariance properties of the free
Schr\"odinger equation; these are the generators of translation, of rotations, and of
Galilean transformations.  We let $H_\pm^{s,l;\kappa,k}$ denote the elements of the space $H_{\mathrm{par}}^{s,l}(\RR^{n+1})$ with $k$ orders of small module regularity and $\kappa$ additional orders of module regularity with respect to $\mathcal{M}_\pm$ (see Definition~\ref{def:2modules}). Defining, analogously to \eqref{eq:Xvardef}, 
$\mathcal{Y}_\pm ^{s,l;\kappa,k} = H_\pm^{s,l;\kappa,k}$ and 
\begin{equation}\label{eq:Xmoddef}
\mathcal{X}_\pm^{s,l;\kappa,k} = \{ u \in H_\pm^{s,l;\kappa,k} \mid Pu \in H_\pm^{s-1,l+1;\kappa,k}\},
\end{equation}
then for suitable $l$ and $\kappa$ we also obtain a Hilbert space isomorphism 
	\begin{equation}
	\label{eq:mod.invertible.intro}
		P:\mathcal{X}_\pm^{s,l;\kappa,k}\to \mathcal{Y}_\pm^{s-1,l+1;\kappa,k}.
	\end{equation}
See Proposition \ref{thm:mod.invertible}.

We note that, if $\kappa$ is at least $1$, then we can take the spatial order $l$ in \eqref{eq:mod.invertible.intro} to be \emph{constant} in the range $-3/2 < l < -1/2$. The reason for this is that (choosing the + sign arbitrarily) the $\mathcal{M}_+$ module is elliptic at $\mathcal{R}_-$, in effect raising the spatial regularity weight at $\mathcal{R}_-$ by one order, thus raising it above the threshold value of $-1/2$ (since $l > -3/2$), while not affecting the regularity at $\mathcal{R}_+$. Being able to take a constant spatial order is advantageous when proving multiplicative properties, as has been explained in \cite{HVsemi}, and will be important in our planned future work on the nonlinear Schr\"odinger equation.

Moreover, the consideration of module regularity spaces enables us to
prove a precise scattering result in terms of natural 
spaces $\mathcal{W}^k(\mathbb{R}^n)$ of incoming/outgoing data of solutions to
$Pu = 0$. Here $k \in \ZZ$ is a regularity index, measuring both smoothness and decay at infinity, and is such that $\cap_k \mathcal{W}^k(\mathbb{R}^n) = \mathcal{S}(\RR^n)$, while $\cup_k \mathcal{W}^k(\mathbb{R}^n) = \mathcal{S}'(\RR^n)$.   
The $\mathcal{W}^k(\RR^n)$ for $k \geq 1$ are themselves module regularity spaces,
relative to a module $\Nhat$ induced by the small module $\mathcal{N}$
above. In fact, these modules are such that the free Poisson operator
$\mathcal{P}_0$, i.e.\ the Poisson operator from Theorem
\ref{thm:finalstate-intro} for the flat Euclidean metric with $V
\equiv 0$, intertwines $\Nhat$ and $\mathcal{N}$. We then show that the Poisson
operators $\Poipm$ extend from Schwartz space to all tempered
distributions, and we precisely characterise the range of $\Poipm$ on
$\mathcal{W}^k(\RR^n)$.  In the theorem below, the weight functions
$\sw_\pm$ are chosen as in Section
\ref{subsec:global.fredholm}, in particular they are equal to $-1/2$
off a neighborhood of the radial sets on the characteristic set:

\begin{theorem}\label{thm:Poissonbounds-intro}
 For $k \in \NN$, the range of the Poisson operator $\Poip$ on $\mathcal{W}^k(\RR^n)$ is precisely 
$$
\{ u \in \mathcal{X}_+^{1/2, \sw_+; k,0}(\RR^{n+1}) + \mathcal{X}_-^{1/2, \sw_-; k,0}(\RR^{n+1}) \mid Pu = 0 \}, 
$$
i.e. that is, those elements of $\mathcal{X}^{1/2, \sw_+} + \mathcal{X}^{1/2, \sw_-}$ in the kernel of $P$ having module regularity of order $k$. %

For $k \leq -1$, the range of $\Poip$ on $\mathcal{W}^k(\RR^n)$ is precisely the elements of $\ker P$ in 
$$
\{ u \in H_{\mathrm{par}}^{k+1/2, k-1/2}(\RR^{n+1}) \mid Pu = 0 \}. 
$$
Provided that  $k \geq 2$, the global solution $u = \Poip f_+$ admits the asymptotic \eqref{upm} in the precise sense that 
\begin{equation}
 \lim_{t \to +\infty} (4\pi it)^{n/2} e^{-it|\zeta|^2} u_+(2t \zeta, t) = f_+(\zeta)
\end{equation} 
as a limit in the space $\ang{\zeta}^{1/2 + \epsilon}\Hdata^{k-2}(\RR^n_\zeta)$. 
\end{theorem}

\begin{remark} The difference between the two cases $k \geq 0$ and $k \leq -1$ is that, in the latter case, the spacetime regularity is everywhere below threshold, so nothing special happens at the radial sets, while in the former case, the spacetime regularity must drop to below threshold at the radial sets. In both cases, given a solution to $Pu = 0$, its outgoing data is in $\mathcal{W}^k(\RR^n)$ if and only if it is microlocally in $H_{\mathrm{par}}^{k+1/2, k-1/2}(\RR^{n+1})$ away from the radial sets.
\end{remark}

Moreover, we show that the scattering map, which maps the incoming data $f_-$ of global solutions $u$ to the outgoing data $f_+$, preserves the spaces $\mathcal{W}^k(\RR^n)$:

\begin{theorem}\label{thm:sc-intro}
The scattering map $S$, initially defined for $f_- \in \mathcal{S}(\RR^n)$, extends to a bounded map from $\mathcal{W}^k(\RR^n)$ to itself for each $k \in \ZZ$. 
\end{theorem}

\subsection{Parabolic calculus}

We begin by developing the calculus of parabolic pseudodifferential
operators on $\mathbb{R}^{n + 1}$.  These are quantizations of symbols
$S^{m,l}_{\mathrm{par}}(\mathbb{R}^{n + 1})$, of fibre (or differential) order $m$ and spacetime order $l$, defined using the standard 
spacetime weight function $(1 + |z|^2 + t^2)^{1/2}$ and the
\textit{parabolic} weight function
$(1 + |\zeta|^4 + \tau^2)^{1/4}$ in the dual variables --- see \eqref{symbolest} for the precise definition.  Thus,
unlike usual pseudodifferential operators, classical parabolic pseudodifferential operators
do not have principal symbols that are homogeneous functions on
$\tau, \zeta$ in the standard sense but instead are homogeneous with
respect to the \emph{parabolic scaling} 
\begin{equation}\label{eq:parabolic.scaling}
(\zeta, \tau) \mapsto (c \zeta, c^2 \tau), \quad c > 0.
\end{equation}
In addition, the behavior in the spacetime variables $(z, t)$ is
assumed to be uniformly symbolic in the usual sense.  All of this is
accomplished through the introduction of radial
compactification of spacetime and parabolic compactification of the
scattering cotangent bundle,
${}^{sc} \overline{T}^{*}_{\mathrm{par}}(\mathbb{R}^{n + 1})$.  As in other
Fredholm analysis of non-elliptic operators, it is convenient to have
variable order Sobolev spaces at our disposal, and in our case, as it
is only necessary to have variable spatial decay weights, we define
spaces of pseudodifferential operators, $\Psi^{s, \sw}_{\mathrm{par}}(\mathbb{R}^{n + 1})$,  for constant
$s \in \mathbb{R}$ and $\sw \in S^{0,0}_{\mathrm{par}}(\mathbb{R}^{n +
  1})$.  Here $s$ is the order of
(parabolic) differential regularity and $\sw$ is the
spatial decay order. Then choosing an elliptic and invertible element $A$ of this space, 
we define
\begin{equation}
H^{s, \sw}_{\mathrm{par}} = \{ u \in H^{-M, -N}_{\mathrm{par}} \mid Au \in L^2 \},
\end{equation}
for any sufficiently large $M, N$. In this space, the (parabolic) differential order is fixed at 
$s$, but the order of decay at spacetime infinity varies microlocally.

The Schr\"odinger operator is a differential operator of order $(2, 0)$ lying in
this parabolic calculus, whose characteristic set $\Sigma(P)$ contains two disjoint
`radial sets' $\mathcal{R}_\pm$ mentioned above. These are submanifolds of sources ($-$)
and sinks ($+$) for the rescaled Hamilton flow on the characteristic
set $\Sigma(P)$.  
We show that this parabolic calculus enjoys 
structures and features similar to those of Melrose's scattering calculus,
including extensions of the notions of characteristic set and
rescaled-Hamilton flow to the boundaries, introduced by compactification, at both spacetime
and fiber infinity.  As in the scattering calculus, this
allows one to formulate and prove propagation estimates, including at the radial sets $\mathcal{R}_\pm$, uniformly up to spacetime and fibre infinity.  By adapting positive
commutator estimates introduced originally by H\"ormander, and developed by Melrose \cite{RBMSpec} and Vasy \cite{VD2013}, to the parabolic calculus, we prove microlocal propagation estimates for $P$ in each region of phase space, which we put together to obtain global Fredholm estimates. These take the form, in the setting of Theorem~\ref{thm:sob.invertible},
\begin{equation}\label{eq:Fred1}
\| u \|_{s, \sw_+} \leq C \Big( \| Pu \|_{s-1, \sw_+ +1} + \| u \|_{M, N} \Big), \quad u \in \mathcal{X}^{s, \sw_+}, 
\end{equation}
together with the dual estimate 
\begin{equation}\label{eq:Fred2}
\| u \|_{s', \sw_-} \leq C \Big( \| P^* u \|_{s'-1, \sw_- +1} + \| u \|_{M', N'} \Big), \quad u \in \mathcal{X}^{s', \sw_-}.
\end{equation}
Here $s$ is arbitrary and $M, N$ should be thought of as very negative, so that $H^{s, \sw_+}_{\mathrm{par}}$ embeds compactly into $H^{M, N}_{\mathrm{par}}$ (thus, we require $M < s$ and $N < \inf \sw_+$, and similarly for the second estimate). For convenience, we assume here that the potential function $V$ is real, so that $P = P^*$. As shown by Vasy, these estimates imply that $P$ is a Fredholm map from 
$$
\{ u \in H^{s, \sw_\pm}_{\mathrm{par}} \mid Pu \in H^{s-1, \sw_\pm+1}_{\mathrm{par}} \} \to H^{s-1, \sw_\pm+1}_{\mathrm{par}},
$$
that is, between $\mathcal{X}^{s, \sw_\pm}$ and $\mathcal{Y}^{s-1, \sw_\pm +1}$ in our notation. 
Moreover, the formal self-adjointness of $P$ implies that the index of $P$, mapping between these spaces, is zero. So $P$ is invertible between these spaces if and only if its null space is trivial. This triviality is easy to show by considering the evolution in time of the spatial $L^2$ norm of global solutions.

\subsection{Relation to previous literature}

There have been many approaches to solving variable-coefficient time-dependent Schr\'odinger equations,
including via ODE methods in Banach spaces \cite{Kato1953}, approximating Feynman integrals \cite{Fuji}, via oscillatory integrals \cite{Treves, KS, IK1985, HW2005} or via the FBI transform (e.g. \cite{ST2002}). Our approach is, to the best of our knowledge, essentially different to any previous method for treating the time-dependent Schr\"odinger equation, although inspired by previous Fredholm treatments of non-elliptic problems for the wave equation \cite{VD2013, BVW2015, HVsemi, GHV2016} and the Helmholtz equation \cite{NEH}. The first example of Fredholm theory used to treat a non-elliptic problem appears to be Faure and Sj\"ostrand's treatment of Anosov flows in 
\cite{FS2011}. This appears at first sight to be very different in nature to the other treatments, but Dyatlov and Zworski \cite{DZ2016} showed that this example in fact fits into the general framework set out by Vasy in \cite{VD2013}. Fredholm theory in a Lorentzian (hence non-elliptic) setting was considered by B\"ar and Strohmaier in \cite{BS19}. 
Very recently, Sussman has used microlocal propagation estimates to study the Klein-Gordon equation \cite{Sussman}. 

Our work builds off the results of Lascar on inhomogeneous
pseudodifferential operators and the geometric microlocal scattering
theory of Melrose \cite{lascar,RBMSpec}.  The former develops a
general theory of operators with inhomogeneous symbols, and extends many
of the standard structures and theorem in microlocal analysis to these
operators, including propagation of singularities.  Lascar's work is
local in nature, and does not lead directly to quantitative global
estimates.  The global microlocal perspective imparts exactly that; as
in Melrose's work, the notion of (parabolic) wavefront set of
distributions on spacetime can be extended up to and including the
introduced boundary via radial compactification. We also use the module regularity formalism introduced by the third author with Melrose and Vasy in \cite{HMV2004}. 

There is a vast literature on scattering theory for the Schr\"odinger equation, that we will not attempt to discuss here. See for example the monographs \cite{RS3},  \cite{Yafaevbook1992, Yafaevbook2010} or \cite{DGbook}. Relatively little of this literature treats the case of time dependent metrics or potentials. Yafaev \cite{Yaf80, Yaf82} wrote several studies on wave operators for time-dependent potentials (including periodic potentials), and Chapter 3 of \cite{DGbook} is devoted to time-decaying potentials. Rodnianski and Schlag \cite{MR2038194} considered rough and time-dependent potentials, and more recently Soffer and Wu \cite{SofferWu} proved local decay for NLS with time-dependent potentials.

Asymptotic decay of solutions to Schr\"odinger's equation is a widely
studied topic, going back at least to Jensen and Kato
\cite{MR544248}.  There, as in the general results on the scattering
operator in \cite{Kato_book}, the Hamiltonians under consideration are
time-independent.

This work is intended to be a foundation for a wide-ranging program of research into nonlinear Schr\"odinger operators with nonlinearity polynomial in $u$ and $\overline{u}$. 
We expect that our method will be advantageous for analyzing the large-time asymptotics of solutions. Indeed, in \cite{NLS1}, combining the linear theory in the present work with a multiplication result for module regularity spaces, along the lines of \cite{NEH} in the Helmholtz case, leads directly to a small-data result for NLS \cite{NLS1}. Large data results should be achievable by combining our techniques with a priori estimates on solutions provided by Strichartz or Morawetz estimates. 
In the focusing case, we anticipate that the method --- after some further development --- will be effective in analyzing the interaction of solitons and radiation. This will require developing a Fredholm approach to `three-body-type' potentials, which is a topic of independent interest and one currently being pursued.

\subsection{Structure of the paper}
In Section~\ref{sec:scatcalc}, we discuss the compactification of phase space and set up the parabolic scattering calculus. We obtain standard results for composition, $L^2$-boundedness for zeroth order operators, and elliptic parametrices, and define weighted parabolic Sobolev spaces including variable order weights. 

In Section~\ref{sec:schrod geom}, we discuss the geometry of the characteristic variety $\Sigma(P)$ in our compactified phase space, and particularly properties of the Hamilton vector field relative the radial sets $\mathcal{R}_\pm$. This geometry, particularly the fact that $\mathcal{R}_-$ is a source, and $\mathcal{R}_+$ a sink, for the (rescaled) bicharacteristic flow, is crucial for the estimates in Section~\ref{sec:propfred}. 

In Section~\ref{sec:modregdef} we introduce the modules $\mathcal{M}_\pm$ and $\mathcal{N}$ with respect to which we shall prove module regularity estimates, and derive a  basic positivity property that makes the iterative module regularity argument possible. 

In Section~\ref{sec:propfred} we give the microlocal propagation estimates that we need to assemble the global Fredholm estimate, as in \eqref{eq:Fred1}, \eqref{eq:Fred2}. These estimates are actually proved in Section~\ref{sec:fredholm} for general operators in the calculus obeying some structural conditions. 

In Section~\ref{sec:solv} we show invertibility of $P$ both in the case of weighted parabolic Sobolev spaces with variable weights, and in the case of module regularity spaces. This establishes Theorem~\ref{thm:sob.invertible}. We deduce solvability of the final state problem for Schwartz outgoing data, proving Theorem~\ref{thm:finalstate-intro}. 

In Section~\ref{sec:scat mat}, we define the spaces $\mathcal{W}^k(\RR^n)$ of incoming and outgoing boundary data, and analyze the Poisson operator and scattering map on these spaces, proving Theorems~\ref{thm:Poissonbounds-intro} and \ref{thm:sc-intro}.

In the appendix, Section \ref{sec:fredholm}, we prove various propagation estimates
which we apply to $P$, including radial points estimates and module regularity
propagation estimates..  To maximize the
utility of these results, we work in a general setting analogous to
that of \cite{grenoble}, in which we assume only that the operator
under consideration has a non-degenerate characteristic set with
smooth submanifolds of radial sets.

\subsection{Acknowledgements} The authors thank Andras Vasy and Peter Hintz for their encouragement and for several enlightening conversations. They also thank MATRIX for its hospitality during the workshop ``Hyperbolic Differential Equations in Geometry and Physics'' during April 2022.


\section{Scattering calculus and parabolic scattering calculus}
\label{sec:scatcalc}

\subsection{Definition of the parabolic scattering calculus}
In order to make use of propagation estimates in our present setting, it is necessary for us to work with a calculus of pseudodifferential operators which contains $P=\Lap-i\partial_t$ as an operator of principal type. 

Such a calculus must be anisotropic, so that $i\partial_t$ and $\Lap$ can be viewed as operators of the same order. Anisotropic calculi with this property are considered in \cite{lascar}, \cite{Hormander:Existence}, as well as propagation estimates at ``interior" points. We shall require a scattering version of this calculus in order to obtain propagation estimates along bicharacteristics lying in the boundary of the radial compactification (in each factor) of $T^*\RR^{n+1}\cong \RR^{n+1}\times\RR^{n+1}$.

\begin{figure}\label{fig:rad comp spacetime}
   \centering
    \def\svgwidth{70mm}
    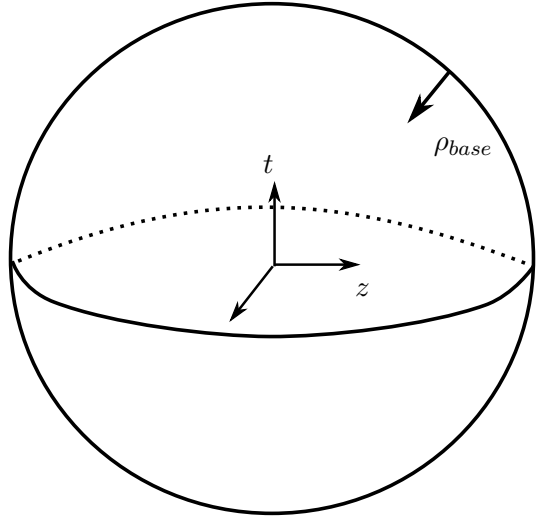
    \caption{The radially compactified spacetime.  The function
  $\rho_{base} = (1 + t^2 + |z|^2)^{-1/2}$ defines (i.e.\ vanishes at) spacetime
  infinity.  The radial set $\mathcal{R}_+$ lies over the top
  hemisphere, while the radial set $\mathcal{R}_-$ lies over the bottom
  hemisphere.}
\end{figure}
  
We denote the elements of $T^*\RR^{n+1}$ as $(z,t,\zeta,\tau)$ where
$z,\zeta\in \RR^n$ and $t,\tau\in\RR$.  Our parabolic pseudodifferential
calculus will consist of operators which are quantizations of symbols
defined with respect to an anisotropic weight function $R = R(\zeta, \tau)$, 
defined by 
	\begin{equation}\label{eq:fiber weight function}
 R^4 = |\zeta|^4 + \tau^2, \quad R > 0. 
            \end{equation}
\begin{defn}\label{def:parab symbol spaces}
  For $m,l\in\RR$, $a \in C^\infty( T^* \RR^{n + 1})$, define the seminorms,
  $$
\| a \|_{S^{m,l}_{\mathrm{par}}, N} = \!\!\!\! \sum_{|\alpha| + k + |\beta| + j \le N} \sup_{T^* \mathbb{R}^{n + 1}} \left| \langle(z,t)\rangle^{-(l-|\alpha|-k)}\langle R \rangle ^{-(m-|\beta|-2j)} \partial_z^\alpha\partial_t ^k \partial_\zeta ^\beta\partial_\tau^j a(z, t, \zeta, \tau)\right|.
  $$
  We denote the Fr\'echet space defined by these seminorms by
	\begin{equation}\label{symbolest}
		S_{\mathrm{par}}^{m,l}(T^*{\RR^{n+1}}) =\{a\in\SC^\infty(T^*{\RR^{n+1}}):
                \| a \|_{S^{m,l}_{\mathrm{par}}, N} < \infty, \mbox{ for all } N
                \in \mathbb{N}_0
                \}
              \end{equation}
      \end{defn}
This is a statement about decay of $a$ and its derivatives in both the
spatial and momentum variables.  Indeed, it is easy to see that $R$ is comparable to $\max(\ang{\zeta}, \ang{\sqrt{|\tau|}})$. 
Thus an $a \in S^{1,0}_{\mathrm{par}}$ grows at most like $|\tau|^{1/2}$ in the 
region $|\zeta|^2 / |\tau| <  2$ and at most like $|\zeta|$ in the region $|\tau|/|\zeta|^2 < 2$.

Note that the residual symbol space is exactly the Schwartz functions:
      $$
S_{\mathrm{par}}^{-\infty,-\infty}(T^*{\RR^{n+1}}) := \bigcap_{m, l}
S_{\mathrm{par}}^{m,l}(T^*{\RR^{n+1}}) = \mathcal{S}(\mathbb{R}^{n + 1} \times \mathbb{R}^{n + 1}).
      $$

A careful development of the local properties of the
pseudodifferential calculus obtained by quantizing these
symbols can be found in \cite{lascar}, \cite{Hormander:Existence}.  We briefly
summarize some of its properties.
\begin{defn}
	The class $\Psi_{\mathrm{par}}^{m,l}(\RR^{n+1})$ of parabolic pseudodifferential operators corresponding to $S_{\mathrm{par}}^{m,l}(T^*\RR^{n+1})$ are the operators with Schwartz kernels
	\begin{equation}\label{eq:kernel}
q_L(a) = 		(2\pi)^{-n-1}\int_{\RR^{n+1}}e^{i(x-y)\cdot \zeta+(t-s)\tau}a(z,t,\zeta,\tau)\, d\zeta\, d\tau
	\end{equation}
	for some $a\in S_{\mathrm{par}}^{m,l}(T^*\RR^{n+1}) $, where
        the integral \eqref{eq:kernel} is interpreted in the
        distributional sense.

        We also denote $\Diff^{m,l}_{\mathrm{par}}(\mathbb{R}^{n +
          1})$ denote the set of differential operators in
        $\Psi^{m,l}_{\mathrm{par}}(\mathbb{R}^{n + 1})$.
\end{defn}

Some results about operators $A \in \Psi_{\mathrm{par}}^{m,l}(\RR^{n+1})$ can be obtained easily from the
containment 
\begin{equation}
S^{m,l}_{\mathrm{par}}(\mathbb{R}^{n + 1}) \subset \begin{cases} S^{m,l}_{1/2,
  0}(\mathbb{R}^{n + 1}), \quad m \geq 0 \\
  S^{m/2,l}_{1/2,
  0}(\mathbb{R}^{n + 1}), \quad m \leq 0 \end{cases}
  \label{eq:basic containment}
\end{equation}
where $S^{m,l}_{\delta, \delta'}(\mathbb{R}^{n + 1})$ are the standard
scattering symbol spaces, as follows immediately from the definition. 
Thus e.g.\ we conclude that $A$ maps $\mathcal{S}(\mathbb{R}^{n + 1})$
to itself.

\subsection{Compactification of phase space}
As in the standard scattering calculus (see \cite{RBMSpec}), we compactify
$T^*\RR^{n+1}\cong \RR^{n+1}\times\RR^{n+1}$ in each factor, but, as
we describe now, we do so inhomogeneously in the dual (momentum) variables.

We compactify the spatial factor using the standard radial
compactification, namely using the map $\totz=(z,t)\in \RR^{n+1}$ 
\begin{equation}
	\psi_1(\totz)= \br{\frac{1}{(1+|\totz|^2)^{1/2}},\frac{\totz}{(1+|\totz|^2)^{1/2}}}\in S^{n+1}_+,
      \end{equation}
      where $S^{n+1}$ is the unit sphere in $\mathbb{R}^{n +
        2}$ and $S^{n+1}_+$ the half sphere with first coordinate nonnegative. 
Since $\psi_1$ is a diffeomorphism onto the interior of $(S^{n + 1}_+
)^\circ$, we may abuse notation and allow $\totz$ to denote the
corresponding interior point $\psi_1(\totz)$.  With this notation, the
function 
\begin{equation}\label{rho-defn}
\rho_{base} = \frac{1}{(1+|\totz|^2)^{1/2}}
\end{equation}
 extends smoothly to
all of $S^{n + 1}_+$ and is a boundary defining function (bdf) of $\partial
S^{n + 1}_+$ (meaning $\partial S^{n + 1}_+ = \{ \rho_{base} = 0 \}$
and $d\rho_{base}$ is non-vanishing over $\partial S^{n + 1}_+$).
A function $f \in C^\infty(S^{n + 1}_+)$ thus equivalently satisfies
that $f(\totz)$ is smooth and on $|\totz| > C$, $f  \in C^\infty([0,
1)_{\rho_{base}} \times \partial{S^{n + 1}_+})$.
It is straightforward to show that away from $\totz = 0$, $1/|\totz|$ is also
a boundary defining function, so smoothness of $f$ means that, with $\hat \totz = \totz/|\totz|$,
$$
f(1/|\totz|, \hat \totz) \sim \sum_{j = 0}^\infty |\totz|^{-j} a_j(\hat \totz)
$$
where $\hat \totz \in \partial{S^{n + 1}_+}$ and $a_j \in
C^\infty(\partial{S^{n + 1}_+})$.

For the other factor, $\RR^{n+1}_{\zeta, \tau}$, we consider a `parabolic sphere' given by 
\begin{equation}
S^{n+1}_{par, +} = \{ (\zeta_0, \zeta, \tau) \in \RR^{n+2} \mid \zeta_0^4 + |\zeta|^4 + \tau^2 = 1, \ \zeta_0 \geq 0 \}.
\end{equation}
Consider the smooth mapping
$$
\psi_2 \colon \mathbb{R}^{n + 1}_{\zeta, \tau} \longrightarrow S^{n + 1}_{par, +}
$$
given by 
\begin{equation}
	\psi_2(\zeta, \tau)= \br{\frac{1}{(1+R^4)^{1/4}},\frac{\zeta}{(1+R^4)^{1/4}},\frac{\tau}{(1+R^4)^{1/2}}}\in  S^{n+1}_{par, +}
\end{equation}
where $R$ is as in \eqref{eq:fiber weight function}. This  is a diffeomorphism onto the interior $(S^{n + 1}_{par, +} )^\circ$ (using the fact that $R^4$ is a smooth function).
This defines our parabolic compactification of the fibres $R^{n+1}_{\zeta, \tau}$, namely we take the closed parabolic half-sphere $S^{n+1}_{par, +}$ as the compactification. This shows that 
\begin{equation}\label{rhofib}
\rho_{fib} = (1+R^4)^{-1/4}
\end{equation}
is a boundary defining function for the boundary of this compactified space, which we shall call `fibre-infinity'. Natural `angular' variables are induced by smooth coordinates on the boundary of the parabolic half-sphere, extended into the interior by requiring them to be invariant under the parabolic scaling \eqref{eq:parabolic.scaling}. 
In the region 
$|\zeta|^2 / |\tau| <  2$, $|\tau| \geq 1$ we can use the angular coordinates $\zeta_i/\sqrt{|\tau|}$ and the `radial' (i.e.\ homogeneous of degree 1 with respect to the parabolic scaling) coordinate $\sqrt{|\tau|}$, while in the region $|\tau|/|\zeta|^2 < 2$, $|\zeta| \geq 1$, we can assume without loss of generality that $|\zeta_1| \geq \max_i |\zeta_i|/2$ locally, and then we can use angular coordinates $\zeta_j/\zeta_1$, $j = 2, \dots, n$ together with $\tau \zeta_1^{-2}$, and the radial coordinate $\zeta_1$. 
We will call the first region $\tau$-dominant, and the second type of region $\zeta$-dominant. (Of course there is an overlap region which is both $\tau$-dominant and $\zeta$-dominant.)

We thereby obtain a compactification
\begin{equation}
  \label{eq:compactification}
  \psi = (\psi_1, \psi_2) \colon T^*\mathbb{R}^{n + 1} \longrightarrow
  S^{n+1}_+ \times S^{n + 1}_+ =: \overline{T}^*_{\mathrm{par}}\mathbb{R}^{n + 1}.
\end{equation}
This is a manifold with corners of codimension two, the boundary being
a union of two boundary hypersurfaces
\begin{equation}
  \label{eq:ends}
    \partial \overline{T}^*_{\mathrm{par}}\mathbb{R}^{n + 1} = \left\{ \rho_{base} = 0 \right\} \cup \left\{ \rho_{fib} = 0
    \right\} = \partial_{base} \overline{T}^*_{\mathrm{par}}\mathbb{R}^{n + 1} \cup \partial_{fib} \overline{T}^*_{\mathrm{par}}\mathbb{R}^{n + 1},
  \end{equation}
  where
  $$
    \partial_{base} \overline{T}^*_{\mathrm{par}}\mathbb{R}^{n + 1}  \simeq \partial S^{n + 1}_+  \times S^{n +
      1}_{par, +} \mbox{ is ``spacetime infinity''}
    $$
    and
    $$
  \partial_{fib} \overline{T}^*_{\mathrm{par}}\mathbb{R}^{n + 1}   \simeq S^{n + 1}_+ \times
    \partial S^{n+1}_{par, +} \mbox{ is ``fiber infinity''}. 
    $$

Similar to how we write the variables $(z, t)$ collectively as $Z$, we shall write $(\zeta, \tau)$ collectively as $\Zdual$. We also write 
$$
|\Zdual| = R = (|\zeta|^4 + \tau^2)^{1/4} 
$$
(this would more correctly be denoted $|\Zdual|_{\mathrm{par}}$ or similar, but we prefer the simpler notation), and somewhat imprecisely write $\hat \Zdual$ for a set of $n$ angular variables, which can take various forms as described above. We also write $\ang{\Zdual} := (1 + R^4)^{1/4}$. 
  
As in the standard scattering calculus, classical symbols can be
defined using smooth functions on $\overline{T}^*_{\mathrm{par}}\mathbb{R}^{n +
  1}$. To be absolutely concrete, a function $a_0$ lies in
$C^\infty(\overline{T}^*_{\mathrm{par}}\mathbb{R}^{n + 1})$ if and only if
$a_0$ is smooth in the interior of $\overline{T}^*_{\mathrm{par}}\mathbb{R}^{n
  + 1}$ and satisfies the following, 
\begin{itemize}
\item \textbf{Spatial infinity/momentum interior:} On sets $|\totz| >
  C$, $|\Zdual| < C$, the function
  $a_0$ is a smooth funtion of $\rho_{base} = \ang{Z}^{-1}, \hat{\totz}, \Zdual$;
  \item \textbf{Spatial interior/momentum infinity:} On sets $|\totz| < C$,
  $|\Zdual| > C$, \\ $a_0$ is a smooth function of $\totz, \rho_{fib} = 1/\ang{\Zdual}, \hat\Zdual$;
   \item \textbf{Spatial infinity/momentum infinity (the corner):} On
      sets $|\totz| > C$,  $|\Zdual| > C$, $a_0$ is a smooth function of $\rho_{base}, \rho_{fib}, \hat\totz, \hat\Zdual$. 
\end{itemize}
It is straightforward to show that if $a_0 \in
C^\infty(\overline{T}^*_{\mathrm{par}}\mathbb{R}^{n + 1})$ then $a_0 \circ
\psi^{-1} \in C^\infty(T^*\mathbb{R}^{n + 1})$ in fact lies in
$S^{0,0}_{\mathrm{par}}$.

\begin{defn}\label{def:classical}
  Still with $\Zdual = (\zeta, \tau)$, a symbol $a\in S^{m,l}_{\mathrm{par}}(\RR^{n + 1})$ is said to be classical if 
  $$
\langle \totz \rangle^{-l} \ang{\Zdual}^{-m} a \equiv \rho_{base}^l \, \rho_{fib}^m \, a  \, \in
C^\infty(\overline{T}^*_{\mathrm{par}}\mathbb{R}^{n + 1}),
$$
meaning it extends to a smooth function on this space. We denote the set of classical symbols by $S^{m,l}_{\mathrm{par, cl}}(\RR^{n + 1})$ and the set of pseudodifferential operators obtained by quantizing such operators by $\Psi_{\mathrm{par, cl}}^{m,l}(\RR^{n+1})$. 

For $A = q_L(a)$, $a \in S^{m,l}_{\mathrm{par, cl}}(\RR^{n + 1})$, we shall define\footnote{this was denoted $\sigma_{base, m, l}(A)$ in \cite{NEH}}
\begin{equation}\label{eq:sigmabase}
\sigma_{base, l}(A) := [\rho_{base}^l a] \, |_{\partial_{base} \overline{T}^*_{\mathrm{par}}\mathbb{R}^{n + 1}},
\end{equation}
which is a classical symbol of order $m$ on the fibres of $\partial_{base} \overline{T}^*_{\mathrm{par}}\mathbb{R}^{n + 1}$. 
\end{defn}

\begin{remark}
  The symbols in Definition~\ref{def:parab symbol spaces} can also be
  characterized by a regularity condition when they are thought of as
  functions on $\overline{T}^*_{\mathrm{par}}\mathbb{R}^{n + 1}$,
  namely they are conormal to the boundary with
  the appropriate weights.
  \end{remark}

\subsection{Symbols, ellipticity, operator wavefront sets and Hamilton vector fields}

We have the principal symbol mapping 
\begin{equation}
  \label{eq:principal symbol}
  \sigma_{m,l} \colon \Psi_{\mathrm{par}}^{m,l}(\RR^{n+1})
  \longrightarrow S^{m,l}_{\mathrm{par}} / S^{m-1,l-1}_{\mathrm{par}} = S^{[m,l]}_{\mathrm{par}},
\end{equation}
with kernel equal to $\Psi_{\mathrm{par}}^{m-1,l-1}$.
Restricting attention to classical operators, given $q_L(a) = A,
q_L(b) = B \in \Psi_{\mathrm{par, cl}}^{m,l}(\RR^{n+1})$, we have
$$
\sigma_{m,l}(A) = \sigma_{m,l}(B) \iff \left( \langle \totz \rangle^{-l}
\ang{\Zdual}^{-m} a \right) \rvert_{\partial \overline{T}^*_{\mathrm{par}}\mathbb{R}^{n +
    1}} \equiv \left( \langle \totz \rangle^{-l} \ang{\Zdual}^{-m} b \right) \rvert_{\partial \overline{T}^*_{\mathrm{par}}\mathbb{R}^{n +
    1}}.
$$
This allows us to view the principal symbol of a classical operator --- renormalized by suitable powers of the boundary defining functions of fibre and spacetime infinity --- as a function on the boundary of compactified phase space, $\partial \overline{ T^*_{\mathrm{par}}\mathbb{R}^{n +1}}$.

The appropriate notion of ellipticity in this calculus is uniform in
the spacetime weight function.  Thus, we say 
$A \in \Psi^{m,l}_{\mathrm{par}}$ is \textit{globally} elliptic if 
\begin{equation}
  \label{eq:ell/char def}
\sigma_{m,l}(A)(Z, \Zdual) \ge C \langle Z \rangle^{l} \ang{\Zdual}^m.
\end{equation}
More generally we consider microlocal ellipticity at a boundary point of the compactified parabolic cotangent bundle, $q \in \partial \overline{T}^*_{\mathrm{par}}\mathbb{R}^{n + 1}$. We say that $A \in \Psi^{m,l}_{\mathrm{par}}$ is (microlocally) elliptic at $q$, and write $q \in \Ell_{m,l}(A)$, if this estimate holds in a
neighborhood of $q$ in $\overline{T}^*_{\mathrm{par}}\mathbb{R}^{n + 1}$.  To clarify this we can write down the estimates
at the three regions of phase space we considered above.  
\begin{enumerate}
\item For $q$ in the spacetime interior (and therefore at fibre infinity), points in phase space are in
  $\Ell_{m,l}(A)$ if and only if $A$ is elliptic there in the
  standard parabolic sense, i.e.\ for $q = (Z_0, \hat \Zdual_0)$, $q \in
  \Ell_{m,l}(A)$ if and only if for some $C, \epsilon > 0$, $a(Z, \Zdual) \ge C  |\Zdual|^m$ for
  $ |\Zdual| > \epsilon^{-1}, |Z - Z_0|, |\hat \Zdual - \hat \Zdual_0|<\epsilon.$
 \item At spacetime infinity with $\Zdual$ finite, $q = (\hat Z_0,
   \Zdual_0)$ lies in $\Ell_{m,l}(A)$ if and only if for some $C,
   \epsilon > 0$, $a(Z, \Zdual) \ge C  | Z |^l$ for
  $|Z| > \epsilon^{-1}, |\hat Z - \hat Z_0|, |\Zdual -
  \Zdual_0|<\epsilon.$
  \item At the corner, $q = (\hat Z_0,
   \hat \Zdual_0)$ lies in $\Ell_{m,l}(A)$ if and only if for some $C,
   \epsilon > 0$, $a(Z, \Zdual) \ge C  | Z |^l |\hat \Zdual|^m$ for
  $|Z|, |\Zdual | > \epsilon^{-1}, |\hat Z - \hat Z_0|, |\hat \Zdual -
  \hat \Zdual_0|<\epsilon.$ 
\end{enumerate}
The elliptic set of $A$ is by definition an open set. 
The characteristic set is simply the complement of the elliptic set (hence closed):
\begin{equation}\label{eq:chardef}
\Char_{m,l}(A) = \partial \overline{T}^*_{\mathrm{par}}\mathbb{R}^{n
    + 1} \setminus \Ell_{m,l}(A). 
\end{equation}

For classical operators, it is sometimes convenient to think of the
properties of these symbols in terms of the boundary
restriction of the reweighted function $a_0 = \langle \totz \rangle^{-l}
\ang{\Zdual}^{-m} a$.  Specifically, it follows easily
that, for $A \in \Psi^{m,l}_{\mathrm{par, cl}}$ that
\begin{equation}
  \label{eq:ell/char def 2}
  \Ell_{m,l}(A) = \{ q \in \partial \overline{T}^*_{\mathrm{par}}\mathbb{R}^{n
    + 1} : a_0(q) \neq 0 \},
\end{equation}
and thus $\Char_{m,l}(A) = \{ q \in \partial \overline{T}^*_{\mathrm{par}}\mathbb{R}^{n
    + 1} : a_0(q) = 0 \},$ i.e.\ it is simply the vanishing locus of
  this smooth function.

In the standard scattering calculus, the fiber principal symbol is, in
a sense, the usual principal symbol, and, if the symbol is classical,
it can be represented by a homogeneous function.  In our parabolic
setting a related statement is true, namely in bounded spatial sets,
away from zero momentum, the principal symbol of fibre order $m$ of a classical symbol in this calculus is represented by a unique function homogeneous of degree $m$, in the parabolic sense, i.e.\ 
\begin{equation}
\lambda > 0, \ \tilde a(t, z, \lambda^2 \tau, \lambda \zeta) = \lambda^m
\tilde a(t, z, \tau, \zeta)\label{eq:quasihomo}.
\end{equation}
Indeed, this function is the unique function $a$ homogeneous of degree $m$ such that $\ang{\Phi}^{-m} a$ has the appropriate boundary value at fibre infinity. 
However, if the symbol is not classical, the symbol at spatial
infinity is usually \emph{not} represented by a homogeneous function in the fiber variables.

The concept of operator wavefront set, also known as microlocal support, also carries over directly. Namely, 
if $A = q_L(a)$, then the operator wavefront set $\WF'(A)$ is the essential support of $a$, 
i.e.\ the subset
of $\partial \overline{T}^*_{\mathrm{par}}\mathbb{R}^{n + 1}$ whose complement consist of points
$q$ such that ``$a$ is trivial in an open neighborhood of $q$''. The meaning of this statement is that there is a neighbourhood $U \subset
\overline{T}^*_{\mathrm{par}}\mathbb{R}^{n + 1}$ of $q$  such that $a \rvert_{U \cap T^*
\mathbb{R}^{n + 1}}$ is Schwartz, i.e.\ vanishes to all orders
together with all its derivatives.  In particular, for $A \in \Psi^{m,l}_{\mathrm{par}}$,
$$
\WF'(A) = \varnothing \implies A \in \Psi^{-\infty, -\infty}_{\mathrm{par}},
$$
meaning $A = q_L(a)$ for $a \in \mathcal{S}(\mathbb{R}^{n + 1} \times
\mathbb{R}^{n + 1})$ and therefore $\WF'(A) = \varnothing$ implies
that $A$ maps tempered distributions to Schwartz functions, $A \colon \mathcal{S}'(\mathbb{R}^{n + 1})
\longrightarrow \mathcal{S}(\mathbb{R}^{n + 1})$.

As in the standard scattering setting, given $a \in
S^{m,l}_{\mathrm{par, cl}}(\mathbb{R}^{n +1})$, an appropriate
rescaling of the standard Hamilton vector field $H_a$
extends smoothly to the whole of $\overline{T}^*_{\mathrm{par}}\mathbb{R}^{n +
  1}$.  This can be seen from the following lemma which characterizes
the asymptotic behaviour of the linear vector fields on $\mathbb{R}^{n
  + 1}$.
\begin{lemma}\label{lem:ham vec rescale}
  Let $\rho_{base} , \rho_{fib}$ be boundary defining functions for
  $C^\infty(\overline{T}^*_{\mathrm{par}}\mathbb{R}^{n + 1})$ (see
  belo~\eqref{eq:compactification}) and let $\partial_{\hat{\totz}_i}$ and
  $\partial_{\hat{\zeta}_i}$ denote vector fields tangent to the spheres
  $\mathbb{S}^{n}_{\hat z}$ and   $\mathbb{S}^{n}_{\hat \zeta}$,
  respectively.  Then
  $$
\partial_{\totz_j} , \partial_{\zeta_j}\in \mathrm{span}_{C^\infty(\overline{T}^*_{\mathrm{par}}\mathbb{R}^{n + 1})} \langle \rho_{base}^2
\partial_{\rho_{base}}, \,  \rho_{base} \partial_{\hat \totz_i} , \,  \rho_{fib}^2
\partial_{\rho_{fib}}, \, \rho_{fib} \partial_{\hat \zeta_i}  \rangle,
$$
while
  $$
\partial_{\tau} \in \mathrm{span}_{C^\infty(\overline{T}^*_{\mathrm{par}}\mathbb{R}^{n + 1})} \langle \rho_{fib}^3
\partial_{\rho_{fib}}, \, \rho_{fib}^2 \partial_{\hat \zeta_i}  \rangle
$$
\end{lemma}
\begin{proof}
Since the spatial compactification here is the standard radial
compactification, the statement for the spatial vector fields $\partial_t$ and
  $\partial_{z_i}$ follows from the standard scattering case, and
  indeed there are no $\rho_{fib}^2
\partial_{\rho_{fib}}$ or $\rho_{fib} \partial_{\hat \zeta_i} $
terms.  Here one simply writes the vector fields in polar coordinates $|\totz|,
\hat{\totz}$ with $\rho_{base} = 1/\ang{\totz}$.

For the fiber variables, differentiating \eqref{eq:fiber
  weight function} shows
$$
\partial_{\tau} \rho_{fib} \in  \rho_{fib}^3 C^\infty(\overline{T}^*_{\mathrm{par}}\mathbb{R}^{n + 1}), \quad \partial_{\tau}
(\zeta_i\rho_{fib}) , \ \partial_{\tau}
(\tau \rho_{fib}^2) \in \rho_{fib}^2 C^\infty(\overline{T}^*_{\mathrm{par}}\mathbb{R}^{n + 1})
$$
which implies the statement for $\partial_{\tau}$, and
$$
\partial_{\zeta_j} \rho_{fib} \in  \rho_{fib}^2 C^\infty(\overline{T}^*_{\mathrm{par}}\mathbb{R}^{n + 1}), \quad \partial_{\zeta_j}
(\zeta_i\rho_{fib}) , \ \partial_{\tau}
(\tau \rho_{fib}^2) \in \rho_{fib} C^\infty(\overline{T}^*_{\mathrm{par}}\mathbb{R}^{n + 1}),
$$
which implies the statement for $\partial_{\zeta_j}$.
\end{proof}

 Given a classical symbol $a \in S^{m,l}_{\mathrm{par, cl}}(\mathbb{R}^{n +
    1})$, recalling that
  $$
H_a = \partial_\tau a \partial_t - \partial_t a \partial_\tau +
\sum_{j=1}^n (\partial_{\zeta_j}a\partial_{z_j} - \partial_{z_j}a\partial_{\zeta_j}),
$$
define the (parabolically) rescaled Hamilton vector field
  vector field
  \begin{equation}
    \label{eq:ham vec rescale}
H^{m,l}_a := \rho_{fib}^{m-1} \rho_{base}^{l - 1} H_a
  \end{equation}
 From the lemma, we see that $H^{m,l}_a$ \textit{extends smoothly to
   $\overline{T}^*_{\mathrm{par}}\mathbb{R}^{n + 1}$ and is
tangent to the boundary.}  Moreover, the fact that the components of
$\partial_\tau$ vanish an order faster at fiber infinity than the
components of the other vector fields implies that at fiber infinity, the terms with
$\partial_\tau$ do not contribute to leading order there, i.e.\
$$
H^{m,l}_a  \rvert_{\{ \rho_{fib} = 0 \}} = \rho_{fib}^{m-1} \rho_{base}^{l -
  1} \left( \sum_{j=1}^n \partial_{\zeta_j}a\partial_{z_j}- \partial_{z_j}a\partial_{\zeta_j}\right) 
$$

One can see that the concept of Hamilton vector field is invariant, in a
leading order sense, when applied to
an element $A \in \Psi^{m,l}_{\mathrm{par}}$ as follows. With $a \in
S^{m,l}_{\mathrm{par}}$, $H^{m,l}_a$ lies in the space
$\mathcal{V}_b(\overline{T}^*_{\mathrm{par}}\mathbb{R}^{n + 1})$ of vector
fields tangent to the boundary, the so-called b-vector fields.  Near
the corner, any $V \in \mathcal{V}_b(\overline{T}^*_{\mathrm{par}}\mathbb{R}^{n
  + 1})$ satsifies $V\rho_{fib} = O(\rho_{fib})$ and $V\rho_{base} =
O(\rho_{base})$, meaning
$$
V = c_f \rho_{fib} \partial_{\rho_{fib}} + c_b \rho_{base}
\partial_{\rho_{base}} + \sum c_i \partial_{\hat{\totz}_i} + c_i' \partial_{\hat{\zeta}_i'}
$$
where the $c_f, c_b, c_i, c'_i$ are all smooth functions on $\overline{T}^*_{\mathrm{par}}\mathbb{R}^{n + 1}$
It is straightforward to check that if $a, a' \in S^{m,l}_{\mathrm{par}}$ have
$a - a' \in S^{m-1,l-1}_{\mathrm{par}}$ then 
$$
H^{m,l}_a - H^{m,l}_{a'} = H^{m,l}_{a-a'} \in \rho_{fib} \rho_{base} \mathcal{V}_b(\overline{T}^*_{\mathrm{par}}\mathbb{R}^{n
  + 1}).
  $$
That is to say, if $a, a'$ both represent
$\sigma_{m,l}(A)$, then their Hamilton vector fields agree up to first
order at infinity \emph{as b-vector fields} (which is slightly stronger than saying that these vector fields agree at the boundary).   In particular, $(H^{m,l}_a -
H^{m,l}_{a'})\rho_{fib} = O(\rho_{fib}^2)$ and $(H^{m,l}_a -
H^{m,l}_{a'})\rho_{base} = O(\rho_{base}^2)$.

We also have, as in the standard scattering setting, that the flow of
$H^{m,l}_a$ preserves the characteristic set of $a$.
\begin{prop}
  Let $a \in S^{m,l}_{\mathrm{par, cl}}(\mathbb{R}^{n +1})$ be classical.  Then
  $H^{m,l}_a$ is tangent to the characteristic set $\Char_{m,l}(q_L(a))$.  
\end{prop}
\begin{proof}
Let  $a_0 = \langle \totz \rangle^{-l} \ang{\Zdual}^{-m} a \in
  S^{0,0}_{\mathrm{par}}$.     Since $H^{m,l}_a$ is tangent to the boundary of $\overline{T}^*_{\mathrm{par}}\mathbb{R}^{n + 1}$ and
  $\Char_{0,0}(q_L(a_0)) = \Char_{m,l} (q_L(a))$, it suffices to show that
  $H^{m,l}_a a_0 = 0$.  Using the product rule gives
  $$
H^{m,l}_a a_0 = a_0 H^{m,l}_{(\rho_{base}^{-l} \rho_{fib}^{-m})} a_0 +
\rho_{base}^{-l} \rho_{fib}^{-m} H^{0,0}_{a_0} a_0 = a_0 H^{m,l}_{(\rho_{base}^{-l} \rho_{fib}^{-m})} a_0
$$
which vanishes on all of $\Char_{0,0}(q_L(a_0))$.
\end{proof}

\begin{defn}\label{def:radial set}
For $A \in \Psipcl{m}{l}(\mathbb{R}^{n
  + 1})$, $A = q_L(a)$, the \textbf{radial set} is defined by 
\begin{equation}\label{eq:radial set}
\mathcal{R} = \{ q \in \Char_{m,l}(A)  : H^{m,l}_a \mbox{ vanishes at
} q \}.
\end{equation}
Equivalently, $\mathcal{R}$ is the set of stationary points of
the flow of $H^{m,l}_a$ on $\Char_{m,l}(A)$.
\end{defn}

It is very useful to note that in the interiors of fiber or spatial
infinity, as with the symbol, it is possible only to rescale the
Hamilton vector field in only one variable without losing information.
To avoid introducing further notation, we point out
informally that, for $a \in S^{m,l}_{\mathrm{par, cl}}$,
\begin{itemize}
\item In regions of the form $\{ |\totz| < C\}$, i.e.\ in bounded
  spatial sets, $ \rho_{fib}^{m-1} H_a = \rho_{base}^{-l + 1}H^{m,l}_a$
  extends smoothly to $\{\rho_{fib} = 0\}$.
  \item In regions of the form $\{ |\Zdual| < C\}$, i.e.\ in bounded
 momentum sets, $\rho_{base}^{l-1} H_a = \rho_{base}^{- m + 1}H^{m,l}_a$
  extends smoothly to $\{\rho_{base} = 0\}$.
\end{itemize}
In these regions, these vector fields are smooth positive multiples of
$H^{m,l}_a$ and thus their flows preserves $\Char_{m,l}(a)$ and are
simply non-degenerate reparametrizations of the $H^{m,l}_a$ flow.

We will also require parabolic pseudodifferential operators of variable spatial order.

\begin{defn}
	Given a \emph{weight function}  $\sw \in S^{0,0}_\mathrm{par, cl}$, a \emph{classical} symbol of order $(0,0)$, and $\delta\in (0,1/2)$, we define the following weighted symbol class
	\begin{equation}\label{symbolestdelta}
		S_{\delta,\mathrm{par}}^{m,\sw}=\{a\in\SC^\infty(T^*{\RR^{n+1}}):|\partial_z^\alpha\partial_t
                ^k \partial_\zeta ^\beta\partial_\tau^j a|\lesssim
                \langle \totz \rangle^{\sw-(1-\delta)(|\alpha|+k)+\delta
                  (|\beta|+2j)}\ang{\Zdual}^{m-|\beta|-2j}\}
	\end{equation}
\end{defn}

\begin{remark}
  One can see the necessity of the $\delta > 0$ loss to treat
  non-constant $\sw \in S^{0,0}_\mathrm{par, cl}$ by considering $\langle z
  \rangle^{\sw}$.  Indeed,
  $$
\partial_{z_i} \langle z \rangle^{\sw} =
\sw \langle z \rangle^{\sw - 1}\partial_{z_i} \langle z \rangle +\langle z \rangle^{\sw} \log \langle z \rangle \partial_{z_i}\sw
$$
and similar estimates for further derivatives show that this function
lies in $S^{0, \sw}_{\delta,  \mathrm{par}}$ for any $\delta > 0$ (but not $\delta = 0$).
\end{remark}

Quantisation of symbols in $S_{\delta,\mathrm{par}}^{m,\sw}$ for arbitrary $\delta\in(0,1/2)$ works in exactly the same way as for constant order symbols, giving rise to a class $\Psi_{\delta,\mathrm{par}}^{m, \sw}$ of variable order operators in our parabolic pseudodifferential calculus. We note that we have a similar containment as in \eqref{eq:basic containment}, namely 
\begin{equation}
S^{m,\sw}_{\mathrm{par}}(\mathbb{R}^{n + 1}) \subset \begin{cases} S^{m,\sw}_{1/2,
  \delta}(\mathbb{R}^{n + 1}), \quad m \geq 0 \\
  S^{m/2,\sw}_{1/2,
  \delta}(\mathbb{R}^{n + 1}), \quad m \leq 0 \end{cases}.
  \label{eq:basic containment delta}
\end{equation}

\subsection{Composition, $L^2$-boundedness, Sobolev spaces, and elliptic regularity}

To draw more detailed conclusions about the parabolic scattering
operators, we need to verify that the calculus behaves as expected.

\begin{prop}	\label{prop:parcalc}
Let $q_L(a) = A \in
\Psipd{m}{\sw}(\RR^{n+1})$ and $q_L(b) = B \in
  \Psipd{m'}{\sw'}(\RR^{n+1})$.  Then
\begin{itemize}
\item $A B \in
  \Psipd{m+m'}{\sw + \sw'}(\RR^{n+1})$;
\item $\sigma_{m + m', \sw + \sw'}(AB) = \sigma_{m, \sw}(A) \sigma_{m',\sw'}(B)$;
\item The commutator $[A,B]$ is in the space $\Psipd^{m+m'-1,\sw + \sw'-1 + \delta}$, and
$$\sigma_{m+m'-1,\sw + \sw'-1 + \delta}([A,B])=\frac{1}{i}\{a,b\}$$ 
where $\{a,b\}$ denotes the Poisson bracket:
	\begin{equation}\label{eq:poisson rel}
		\{a,b\}=H_a b '
	\end{equation}
\end{itemize}
\end{prop}
\begin{proof}
This can be obtained cheaply from \eqref{eq:basic containment delta} and the standard expansion of the symbol of a product. 

\end{proof}

The elliptic set $\Ell_{m,\sw}(A)$ for variable order operators $A$ is defined just as for the constant spatial order case. 
The microlocal elliptic parametrix
construction goes through in this context and we conclude:
\begin{prop}\label{prop:elliptic param}
  Let $A \in \Psipd{m}{\sw}(\mathbb{R}^{n + 1})$ and let $K$ be a
  compact subset of $\Ell_{m,\sw}(A)$.  Then there is $B \in
  \Psipd{-m}{-\sw}(\mathbb{R}^{n + 1})$ such that
  $$
K \cap \WF'(AB - Id) =  K \cap \WF'(BA - Id) = \varnothing .
  $$
\end{prop}

The global version of this proposition and the H\"ormander ``square root trick'' then
imply that
\begin{equation}
  \label{eq:L2 bounded}
  A \in \Psi^{0,0}_{\mathrm{par}}(\mathbb{R}^{n + 1}) \implies A \colon L^2
  \longrightarrow L^2 \mbox{ is bounded.}
\end{equation}
Note this also follows from the containment \eqref{eq:basic
  containment}.

We define the parabolic weighted Sobolev spaces, initially with constant orders, 
analogously to the standard scattering spaces by
\begin{defn}
	\label{def:sobolev}
	\[H^{m,l}_{\mathrm{par}}:=\{u\in \mathcal{S}'(\RR^{n+1}):Au\in L^2(\RR^{n+1})\textrm{ for all }A\in\Psip{m}{l}\}.\]
      \end{defn}
      Note that the operator
      $$
\Lambda_{m,l} = \mathrm{Op}(\langle \totz \rangle^{-m} \ang{\Zdual}^{-l}) =
\langle z,t \rangle^{-m}  \mathcal{F}^{-1} \ang{\Zdual} ^{-l} \mathcal{F}.
      $$
lies in $\Psi^{-m,-l}$ and is manifestly invertible.  For any $M,N \in
\mathbb{R}$ it therefore defines an isomorphism
$$
\Lambda_{m,l} \colon H^{M,L}_{\mathrm{par}} \longrightarrow H^{M + m, L + l}_{\mathrm{par}},
$$
and we define a topology on $H^{M,L}_{\mathrm{par}}$ by
\begin{equation}
  \label{eq:Hml top}
  \| u \|_{H^{m,l}_{\mathrm{par}}} = \| \Lambda_{-m,-l} u \|_{L^2}.
\end{equation}
For any $A \in \Psip{m}{l}(\mathbb{R}^{n + 1})$ and $u \in
H^{M,L}_{\mathrm{par}}$, \eqref{eq:L2 bounded} and Proposition
\ref{prop:parcalc} yield the estimate
$$
\| A u \|_{H^{M - m, L - l}_{\mathrm{par}}} = \| (\Lambda_{-M + m, -L + l} A   \Lambda_{M, L})
  \Lambda_{-M, -L} u \|_{L^2} \le C   \| u \|_{H^{M,L}_{\mathrm{par}}},
  $$
  and thus:
      \begin{prop}\label{prop:basic bounde}
        Let $A \in \Psip{m}{l}(\mathbb{R}^{n + 1})$.  Then for any $M,
        L \in \mathbb{R}$, $A \colon H^{M,L} \longrightarrow H^{M - m, L - l}$ is bounded.
      \end{prop}

Putting together this proposition with the microlocal elliptic
parametrix in Proposition \ref{prop:elliptic param}.
In $\Ell(P)$, the standard microlocal elliptic parametrix construction, caried out in the parabolic pseudodifferential calculus yields the following result.
\begin{prop}
	\label{prop:elliptic.estimate}
	Suppose $P\in\Psip{m}{l}$ and $Q,G\in\Psip{0}{0}$ such that $P$ and $G$ are elliptic on $\WF'(Q)$. Then if $GPu\in H_{\mathrm{par}}^{s-m,r-l}$, we have $Qu\in H^{s,r}$ with the estimate
	\begin{equation}
		\|Qu\|_{H^{s,r}}\leq C(\|GPu\|_{H^{s-m,r-l}}+\|u\|_{H^{M,N}})
	\end{equation}
	for any $M,N\in\RR$.
\end{prop}

Variable order spaces can be defined similarly. 
\begin{defn}
	\label{def:sobolevdelta}
	Let $\sw$ be a classical symbol in $
        S^{0,0}_\mathrm{par}$ satisfying $\sw\geq l$ and let $A$ be
        a fixed, classical, globally elliptic element of
        $\Psi^{m,\sw}_{\delta, \mathrm{par}}$
        with $\delta\in (0,1/2)$.  We define
	\[H_{\mathrm{par}}^{m,\sw}:=\{u\in H_\mathrm{par}^{m,l}:Au\in L^2(\RR^{n+1})\}\]
\end{defn}
We equip $H_{\mathrm{par}}^{m,\sw}$ with the norm
\begin{equation}
\|u\|_{H_{\mathrm{par}}^{m,\sw}}:= \|u\|_{H_{\mathrm{par}}^{m,l}}+\|Au\|_{L^2}.
\end{equation}
This imparts a Hilbert space structure on $H_{\mathrm{par}}^{m,\sw}$, and moreover this structure is independent of the choice of $l$ and $A$. To see the independence we note that by the standard elliptic parametrix construction we can choose $B$ such that $I=BA+R$ with $B\in\Psip{-m}{-\mathrm{I}}$ and $R\in \Psip{-\infty}{-\infty}$, for any other $\tilde{A}\in \Psi^{m,\sw}$ we can write $\tilde{A} u=\tilde{A}BA u+\tilde{A}Ru$ and use that $\tilde{A}B\in \Psi_{\delta,\mathrm{par}}^{0,0}$ to bound $$\|\tilde{A}u\|\lesssim \|u\|_{H_{\mathrm{par}}^{m,l}}+\|Au\|_{L^2}$$

\begin{rem}
Proposition \ref{prop:basic bounde} and	Proposition \ref{prop:elliptic.estimate} hold for variable order operators $A,P\in \Psip{m}{\sw}$ with only minor modifications to the proof. See for example \cite[Proposition~5.15]{grenoble}.
\end{rem}


\section{Geometry of the time-dependent Schr\"odinger equation}
\label{sec:schrod geom}

\subsection{Characteristic variety}\label{subsec:charvar}
Let $P$ denote the operator $D_t + \Delta_g + V$, where $D_t = -i \partial_t$, $\Delta_g$ is the (positive)  Laplacian on $\RR^n$ with respect to a metric $g = g(t)$ and the metric $g$ and potential $V$ are as in Section~\ref{subsec:intro}.  The operator $P$ 
lies in $\Diff^{2,0}_{\mathrm{par}}$  and has 
principal symbol (written using the Einstein summation convention)
$$
\sigma_{2,0}(P) = p(z, t, \zeta, \tau)  = \tau + g^{ij}(z,t) \zeta_i \zeta_j.
$$
In this section we will study this
operator using the structures developed in the previous section; in
particular we will identify its (parabolic) characteristic set, its radial set,
and explain the Hamiltonian dynamics thereon.

Recall that the characteristic set, defined in \eqref{eq:chardef}, is a subset of the boundary of the compactified parabolic cotangent bundle, and therefore has a component at spacetime infinity and a component at fibre infinity, which intersect at the corner (both spatial and fibre infinity). 

By assumption, in a neighbourhood of spacetime infinity the operator $P$ coincides with $P_0 = D_t + \Delta_0$, where $\Delta_0$ is the (positive) flat Laplacian. The symbol of this operator is $\tau + |\zeta|^2$, and thus at spacetime infinity and in regions of bounded $\Phi = (\zeta, \tau)$, the characteristic variety $\Sigma_0$ is given simply
by $\tau = - |\zeta|^2$. Near fibre infinity (and still near spacetime infinity),  
this can be written in the form 
$$
\{ \tau / R^2 = - |\zeta/R|^2 \} \cap \partial \overline{T}^*_{\mathrm{par}}\mathbb{R}^{n + 1},
$$
where $R$ is as in \eqref{eq:fiber weight function}, 
which shows that this set is a smooth submanifold of $\overline{T}^*_{\mathrm{par}}\mathbb{R}^{n + 1}$ intersecting both fibre infinity and
spatial infinity nontrivially and transversally.  (Note that  $\tau/R^2$ and $\zeta/R$ are `angular' variables that are smooth up to fibre infinity.)  
In a bounded spacetime region, near fibre infinity, the characteristic variety can be similarly written 
$$
\{ \tau / R^2 = - g^{ij}(z, t)(\zeta/R)_i (\zeta/R)_j \} \cap \partial \overline{T}^*_{\mathrm{par}}\mathbb{R}^{n + 1},
$$
which shows that the characteristic variety is a smooth codimension one submanifold of $\partial \overline{T}^*_{\mathrm{par}}\mathbb{R}^{n + 1}$ in this region. 

The Hamilton vector field of $P$ is 
\begin{equation}
   H_p = \frac{\partial}{\partial t} +  2 g^{ij}(z,t) \zeta_i \frac{\partial}{\partial z_j} - \frac{\partial g^{ij}}{\partial t}\zeta_i \zeta_j \frac{\partial}{\partial \tau} - \frac{\partial g^{ij}}{\partial z_k}\zeta_i \zeta_j \frac{\partial}{\partial \zeta_k}.
\end{equation}
In bounded regions of spacetime, the fiber
rescaled Hamilton vector field is simply $\frac{1}{R} H_p$. Restricting this to fibre infinity, that is, taking the limit as $R \to \infty$, we see that the coefficients of $\partial_t$ and $\partial_\tau$ vanish. Thus the flow at fibre-infinity takes place at one moment of time, say $t=t_0$ (reflecting `infinite propagation speed' for the time-dependent Schr\"odinger equation). If we let $\zeta' = \zeta/R$ then we obtain the rescaled flow equations 
$$
\dot z^i = 2 g^{ij}(z, t_0)  \zeta'_j, \quad \dot \zeta'_k = - \frac{\partial g^{ij}(z, t_0)}{\partial z_k}\zeta'_i \zeta'_j 
$$
which we recognize as the geodesic equations for the metric $g(t_0)$ at the fixed time $t_0$. Moreover, on the characteristic variety we have $|\zeta'| = 2^{-1/4}$, so, because of the uniform positive-definiteness of $g^{ij}$, we see that this rescaled Hamilton vector field is nonvanishing over the spacetime interior.

\subsection{Radial sets and the Hamilton Flow}
\label{subsec:radial}
We now determine the radial set (see Definition~\ref{def:radial set}) for $P$, and the nature of the Hamilton flow in a neighbourhood of the radial set. 

As has just been shown, the rescaled Hamilton vector field in the spacetime interior is nonvanishing. Radial points, if they exist, therefore lie over spacetime infinity, that is, in $\partial_{base} \overline{T}^*_{\mathrm{par}}\mathbb{R}^{n + 1}$. We will show
\begin{proposition}\label{prop:radial set geom}
	The radial set $\mathcal{R}$ of $P$ is a disjoint union 
	$\mathcal{R} = \mathcal{R}_+ \cup \mathcal{R}_-$ 
	of smooth submanifolds of $\partial_{base} \overline{T}^*_{\mathrm{par}}\mathbb{R}^{n + 1}$ of dimension $n$. The component $\mathcal{R}_+$ is a family of global sinks for the rescaled Hamilton vector field $H_p^{2,0}$, and $\mathcal{R}_-$ is a family of global sources for this rescaled Hamilton vector field. Every bicharacteristic $\gamma(s)$ of $P$ (meaning a flowline of $H_p^{2,0}$ within the characteristic variety $\Sigma(P)$) converges to $\mathcal{R}_+$ as $s \to \infty$ and  to $\mathcal{R}_-$ as $s \to -\infty$. 
\end{proposition}

\begin{proof}
To prove this, we start by noting that in a neighbourhood of spacetime infinity, $P$ coincides with $P_0 = D_t + \Delta_0$, where $\Delta_0$ is the flat (positive) Laplacian on $\RR^n$. Thus we only need to consider the Hamilton flow for this flat model, which is 
\begin{equation}\label{eq:simple ham}
H_{p_0} = \frac{\partial}{\partial t} + 2 \zeta \cdot \frac{\partial}{\partial z}.
\end{equation}

We first consider the region spacetime region $\{ |t| \geq \epsilon |z|, t \geq C \}$ for arbitrary $C, \epsilon > 0$. In terms of Figure~\ref{fig:rad comp spacetime}, this is strictly in the `northern hemisphere'. 
In this region, $w := z/t$ is a coordinate on the spacetime boundary, we take $\rho_b = 1/t$ as a boundary defining function, and we rescale the Hamilton vector field by dividing by $\rho_b$, or equivalently multiplying by $t$; that is, we consider  
$$
t H_{p_0} = t \frac{\partial}{\partial t} + 2t \zeta \cdot \frac{\partial}{\partial z}.
$$
Using coordinates $(\rho_b, w, \zeta, \tau)$ which are valid near spacetime infinity and for bounded $\Phi = (\zeta, \tau)$, this is 
$$
- \rho_b \frac{\partial}{\partial \rho_b} + (2 \zeta - w) \cdot \frac{\partial}{\partial w}.
$$
Then changing coordinates to $(\rho_b, \tilde w, \zeta, \tau)$ where $\tilde w = w - 2\zeta$, we obtain 
$$
- \rho_b \frac{\partial}{\partial \rho_b}  - \tilde w \cdot \frac{\partial}{\partial \tilde w}.
$$
Thus the radial set in this region, which we denote $\mathcal{R}_+$, is given by $\rho_b = 0, \tilde w = 0$ and $\tau + |\zeta|^2 = 0$ (which is just the condition of lying in $\Sigma(P_0)$). It is clear that the rescaled Hamilton vector field is a sink near $\mathcal{R}_+$. Thus, in this region, we have 
\begin{equation}\label{eq:rad.interior}
\mathcal{R}_+ = \{ \rho_b = 0, \ \zeta  = \frac{w}{2}, \ \tau  = \frac{|w|^2}{4} \}, \quad w = \frac{z}{t}.
\end{equation}
Noting that $w$ is a coordinate on spacetime infinity in this region, we see that 
the radial set $\mathcal{R}_+$ is a \emph{graph over spacetime infinity} in this region. It also reflects the fully \emph{dispersive} nature of the Schr\"odinger equation: each frequency propagates in a different direction or at a different speed, and therefore ends up at a different point of spacetime infinity. 

There is an analogous radial set over the `southern hemisphere' (in terms of Figure~\ref{fig:rad comp spacetime}), where we restrict to  $\{ |t| \geq \epsilon |z|, t \leq -C \}$. As we will shortly show, this is a different component of the radial set, which we denote $\mathcal{R}_-$. If we now redefine our coordinates so that 
$\tilde \rho_b = -1/t$ (so that it is a nonnegative function), $w = z/|t|$ and $\tilde w = w + 2\zeta$, we have 
\begin{equation}\label{eq:rad.interior.2}
\mathcal{R}_- = \{ \rho_b = 0, \ w = \frac{z}{|t|} = -2\zeta, \ \tau + |\zeta|^2 = 0 \}
\end{equation}
which is a graph over (part of) the southern hemisphere. 
The rescaled Hamilton vector field $|t| H_{p_0}$ takes the form 
$$
 \rho_b \frac{\partial}{\partial \rho_b}  + \tilde w \cdot \frac{\partial}{\partial \tilde w},
$$
i.e.\ it is a source near $\mathcal{R}_-$. 

We now consider the case where $|z| \geq C |t|$ and $|z| \geq C$, that is, near the `equator' in terms of Figure~\ref{fig:rad comp spacetime}. Working in a small neighbourood of the equator, we may suppose without loss of generality that the first spatial coordinate $z_1$ is positive, and satisfies $z_1 \geq 1/2 \max_i |z_i|$. In that case, we may take the spacetime boundary defining function $\rho_b$ to be $1/z_1$. We also write $s = t/z_1$ and $v_j = z_j/z_1$ for $j \geq 2$. First working in a region where $\Phi$ is bounded, we rescale the Hamilton vector field by dividing by $\rho_b$, that is, multiplying by $z_1$. Using coordinates $(\rho_b, s, v_j, \zeta, \tau)$, we have 
$$
\rho_b^{-1} H_{p_0} = \big( 1 - 2s \zeta_1 \big) \frac{\partial}{\partial s} + \sum_{j \geq 2} \big( 2 \zeta_j - 2 \zeta_1 v_j \big) \frac{\partial}{\partial v_j} - 2 \zeta_1 \rho_b \frac{\partial}{\partial \rho_b}.
$$
We are interested in the region where $s$ is small, otherwise the previous calculation applies. In order for this vector field to vanish, we see from the $\partial_s$ coefficient that when $s$ is small, necessarily $|\zeta_1|$ is large, and either positive or negative depending on the sign of $s$. First taking the case that $\zeta_1$ is large and positive, we  use fibre boundary defining function $\rho_f = 1/\zeta_1$ and coordinates $\omega_j = \zeta_j/\zeta_1$ and $\sigma = \tau/|\zeta|^2$. We further rescale the Hamilton vector field by multiplying by $\rho_f$. An easy computation shows  
$$
\rho_b^{-1} \rho_f H_{p_0} = \big( \rho_f - 2s \big) \frac{\partial}{\partial s} + \sum_{j \geq 2} \big( 2 \omega_j - 2  v_j \big) \frac{\partial}{\partial v_j} - 2  \rho_b \frac{\partial}{\partial \rho_b}.
$$
Changing variables to $\tilde s = 2s - \rho_f$, $\tilde v_j = v_j - \omega_j$ then in coordinates $(\tilde s, \tilde v_j, \rho_b, \rho_f, \omega_j, \sigma)$ we have the vector field
\begin{equation}\label{eq:Hvf Rplus equator}
-2\tilde s \frac{\partial}{\partial \tilde s} -2  \sum_{j \geq 2} \tilde v_j \frac{\partial}{\partial \tilde v_j} - 2  \rho_b \frac{\partial}{\partial \rho_b}.
\end{equation}
In this region, where $\zeta_1 >> 0$, the radial set is given by $\{ \rho_b = 0, \tilde s = 0, \tilde v_j = 0, \sigma = -1 \}$.  These equations define a submanifold of dimension $n$ inside $\partial_{base} \overline{T}^*_{\mathrm{par}}\mathbb{R}^{n + 1}$, which is transverse to fibre infinity. It is not hard to check that where $s$ is strictly positive, this set coincides with the set $\mathcal{R}_+$ defined in \eqref{eq:rad.interior}. We write the equations for $\mathcal{R}_+$ using more natural coordinates in this region as 
\begin{equation}\label{eq:rad.corner}
\mathcal{R}_+ = \{ \rho_b = 0, \ s = \rho_f/2, \ \hat \zeta = \hat z, \ \tau/|\zeta|^2 = -1 \}.
\end{equation}
The rescaled Hamilton vector field in \eqref{eq:Hvf Rplus equator} is clearly a sink at $\mathcal{R}_+$ in this region. Notice that the projection to spacetime infinity gives the closed northern hemisphere, since $s = \rho_f/2 \geq 0$ in \eqref{eq:rad.corner}. Wherever $s$ is strictly positive, $\mathcal{R}_+$ is a graph, but it fails to be so at the boundary $s=0$: the graph `turns vertical' at the equator, and has a boundary at fibre-infinity. See Figure~\ref{fig:Rpm}. 

\begin{figure}\label{fig:Rpm}
	\centering
	\def\svgwidth{70mm}
	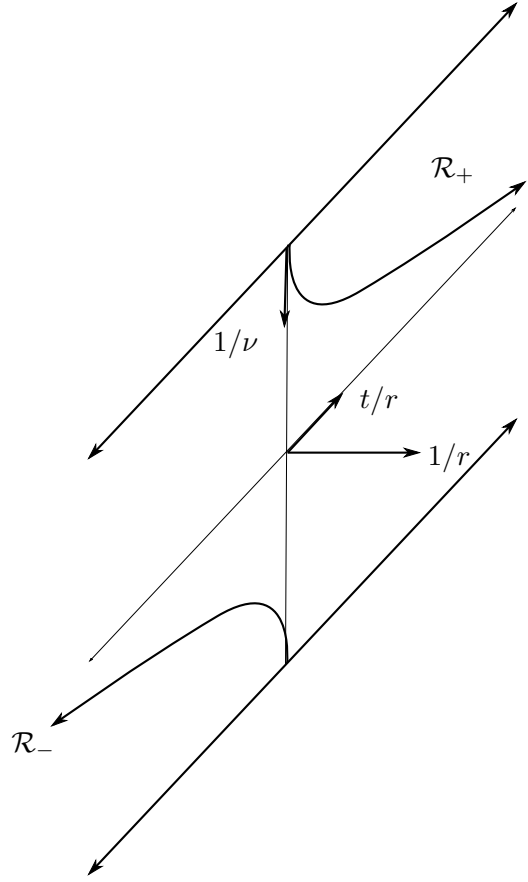
	\caption{Near the ``equator'' $t = 0$ in the boundary, the radial
		sets remain disjoint and intersect the corner normally.}  
\end{figure}

We next take the case that $\zeta_1$ is very negative. In this case we  use fibre boundary defining function $\rho_f = -1/\zeta_1$ and coordinates $\omega_j = \zeta_j/|\zeta_1|$ and $\sigma = \tau/|\zeta|^2$. We rescale the Hamilton vector field by multiplying by (the new) $\rho_f$ and repeat the calculation. We obtain the vector field 
$$
\rho_b^{-1} \rho_f H_{p_0} = \big( \rho_f + 2s \big) \frac{\partial}{\partial s} + \sum_{j \geq 2} \big( 2 \omega_j + 2  v_j \big) \frac{\partial}{\partial v_j} + 2  \rho_b \frac{\partial}{\partial \rho_b}.
$$
Changing variables to $\tilde s = 2s + \rho_f$, $\tilde v_j = v_j + \omega_j$ then in coordinates $(\tilde s, \tilde v_j, \rho_b, \rho_f, \omega_j, \sigma)$ we have the vector field
\begin{equation}\label{eq:Hvf Rminus equator}
2\tilde s \frac{\partial}{\partial \tilde s} + 2  \sum_{j \geq 2} \tilde v_j \frac{\partial}{\partial \tilde v_j} + 2  \rho_b \frac{\partial}{\partial \rho_b}.
\end{equation}
This part of the radial set coincides with $\mathcal{R}_-$ where $s < 0$, and can be expressed as 
\begin{equation}\label{eq:rad.corner.2}
\mathcal{R}_- = \{ \rho_b = 0, \ s = -\rho_f/2, \ \hat \zeta = -\hat z, \ \tau/|\zeta|^2 = -1 \}.
\end{equation}
Comparing \eqref{eq:rad.corner} and \eqref{eq:rad.corner.2}, we see that the two sets are disjoint. In fact, we have $\hat \zeta =  \hat z$ for $\mathcal{R}_+$, while $\hat \zeta =  -\hat z$ for $\mathcal{R}_-$. 
The rescaled Hamilton vector field in \eqref{eq:Hvf Rminus equator} is clearly a source at $\mathcal{R}_-$ in this region (which is another way to understand the disjointness of the two components). Similar to $\mathcal{R}_+$, the component $\mathcal{R}_-$ is a graph over the open southern hemisphere but turns vertical at the equator and reaches fibre-infinity there.

It remains to prove the last statement in the Proposition. For this we need the nontrapping assumption, that is, that for each fixed $t_0$, the metric $g(t_0)$ is nontrapping: every geodesic for $g(t_0)$ in $T^* \RR^n$ reaches spatial infinity both forwards and backwards. That means, in particular, that every bicharacteristic $\gamma(s)$ for $P$ coincides with a bicharacteristic for $P_0$ when $s$ is sufficiently negative or sufficiently positive. Recalling the discussion at the end of Section~\ref{subsec:charvar}, it follows that it suffices to prove the statement for $P_0$. 

Now because the rescaled Hamilton vector field has a smooth extension to the boundary of the compactified parabolic cotangent bundle, which is tangent to the boundary, it suffices to prove the statement for flow lines in the interior, and take a limit as the flow lines approach the boundary. So we consider an interior flow line for $P_0$, contained within $\{ p_0 = 0 \}$. These take the form for some $t_0, z_0, \zeta_0$, 
$$
t(s) = t_0 + s, \quad z(s) = z_0 + 2\zeta_0 s, \quad \zeta = \zeta_0, \quad \tau = - |\zeta_0|^2.
$$
As $s \to +\infty$ we see that $z/t$ converges to $2\zeta_0$, while $\tau$ is fixed at $-|\zeta_0|^2$. We see that this converges to a point of $\mathcal{R}_+$. Similarly, as $s \to -\infty$, we have $z/|t| \to -2\zeta_0$, $\tau = -|\zeta_0|^2$ so this converges to a point of $\mathcal{R}_-$.

\end{proof}

We see that the radial sets $\mathcal{R}_\pm$ are located as illustrated in Figure~\ref{fig:Rpm}: their projections to spacetime infinity meet, but they are nevertheless separated as subsets of the compactified phase space.

The final task in this section is to observe the symplectic nature of the radial sets. 
Let $\tilde \SR_\pm$ denote the $(n+1)$-dimensional submanifold of $\overline{T}^*_{\mathrm{par}}\mathbb{R}^{n + 1}$ uniquely determined by the following two conditions:  
\begin{itemize}
\item $\tilde \SR_\pm$ are invariant under spatial dilation, $Z \mapsto aZ$ for $a \in \RR_+$, and
\item The intersection of $\tilde \SR_\pm$ with $\partial_{base} \overline{T}^*_{\mathrm{par}}\mathbb{R}^{n + 1}$ is $\SR_\pm$.
\end{itemize}

\begin{lemma}\label{lem:Lagrangian} The submanifolds $\tilde \SR_\pm$ defined above are Lagrangian submanifolds for the standard symplectic form $d\zeta dz + d\tau dt$ on $\overline{T}^*_{\mathrm{par}}\mathbb{R}^{n + 1}$. 
\end{lemma}

\begin{proof} Since $\tilde \SR_\pm$ are smooth submanifolds with boundary, of the correct dimension, it is only necessary to verify the Lagrangian condition, i.e.\ that the symplectic form vanishes when restricted to $\tilde \SR_\pm$, in the interior of $\tilde \SR_\pm$. In this region we can use the coordinates $(w = z/t, \rho = 1/t, \zeta, \tau)$, as in the beginning of the proof of Proposition~\ref{prop:radial set geom}. In these coordinates $\tilde \SR_\pm$ is given by 
$$
\tilde \SR_\pm = \{ \zeta = \frac{w}{2}, \tau = - \frac{|w|^2}{4}, \pm \rho > 0 \},
$$
and the symplectic form restricted to $\tilde \SR_\pm$ is 
$$
d \big( \frac{z}{2t} \big) dz + d \big( -\frac{|z|^2}{4t^2} \big) dt = - \frac{z}{2t^2} dt dz - \frac{z}{2t^2} dz dt = 0.
$$
\end{proof}


\section{Module regularity}
\label{sec:modregdef}
\subsection{Test modules}
In addition to the (parabolic) scattering Sobolev spaces considered in Definition \ref{def:sobolev} and their variable order analogues in Definition \ref{def:sobolevdelta}, we shall require the notion of iterated regularity with respect to a test module of pseudodifferential operators, introduced in \cite{HMV2004}. In this section we shall work exclusively with pseudodifferential operators with classical symbols, $\Psi_{\mathrm{par, cl}}^{m,l}(\RR^{n+1})$ (see Definition~\ref{def:classical}).

\begin{defn}\label{def:modreg} 
A test module of operators contained in $\Psipcl{1}{1}(\RR^{n+1})$ is a vector subspace of 
$\Psipcl{1}{1}(\RR^{n+1})$ that contains and is a module over $\Psipcl{0}{0}(\RR^{n+1})$, is finitely generated over $\Psipcl{0}{0}(\RR^{n+1})$, and is closed under commutators. 
\end{defn}

Consider an arbitrary test module $\SM\subset \Psip{1}{1}$, with generating set $\mathbf{A}=(\mathbf{A}_j)_{j=0}^N\subset \Psip{1}{1}$ where $\mathbf{A}_0=\mathrm{Id}$. 
Powers of the module can be defined in the natural way.
\begin{defn}
	\label{def:module.power}
	For arbitrary $\kappa\in\NN$ we define $\SM^{\kappa}$
	to be the $\Psip{0}{0}$-module generated by the set 
	\begin{equation}
		\{A_\alpha=\mathbf{A}^\alpha:|\alpha|\leq \kappa \},
	\end{equation}
where $\alpha = (\alpha_0, \dots, \alpha_N)$ is a multi-index and $\mathbf{A}^\alpha$ denotes the composition $$\mathbf{A}^\alpha := \mathbf{A}_0^{\alpha_0} \mathbf{A}_1^{\alpha_1} \dots \mathbf{A}_N^{\alpha_N}.$$ Notice that the ordering of the factors is immaterial  due to the fact that the module is by definition closed under commutators. 
	
	Equivalently, $\SM^{\kappa}$ is the module generated by all $\kappa$-fold products of elements of $\SM$. 
\end{defn}


Let $\SM$ be a test module. We define Sobolev spaces of functions with additional regularity with respect to $\SM$ as follows.
\begin{defn}
	\label{def:module.reg}
	Let $s$ and $l$ be real numbers, and $\kappa$ a natural number. We define the space of functions with $\SM$-module regularity of order $\kappa$ in $H^{s,l}_{\mathrm{par}}(\RR^{n+1})$ by 
	\begin{equation}\label{eq:1module}
		H_\SM^{s,l;\kappa}:=\{u\in H^{s,l}_{\mathrm{par}}(\RR^{n+1}):A u \in H^{s,l}_{\mathrm{par}}\textrm{ for all } A\in\SM^{\kappa}\}.
	\end{equation}
\end{defn}
Concretely, $u\in H_\SM^{s,l;\kappa}$ if and only if $\mathbf{A}^\alpha u\in H^{s,l}_\mathrm{par}$ for all $\alpha$ with $|\alpha| \leq \kappa$.
It shall be useful to assume additional regularity with respect to a fixed submodule  $\SN\subset \SM$ with generating set $\mathbf{B}=(\mathbf{B}_j)_{j=0}^{N'}$. To this end, we introduce the the following refinement of \eqref{eq:1module}.
\begin{defn}
	\label{def:2modules}
	We define the space of functions with $\SM$-module regularity of order $\kappa$ and $\SN$-module regularity of order $k$ in $H^{s,l}_{\mathrm{par}}(\RR^{n+1})$ by
	\begin{equation}\label{eq:2modules}
		H_\SM^{s,l;\kappa,k}:=\{u\in H^{s,l}_{\mathrm{par}}(\RR^{n+1}):AB u \in H^{s,l}_{\mathrm{par}}\textrm{ for all }A\in \SM^{\kappa}\textrm{ and }B\in\SN^{k}.\}
	\end{equation}
\end{defn}
We suppress the $\SN$ in our notation, as in this paper it will only ever be used with the specific module $\SN$ from Definition~\ref{def:module.def.characteristic}. Concretely, $u\in H_\SM^{s,l;\kappa,k}$ if and only if $\mathbf{A}^\alpha \mathbf{B}^\beta u\in H^{s,l}_\mathrm{par}$ for all $\alpha,\beta$ with $|\alpha| \leq \kappa$ and $|\beta|\leq k$.

We equip the spaces $H_\SM^{s,l;\kappa},H_\SM^{s,l;\kappa,k}$ with a Hilbert space structure by fixing a choice of generators $\mathbf{A},\mathbf{B}$ and taking
\begin{equation}
	\|u\|_{H_\SM^{s,l;\kappa}}^2:= \sum_{|\alpha|\leq \kappa } \|\mathbf{A}^\alpha  u\|_{H_\mathrm{par}^{s,l}}^2
\end{equation}
and
\begin{equation}
	\label{eq:mod.reg.sob.norm}
	\|u\|_{H_\SM^{s,l;\kappa,k}}^2:= \sum_{|\alpha|\leq \kappa }\sum_{|\beta|\leq k} \|\mathbf{A}^\alpha \mathbf{B}^\beta u\|_{H_\mathrm{par}^{s,l}}^2 .
\end{equation}
\begin{remark}
	Definition \ref{def:module.reg} and Definition \ref{def:2modules} generalise in the natural way to the case of variable spatial weight $\sw\in\Psip{0}{0}$.
\end{remark}

\subsection{The modules $\SM_\pm$ and $\SN$}

We now introduce two specific modules $\SM_{\pm}$, and a common submodule $\SN$ that shall be the modules of interest in this paper. 

\begin{defn}
	\label{def:module.def.characteristic}
	If $\SR=\SR_+\cup \SR_-$ is the radial set for the operator $P$ introduced in Section \ref{subsec:radial}, we define $\SM_\pm$ by 	
	\begin{equation}
	\SM_\pm :=	\{A\in \Psipcl{1}{1}(\RR^{n+1}) \mid 
		[\rho_{base}\sigma_{1,1}(A)] \, |_{\SR_{\pm}}=0\}. 
	\end{equation}
	We also define
	\begin{equation}\label{eq:Ndef}
		\SN:= \SM_+\cap \SM_-
	\end{equation}
and we write $H_\pm^{s,l;\kappa,k}$ for the module regularity space \eqref{eq:2modules} when $\SM = \SM_\pm$ is as in \eqref{def:module.def.characteristic} and $\SN$ is as in \eqref{eq:Ndef}. 
\end{defn}

\begin{remark} We restrict to classical operators precisely so the above definition makes sense. Indeed, the classical assumption means that $\rho_{base} \sigma_{1,1}(A)$ has an smooth extension to the compactified parabolic cotangent bundle $\overline{T}^*_{\mathrm{par}}\mathbb{R}^{n + 1}$, so the restriction of this function to $\SR_\pm$ is meaningful; this is not true for a general operator in $\Psip{1}{1}(\RR^{n+1})$. 
\end{remark}

\begin{proposition} The modules $\SM_\pm$, and therefore also $\SN$, are test modules in the sense of Definition~\ref{def:modreg}. 
\end{proposition}

\begin{proof} 
It is clear that $\SM_\pm$ are vector subspaces of $\Psipcl{1}{1}(\RR^{n+1})$ that contain and are modules over $\Psipcl{0}{0}$. It remains to show that these are finitely generated and are closed under commutators. 

Closedness under commutators follows from Lemma~\ref{lem:Lagrangian}. Indeed, if $A_1$ and $A_2$ are two module elements, then we can find $\tilde A_i$, $i = 1, 2$,  that differ from $A_i$ by an element of $\Psipcl{0}{0}(\RR^{n+1})$ and have symbol invariant under the scaling $Z \mapsto aZ$ near spacetime infinity. Indeed we just take the symbol of $\tilde A_i$ so that $\rho_{base} \sigma_{1,1}(\tilde A_i)$ is invariant under the scaling and to agree with  $\rho_{base} \sigma_{1,1}(A_i)$ at spacetime infinity. It is clear that $\tilde A_i$ are also in the module. Then the symbols of $\tilde A_1$ and $\tilde A_2$ vanish on $\tilde \SR_\pm$. We now use the standard fact in symplectic geometry that if a Hamiltonian is constant on a Lagrangian submanifold, then its Hamilton vector field is tangent to this submanifold. Denoting the symbols of $\tilde A_i$ by $\tilde a_i$, It follows that $H_{\tilde a_1} \tilde a_2$ vanishes on $\tilde \SR_\pm$. 
This is the principal symbol of $i[\tilde A_1, \tilde A_2]$, so it follows that $\rho_{base} \sigma_{1,1}(i[\tilde A_1, \tilde A_2])$ vanishes at $\SR_\pm$, so this operator is also in the module. But $[A_1, A_2]$ differs from $[\tilde A_1, \tilde A_2]$ by an operator of order $(0,0)$, so we see that $[A_1, A_2]$ also lies in the module $\SM_\pm$. 

To show that the modules $\SM_\pm$ are finitely generated, we will exhibit an explicit set of generators. To begin with, we introduce some useful cutoff functions. Taking $\chi\in\mathcal{C}^\infty(\RR)$ to be $0$ on $(-\infty,1/3]$ and is $1$ on $[2/3,\infty)$, we define
\begin{equation}
	\label{eq:north.pole.cutoff}
	\NP:=\chi\left(\frac{t}{\ang{Z}}\right)
\end{equation}
\begin{equation}
	\label{eq:south.pole.cutoff}
	\SP:=\chi\left(-\frac{t}{\ang{Z}}\right)
\end{equation}
\begin{equation}
	\label{eq:equatorial.cutoff}
	\EQ:=1-\NP-\SP
\end{equation}
where $\chi(1/|w|)$ is extended to be $1$ at $w=0$.
These functions can be regarded as cutoffs to the poles and equator of the boundary of space-time. 

Due to the parabolic nature of our calculus, we also need to work with microlocal square roots of $D_t$. This requires pseudodifferential cutoffs based on the sign of $\tau$. To this end we introduce
\begin{equation}
	\TP=\chi\left(\frac{2\tau}{\ang{R^2}}\right)
\end{equation}
\begin{equation}
	\TN=\chi\left(-\frac{2\tau}{\ang{R^2}}\right).
\end{equation}

Our candidate generating sets are then as follows.
\begin{equation}
	\label{eq:tildeG.def}
	\mathcal{G}_{\pm}:=\{z_iD_{z_j}-z_jD_{z_i},tD_{z_i}-z_i/2,\ang{Z}E_{-1}P,E_1, t\chi_{\mathrm{pol},\mp}E_1,B_{\pm}\}
\end{equation}
where $E_s=\mathrm{Op}(\ang{R^2}^{s/2})\in \Psip{s}{0}$ is elliptic, and  $B_{\pm}=\mathrm{Op}(b_{\pm})$, where $b_\pm\in S_{\mathrm{par}}^{1,1}$ is equal to the smooth function
\begin{equation}
	|z|(\rho \mp\sqrt{-\tau})\TN\EQ
\end{equation}
in a neighbourhood of the spacetime boundary.

We only show that $\mathcal{G}_+$ generates $\SM_+$, as the case for $\mathcal{G}_-$ and $\SM_-$ is similar. We claim that it suffices to show two properties of $\mathcal{G}_+$:
\begin{itemize}
\item For any point $q \in \partial_{base} \overline{T}^*_{\mathrm{par}}\mathbb{R}^{n + 1}$ not in $\SR_+$, there is an element of $\mathcal{G}_+$ elliptic at $q$, and 
\item For any point $q \in \SR_+$, there are $n+1$ elements of $\mathcal{G}_+$, say $A_1, \dots, A_{n+1}$, such that the functions $\rho_{base}\rho_{fib}\sigma_{1, 1}(A_i)$, viewed as functions on $\partial_{base} \overline{T}^*_{\mathrm{par}}\mathbb{R}^{n + 1}$, 
 have linearly independent differentials at $q$. 
\end{itemize}

We first prove this claim. Let $A$ be an arbitrary element of $\SM_+$. It suffices to show that $O A$ is in the module generated by $\mathcal{G}_+$ for any $O \in \Psipcl{0}{0}$ with arbitrarily small microsupport. Thus it suffices to prove assuming that the microsupport of $A$ is contained in a small neighbourhood of a point $q \in  \partial_{base} \overline{T}^*_{\mathrm{par}}\mathbb{R}^{n + 1}$.

If $q$ is not in $\SR_+$, then by the first property above, there is an element $E$ of $\mathcal{G}_+$ elliptic at $q$, and we can assume the microsupport of $A$ is contained in the elliptic set of $E$. Then by the standard elliptic construction, we have $A = Q E + R$ where $Q \in \Psipcl{0}{0}$ and $R \in \Psipcl{-\infty}{-\infty}$, hence $A$ is in the module generated by $\mathcal{G}_+$. 

If $q$ is in $\SR_+$, then there exist $A_1, \dots, A_{n+1} \in \mathcal{G}_+$ as in the second property above. Since $\SR_+ \subset \partial_{base} \overline{T}^*_{\mathrm{par}}\mathbb{R}^{n + 1}$ is a smooth submanifold of codimension $n+1$, it follows that $a_1 = \rho_{base}\rho_{fib} \sigma_{1, 1}(A_1), \dots, a_{n+1} = \rho_{base}\rho_{fib} \sigma_{1, 1}(A_{n+1})$ are defining functions for $\SR_+$ locally. That means that any function $a$ vanishing on $\SR_+$ and supported sufficiently close to $q$ can be expressed 
\begin{equation}\label{eq:abj}
a = \sum_{j=1}^{n+1} b_j a_j,
\end{equation}
locally near $q$, for some smooth functions $b_j$. In particular, given arbitrary $A \in \SM_+$, this is true for $a = \rho_{base}\rho_{fib} \sigma_{1, 1}(A)$. We can extend the $b_j$, which are smooth functions on $\partial_{base} \overline{T}^*_{\mathrm{par}}\mathbb{R}^{n + 1}$,  to classical symbols of order $(0,0)$ on $\overline{T}^*_{\mathrm{par}}\mathbb{R}^{n + 1}$; let $B_j$ be the left quantization of these symbols. Then \eqref{eq:abj} implies
$$
A = \sum_{j=1}^{n+1} B_j A_j + A', \quad A' \in \Psipcl{1}{0}.
$$
Since $E_1 \in \mathcal{G}_+$ is elliptic as an element of $\Psip{1}{0}$, similarly to the first case, 
we can write $A' = Q'E_1 + R'$ where $Q' \in \Psipcl{0}{0}$ and $R' \in \Psipcl{-\infty}{-\infty}$. Putting these together we have 
$$
A = \sum_{j=1}^{n+1} B_j A_j + Q'E_1 + R'
$$
which shows that $A$ is in the module generated by $\mathcal{G}_+$. 

The proof is therefore completed by the following Proposition. 
\end{proof}

\begin{prop}
	\label{prop:R.char}
	The radial sets $\SR$ and $\SR_\pm$ can be characterized as the set of common zeroes of the elements of $\mathcal{G}_+ \cap \mathcal{G}_-$ and $\mathcal{G}_\pm$: 
	\begin{equation}
		\label{eq:R.char}
		\SR=\bigcap_{G\in \SG_+\cap \SG_-} \Char_{1,1}(G)
	\end{equation}
	\begin{equation}
		\label{eq:R.char.2}
		\SR_\pm=\bigcap_{G\in \SG_\pm} \Char_{1,1}(G)
	\end{equation}
Equivalently, for each point $q$ not in each of these radial sets, there is an element of the corresponding generating set elliptic at $q$. 

Moreover, for any point $q$ of $\SR$, there are $n+1$ elements of $\mathcal{G}_+ \cap \mathcal{G}_-$, say $A_1, \dots, A_{n+1}$ such that $\rho_{base}\rho_{fib} \sigma_{1,1}(A_i)$, restricted to $\partial_{base} \overline{T}^*_{\mathrm{par}}\mathbb{R}^{n + 1}$, 
 have linearly independent differentials at $q$. 	
\end{prop}

\begin{proof}
	First we show that
	\begin{equation}
		\SR =\bigcap_{G\in \SG_+\cap\SG_-} \Char_{1,1}(G).
	\end{equation}
	
	Since $E_1\in \Psip{1}{0}$ is elliptic, we have that $\Char_{1,1}(E_1)=\bface$. Thus the set of common zeroes of these generating sets is contained in $\bface$. 
	
Away from the equator at spacetime infinity (the subset of spacetime infinity where $t/|z|= 0$), and where $\Phi = (\zeta, \tau)$ is finite, we	can use the coordinates $(w,\zeta,\tau)$ on $\bface$. Then consider the $n+1$ generating operators $A_i = 2tD_{z_i}-z_i$, $i = 1 \dots n$, and $A_{n+1} = \ang{Z}E_{-1}P$. As we are away from fibre-infinity we can consider $\sigma_{base, 1}(A_i) = t^{-1} \sigma_{1,1}(A_i)$. These functions are 
$2\zeta_i-w_i$ and $\tau + |\zeta|^2$, up to a smooth nonvanishing factor. The set of common zeroes is therefore precisely $\SR$, as in the proof of Proposition~\ref{prop:radial set geom}. Moreover, these functions have linearly independent differentials at each point of $\SR$. At fibre-infinity, which is disjoint from $\SR$, it is easy to check that at least one of these operators is elliptic. This proves the Proposition for 
$\SR$ away from the equator.

	At a point $q$ on the equator of $\bface$ and away from fibre-infinity, hence not lying on $\SR$, we can assume without loss of generality (similarly to the proof of Proposition~\ref{prop:radial set geom}) that $z_1 \geq \frac12 \max_i |z_i|>0$ and locally use the coordinates $s=t/z_1$ and $v_i=z_i/z_1$ for $2\leq i \leq n$ together with the fiber coordinates $\zeta,\tau$. In these coordinates $\sigma_{\mathrm{base},1,1}(2tD_{z_1}-z_1)$ is locally given by $2s\zeta_1-1$ up to a smooth nonvanishing factor. At $s=0$, this is nonzero and so $\Char_{1,1}(E_1)\cap \Char_{1,1}(2tD_{z_1}-z_1)$ is disjoint from this region, as is $\SR$.
	
	Near a point on the equator of $\bface$ near $\fface \cap \Char_{1,1}(2tD_{z_1}-z_1)$,  we can locally use the coordinates $s, v_j$ as above together with $\rho_b = 1/z_1$, $\omega_j=\zeta_j/\zeta_1$ for $2 \leq j \leq n$, $\sigma=\tau/|\zeta|^2$ and $\rho_{\mathrm{f}}=\pm 1/\zeta_1$, where the latter has sign chosen such that it is nonnegative, and therefore a local boundary defining function for $\fface\cap\bface\subset \bface$. In this case we choose $A_1 = 2tD_{z_1}-z_1$, $A_j = z_1D_{z_j}-z_jD_{z_1}$, $2 \leq j \leq n$ and $A_{n+1} = \ang{Z}E_{-1}P$. The corresponding functions $a_i = \rho_{base}\rho_{fib} \sigma_{1,1}(A_i)$ are $2s - \rho_{\mathrm{f}}$, $\omega_j - v_j$ and $\sigma + 1$, up to smooth nonvanishing factors. Again, the set of common zeroes is precisely $\SR$, according to 
Proposition~\ref{prop:radial set geom}. It is easy to check that these functions have linearly independent differentials in this region. This completes the proof of the proposition for $\SR$.

	We now prove \eqref{eq:R.char.2} for the top sign choice, with the argument for the bottom sign choice being similar. It suffices to show that the two additional generators $t\chi_{\mathrm{pol},-}E_1,B_+$ are both characteristic on $\SR_+$ but at least one of them is elliptic at any point on $\SR_-$. Due to the space-time cutoff factor, $t\chi_{\mathrm{pol},-}E_1$ is clearly characteristic on the entirety of $\SR_+$, and is elliptic on $\{|w| \leq 2\sqrt 2\}\cap \SR_-$.
	The restriction of $\sigma_{\mathrm{base},1,1}(B_+)$ to the interior of spatial infinity has symbol supported in the equatorial region $\spt(\EQ)\cap \{35\tau^2\geq 1+|\zeta|^4\}$.
	
	As  $(\rho-\sqrt{-\tau})$ is given by $\hat z\cdot \hat\zeta - 1$ on $\SR$ up to a smooth nonvanishing factor, it follows from \eqref{eq:rad.corner},\eqref{eq:rad.corner.2} that $B_+$ is elliptic on the part of $\mathrm{int}(\SR_-)$ in $\spt(\EQ)$ such that $\tau\leq -1/\sqrt{34}$. In fact the latter condition is superfluous, as from \eqref{eq:rad.interior} we see that $|w|\geq \sqrt{5}/2 \Rightarrow \tau \leq -\frac{5}{16}$. Hence every point in $\mathrm{int}(\SR_-)$ lies in $\Ell(G)$ for some $G\in \SG_+$.
	
	Similarly, if $q\in \partial(\SR_-)$, we can use the coordinates $\omega_j,\sigma,\rho_{\mathrm{f}}$ in this region and $\sigma_{\mathrm{base},1,1}(B_+)(q)$ is a nonzero multiple of $\hat z\cdot \hat\zeta-\sqrt{-\sigma}$. As $\sigma=-1$ on $\SR$ and $\hat z=-\hat{\zeta}$ on $\SR_-$, we obtain $\sigma_{\mathrm{base},1,1}(B_+)(q)\neq 0$ hence establishing that $B_+$ is elliptic on $\partial(\SR_-)$.
\end{proof}

The collections $\mathcal{G}_\pm$ generate $\SM_\pm$ globally, but for $u\in H^{s,l}_{\mathrm{par}}$ that is supported in certain subregions of $\RR^{n+1}$, we can characterise $H_{\mathcal{M}_{\pm}}^{s,l;\kappa,k}$ regularity using smaller collection of generators. Three such regions of interest are
\begin{enumerate}
	\item $\{\ang{Z}\leq R\}$ (Space-time interior)
	\item  $\{\ang{Z}\geq R > 0,|t| \geq \frac13 |z| \}$  (Space-time boundary, near poles)
	\item $\{\ang{Z}\geq R > 0,|t| \leq \frac23 |z| \}$ (Space-time boundary, near equator)
\end{enumerate}

In the space-time interior, we can use the single elliptic generator $E_1\in\Psip{1}{0}$.

\begin{prop}
	For any $R>1$ and for any distribution $u\in H_\mathrm{par}^{s,l}$ with $\spt(u)\subset \{\ang{Z}\leq R\}$, we have 
	\begin{equation}
		u\in H_{\SM_{\pm}}^{s,l;\kappa,k}\Leftrightarrow  u\in H_{\SM_{\mathrm{int}}}^{s,l;\kappa+k}\Leftrightarrow u \in H_{\mathrm{par}}^{s+\kappa+k,\infty}
	\end{equation}
	where $\SM_{\mathrm{int}}$ is the $\Psip{0}{0}$-module generated by 
	\begin{equation}
		\SG_{\mathrm{int}}:=\{E_1\}.
	\end{equation}
\end{prop}
\begin{proof}
	From the definition of the module regularity Sobolev spaces, it is immediate that $u\in H_{\mathrm{par}}^{s+\kappa+k,\infty}\Rightarrow u\in H_{\SM_{\pm}}^{s,l;\kappa,k}\cap H_{\SM_{\mathrm{int}}}^{s,l;\kappa+k}$. Thus it suffices to show that membership in either module regularity space, together with the support condition, implies membership in $H_{\mathrm{par}}^{s+\kappa+k,\infty}$.
	
	Suppose $u$ lies in either of the two module regularity spaces. As the generator $E_1$ lies in both modules, this means that
	\begin{equation}
		u, E_1 u, \ldots, E_1^{k+\kappa }u \in H_{\mathrm{par}}^{s,l}
	\end{equation}
	and so by ellipticity it follows that $u\in H_{\mathrm{par}}^{s+k+\kappa,l}$. From the bounded support condition, we can immediately upgrade this to $u\in H_{\mathrm{par}}^{s+\kappa+k,\infty}$.
	
	On the other hand, if we assume $u\in H_{\SM_{\pm}}^{s,l;\kappa,k}$, we can exploit  the regularity of $u$ with respect to the elliptic generator $E_1\in\Psip{1}{0}$. From $u\in H_\mathrm{par}^{s,l}$ and $E_1^{\kappa+k}u\in H_\mathrm{par}^{s,l}$, it follows that $u\in H_\mathrm{par}^{s+\kappa+k,l}$. As $u$ is compactly supported in spacetime, we can once again upgrade this to $u\in H_{\mathrm{par}}^{s+\kappa+k,\infty}$ as required.	
\end{proof}

Similarly, although we shall not use this fact, for distributions supported in the polar and equatorial regions we can characterise $\SM_\pm$-regularity using a smaller collection of generators.

\begin{prop}
	For any $R>1$ and for any distribution $u\in H_\mathrm{par}^{s,l}$ with $\spt(u)\subset \{\ang{Z}\geq R,|t|\geq C|z|\}$, we have 
	\begin{equation}
		u\in H_{\SM_{\pm}}^{s,l;\kappa,k}\Leftrightarrow  u\in H_{\SM_{\pm,\mathrm{pol}}}^{s,l;\kappa,k}.
	\end{equation}
	where $\SM_{\mathrm{\pm,\mathrm{pol}}}$ is the $\Psip{0}{0}$-module generated by 
	\begin{equation}
		\SG_{\pm,\mathrm{pol}}:=\{tD_{z_i}-z_i/2,tE_{-1}P,E_1, t\chi_{\mathrm{pol},\mp}E_1\}
	\end{equation}
	and the submodule $\SN_{\mathrm{pol}}$ is the $\Psip{0}{0}$-module generated by $\SG_{+,\mathrm{pol}}\cap \SG_{-,\mathrm{pol}}$.
\end{prop}

\begin{prop}
	For any $R>1$ and for any distribution $u\in H_\mathrm{par}^{s,l}$ with $\spt(u)\subset \{\ang{Z}\geq R,|t|\leq C|z|\}$, we have 
	\begin{equation}
		u\in H_{\SM_{\pm}}^{s,l;\kappa,k}\Leftrightarrow  u\in H_{\SM_{\pm,\mathrm{eq}}}^{s,l;\kappa,k}
	\end{equation}
	where $\SM_{\mathrm{\pm,\mathrm{eq}}}$ is the $\Psip{0}{0}$-module generated by 
	\begin{equation}
		\SG_{\pm,\mathrm{eq}}:=\{z_iD_{z_j}-z_jD_{z_i},tD_{z_i}-z_i/2,\ang{Z}E_{-1}P,E_1,B_{\pm}\}
	\end{equation}
	and the submodule $\SN_{\mathrm{eq}}$ is the $\Psip{0}{0}$-module generated by $\SG_{+,\mathrm{eq}}\cap \SG_{-,\mathrm{eq}}$.
\end{prop}

\subsection{Positivity properties}

When proving positive commutator estimates for module regularity spaces, the following notions of positivity are extremely useful.

\begin{defn}
	\label{def:p.pos}
	Let $\SM$ be a finitely generated  $\Psip{0}{0}$-module, with generators $A_0 = \Id, A_1,\ldots,A_N \subset \Psipcl{1}{1}$.
	We say $\SM$ is \emph{$P$-positive} on the subset $S\subseteq \Char(P)$ of spatial infinity if for each $j$, there exist $C_{jk}\in \Psip{1}{0}$ and $C_j'\in\Psip{0}{1}$ such that we have 
	\begin{equation}
		\label{eq:gen.commutators}
		i\ang{Z}[A_j,P]=\sum_{k=0}^N C_{jk}A_k+C_j'P
	\end{equation}
	with 
	\begin{equation}
		\sigma_{\mathrm{base},1,0}(C_{jk})|_S=0\textrm{ for }j\neq k
	\end{equation} and
	\begin{equation}
		\label{eq:module.positivity.diagonal}
		\Re(\sigma_{\mathrm{base},1,0}(C_{jj}))|_S\geq 0.
	\end{equation}
	Similarly, we say that $\SM$ is \emph{$P$-negative} on $S$ if the same conditions are satisfied with the inequality \eqref{eq:module.positivity.diagonal} reversed.
	We say $\SM$ is \emph{$P$-critical} on $S$ if $\SM$ is both \emph{$P$-positive} and \emph{$P$-negative} on $S$.
\end{defn}

We conclude the section by showing that the modules $\SM_\pm$ and $\SN$  satisfy positivity properties phrased using Definition \ref{def:p.pos}. We shall exploit these in Section \ref{subsec:module.positive.commutator.radial}.

\begin{prop}
	\label{prop:positivity}
	The modules defined in Definition \ref{def:module.def.characteristic} enjoy the following positivity properties.
	\begin{enumerate}[(i)]
		\item $\SM_+$ is $P$-positive at $\SR_+$;
		\item $\SM_-$ is $P$-negative at $\SR_-$;
		\item $\SN$ is $P$-critical at $\SR=\SR_+\cup \SR_-$.
	\end{enumerate}
\end{prop}
\begin{proof}
	The commutators of the first two differential generators with $P_\mathrm{euc}=\Delta_{\mathrm{euc}}+D_t$ are as follows
	\begin{equation}
		[z_iD_{z_j}-z_jD_{z_i},P_\mathrm{euc}]=0,
	\end{equation}
	\begin{equation}
		[tD_{z_j}-z_j/2,P_\mathrm{euc}]=0,
	\end{equation}
	As $g$ is Euclidean outside of a compact set in space-time, we have that $P_\mathrm{euc}-P$ is a compactly supported differential operator in $\Psip{2}{-\infty}$. This implies that the commutators 
	\begin{equation}
		i\ang{Z}[z_iD_{z_j}-z_jD_{z_i},P]
	\end{equation}
	and
	\begin{equation}
		i\ang{Z}[tD_{z_j}-z_j/2,P]
	\end{equation}
	are compactly supported differential operators in $\Psip{2}{-\infty}$. They can both be written in the form $CE+R$ for some $C\in \Psip{1}{-\infty}$ and $R\in \Psip{-\infty}{-\infty}$ by using the ellipticity of the generator $E$, moreover $\sigma_{\mathrm{base},1,1}(C)$ vanishes on $\SR$ in both instances as $C\in\Psip{1}{-\infty}$.
	The generator involving $P$ has commutator
	\begin{equation}
		i\ang{Z}[\ang{Z}E_{-1}P,P]=i\ang{Z}[\ang{Z}E_{-1},P]P
	\end{equation}
	which is also of the required form as $[\ang{Z}E_{-1},P]\in\Psip{0}{0}$.
	
	The elliptic generator $E_1$ itself is of lower order $(1,0)$, and so the corresponding commutator can also be written in the form
	\begin{equation}
		i\ang{Z}[E_1,P]=CE_1+R
	\end{equation}
	where $C\in \Psip{1}{0}$ and $R\in \Psip{-\infty}{-\infty}$. As $P=P_\mathrm{euc}$ outside of a compact set in space-time, the Poisson bracket $\{\sigma(E),\sigma(P)\}$ vanishes identically near the boundary of the space-time compactification. Consequently $C$ has base symbol vanishing on $\SR$. These computations establish that $\SN$ is $P$-critical at $\SR$.
	
	We now consider the two additional generators of $\SM_+$.	For $A\in\Psip{1}{1}=\mathrm{Op}(a)$, we have $\sigma_{2,0}([A,P])=iH_p a = i(\partial_t +2\zeta\cdot \partial_z )a$ in a neighbourhood of spatial infinity, and so applying this to the generator with symbol
	\[a=t\ang{R^2}^{1/2}\chi_{\mathrm{pol},-}\]
	we obtain 
	\begin{align}
		& i\ang{Z} \sigma_{2,0}([A,P])\\
		&=-\ang{R^2}^{1/2}\left[\chi_{\mathrm{pol},-}\ang{Z}+\chi' \left(\frac{-t}{\ang{Z}}\right)\left(\frac{-t-t|z|^2+2t^2 z\cdot \zeta}{\ang{Z}^2}\right) \right]\\
		&=\label{eq:first.large.gen.positive} \frac{\ang{Z}}{-t}a-\ang{R^2}^{1/2}\chi' \left(\frac{-t}{\ang{Z}}\right)\left(\frac{-t-t|z|^2+2t^2 z\cdot \zeta}{\ang{Z}^2}\right)
	\end{align}
	The prefactor of $a$ is non-negative on the support of $a$. The term $\ang{R^2}^{1/2}\chi'\cdot \left(-\frac{t}{\ang{Z}^2}\right)$ is of lower order $(-\infty,-1)$.
	The remaining terms can be written as
	\begin{equation}
		\label{eq:remains}
		\ang{R^2}^{1/2}\chi'\left(\frac{-t}{\ang{Z}^2}\right)\cdot \left(\frac{2t^2z\cdot\zeta-t|z|^2}{\ang{Z}}\right)
	\end{equation}
	and since \begin{equation}
		\frac{2t^2z\cdot\zeta-t|z|^2}{\ang{Z}^2}=\frac{2tz}{\ang{Z}^2}\cdot(t\zeta-z/2)
	\end{equation}
	the term in \eqref{eq:remains} is a sum of $S^{1,0}_\mathrm{par}$-multiples of $\sigma_{2,0}(tD_{z_i}-z_i/2)$, with each of the $S^{1,0}_\mathrm{par}$ coefficients vanishing on $\SR_+$ due to the cutoff factor.
	
	Thus we have shown
	\begin{equation}
		i\ang{Z}\sigma_{2,0}([A,P])=\sum_{k=0}^Nc_{jk}\sigma_{1,0}(A_k)
	\end{equation}
	for $c_{jk}\in S^{1,0}_{\mathrm{par}}$ with $\sigma_{\mathrm{base},1,0}(c_{jk})= 0$ on $\SR_+$.
	Quantising each symbol in this identity, we see that $tE_1\chi_{\mathrm{pol},-}$ satisfies the required positivity condition.
	
	Finally we compute $i\ang{Z}\sigma([B_+,P])$ using the Poisson bracket. We have
	\begin{equation}
		\begin{split}
			& i\ang{Z}\sigma([B_+,P])\\
			&=-\ang{Z}(\partial_t+2\zeta\cdot \partial_z)\left((z\cdot \zeta -|z|\sqrt{-\tau})\TN\EQ\right)\\
			&= 	-\TN\ang{Z}(\partial_t+2\zeta\cdot \partial_z)\left((z\cdot \zeta -|z|\sqrt{-\tau})\EQ\right)\\
			&=
			-\TN\ang{Z}\left(\left(2|\zeta|^2-\frac{2\sqrt{-\tau}z\cdot
				\zeta}{|z|}\right)\EQ \right.\\
			&\qquad - \left. (z\cdot \zeta
			-|z|\sqrt{-\tau})\chi'\left(\frac{|t|}{\ang{Z}}\right)\left(\frac{1+|z|^2-2tz\cdot
				\zeta}{\ang{Z}^3}\right)\right).\label{eq:ugly}
		\end{split}
	\end{equation}
	The second term in the final equation is in the form $\sum_{j=1}^n c_j(t\zeta_j-z_j/2)+c'$ with $c_j,c'\in S^{1,0}_{\mathrm{par}}$ and with every $c_j$ and $c'$ vanishing on $\SR_+$.
	The remaining term in \eqref{eq:ugly} can be written in the form
	\begin{align}
		-2\TN\EQ \ang{Z}\left(|\zeta|^2-\frac{\sqrt{-\tau}z\cdot \zeta}{|z|}\right)&= 2\TN\EQ\ang{Z}\frac{\sqrt{-\tau}}{|z|}(z\cdot \zeta -|z|\sqrt{-\tau})\\
		&- 2\TN\EQ\ang{Z}(\tau+|\zeta|^2).
	\end{align}
	Since the first term is a positive $S^{1,0}_{\mathrm{par}}$ multiple of $\sigma(B_+)$ and the second is a $S_{\mathrm{par}}^{0,1}$ multiple of $\sigma(P)$, quantising this identity we leads to the required positivity condition for $B_+$.
	
	Taking the opposite sign choices in the final two generators leads to an almost identical computation, with the different sign leading to a conclusion of $P$-negativity rather than $P$-positivity.
\end{proof}

\begin{remark}
	The notion of module regularity can be generalised to the setting where the $\mathbf{A}_j$ are only assumed to lie in $\Psip{s_j}{l_j}$ for collections of positive integers $(s_j)_{j=0}^N$ and $(l_j)_{j=0}^N$. In this setting, it can be useful to work with a \textit{reduced} version $\SM^{(\kappa)}$ of the module powers in Definition \ref{def:module.power} where the indices $\alpha$ in the generating set are restricted to those with $\mathbf{A}^\alpha\in\Psip{\kappa}{\kappa}$. We anticipate that this will be useful in obtaining multiplicative results for spaces with module regularity in a future publication. 
\end{remark}

\subsection{Density}
For later use, we prove the following density result for module regularity spaces $H_{\SM_{\pm}}^{s,\sw;\kappa,k}$. 

\begin{proposition}\label{prop:density} Suppose that the variable order $\sw$ is constant in a neighbourhood of the radial sets. Then the space $\mathcal{S}(\RR^{n+1})$ of Schwartz functions is dense in $H_{\SM_{\pm}}^{s,\sw;\kappa,k}$ for all $s \in \RR$ and $k, \kappa \in \NN$. 
\end{proposition}

\begin{proof} By microlocalizing, we can reduce to the following special cases:
\begin{itemize}
\item[(i)] Proving the same statement for constant spatial weight $\sw$;
\item[(ii)] Proving the statement for variable order but for $k = \kappa = 0$, that is, with module regularity absent. 
\end{itemize}

Indeed, near the radial sets our spatial weight $\sw$ is constant by assumption, while away from the radial sets, both the large and small modules are elliptic (see Proposition~\ref{prop:R.char}), in which case the module regularity space is microlocally identical to $\Hpar{s+k+\kappa}{\sw+ k + \kappa}$. 

In case (i), we first consider the case $s = \sw = 0$. We let $T_\epsilon$, for $\epsilon > 0$, be a family of parabolic scattering pseudodifferential operators of order $(-\infty, -\infty)$ such that the $S^{0,0}$-seminorms of $T_\epsilon$ are uniformly bounded, and $T_\epsilon \to \Id$ strongly. For example, one can take $T_\epsilon  = \Op(e^{-\epsilon (|z|^2 + t^2)} e^{-\epsilon (|\zeta|^4 + \tau^2)})$. Then it is not difficult to show that for any $A_1, \dots, A_j \in \Psipcl 1 1$, the multi-commutator $[A_1, [A_2, \dots [A_j, T_\epsilon] \dots ]]$ tends to zero strongly; we omit the proof. 
Given $u \in H_{\SM_{\pm}}^{0,0;\kappa,k}$, we define $u_j = T_{j^{-1}} u \in \mathcal{S}$. Then $u_j \to u$ in $L^2$, since $T_\epsilon \to \Id$ strongly. Moreover, for any product $A_1 \dots A_q B_1 \dots B_{q'}$ of at most $k$ elements of $\SM_\pm$ and at most $\kappa$ elements of $\SN$, we find that 
\begin{multline}
A_1 \dots A_q B_1 \dots B_{q'} u_j = A_1 \dots A_q B_1 \dots B_{q'} T_{j^{-1}} u = T_{j^{-1}}  A_1 \dots A_q B_1 \dots B_{q'} u \\ + \text{ commutator terms}.
\end{multline}
The first term on the RHS tends to $A_1 \dots A_q B_1 \dots B_{q'} u$ as $j \to \infty$. 
The  commutator factors all tend to zero strongly. We move these factors to the left, at the cost of double commutators, which we move to the left at the cost of triple commutators, and so on.  Eventually, we arrive at a sum of terms, the left factor of which is a multicommutator of module elements with $T_{j^{-1}}$ and the remaining factors are module elements. All of these multicommutator factors tend to zero strongly, and they act on a fixed function in $L^2$, using the fact that $u$ has module regularity of order $(k, \kappa)$. All terms other than the first one above therefore tend to zero in $L^2$ as $j \to \infty$. We deduce that 
$$
A_1 \dots A_q B_1 \dots B_{q'} u_j \to A_1 \dots A_q B_1 \dots B_{q'} u \text{ as } j \to \infty.
$$
We deduce that $u_j \to u$ in the topology of $H_{\SM_{\pm}}^{0,0;\kappa,k}$, proving the density in the case $s = \sw = 0$. 

For general constant $s$ and $\sw$, we choose an elliptic, invertible operator $F \in \Psipcl s {\sw}$. (To do this, we start with an elliptic operator of the form $\Op^w(f)$ where $f$ is a real elliptic symbol of order $(s, \sw)$; then $\Op^w(f)$ is formally self-adjoint and Fredholm, hence has a finite dimensional kernel, which consists of Schwartz functions due to elliptic regularity. Then $F = \Op^w(f) + \Pi$, where $\Pi$ is orthogonal projection onto the null space, is invertible, and $\Pi$ is an operator of order $(-\infty, -\infty)$, so $F \in \Psipcl s {\sw}$ as required.) Given $u \in H_{\SM_{\pm}}^{s, \sw;\kappa,k}$, we have $F u \in H_{\SM_{\pm}}^{0,0;\kappa,k}$. We choose Schwartz $u_j$ converging to $Fu$ in $H_{\SM_{\pm}}^{0,0;\kappa,k}$. Then we claim that $F^{-1} u_j$ converges to $u$ in $ H_{\SM_{\pm}}^{s, \sw;\kappa,k}$. The proof is a standard commutation argument, which we omit. This completes the proof in case (i). 

In case (ii), so now $\sw$ can be a variable order, we choose an elliptic invertible operator of order $(s, \sw)$ as above. Then given $u \in 
\Hpar s {\sw}$, $Fu$ is in $L^2$. We approximate $F u$ in $L^2$ by the Schwartz sequence $u_j = T_{j^{-1}} Fu$ as above, and then $F^{-1} u_j$ converges to $u$ in the topology of $\Hpar s {\sw}$. This proves case (ii). 
\end{proof}


\section{Fredholm estimates}\label{sec:propfred}

In this section, we show that the operator $P=D_t + \Delta_g+V$ is a Fredholm map between suitable function spaces, following closely the methodology introduced in 
\cite{VD2013}, and followed in \cite{NEH}, in which microlocal estimates, including radial points
propagation estimates, are combined to prove global Fredholm estimates. 

\subsection{Microlocal propagation estimates}\label{subsec:mic.prop.est}
Here we collect together various microlocal estimates for $P$. These are proved in the appendix for a general class of operators. Here we restate these estimates in the special case of the operator $P$ under consideration. 

We can distinguish four different estimates, each valid in a particular microlocal region. The first region is the elliptic region $\Ell(P)$. In this case we obtain an estimate without loss of derivatives. This was already stated as Proposition~\ref{prop:elliptic.estimate} but for ease of reference we restate it here. 
The second region is near the characteristic variety $\Sigma(P)$ and away from the radial sets. This is the region of principal-type propagation, and is essentially H\"ormander's original `propagation of singularities' (really propagation of regularity) estimate from \cite{Hormander:Existence}. The third region is near the radial sets. In this case, there are two estimates required, depending on whether the spatial regularity order is greater than or less than the threshold value of $-1/2$ (see the discussion in the Introduction). From the technical point of view, the significance of the threshold value is precisely  the different form that the radial point estimates necessarily take in the two cases.

The elliptic estimate, Proposition~\ref{prop:elliptic.estimate} in the particular case of our Schr\"odinger operator $P = D_t + \Delta_g + V$ takes the following form. 
\begin{prop}[Elliptic estimate]
	\label{prop:elliptic.estimate.P}
	Suppose that $Q,G\in\Psip{0}{0}(\RR^{n+1})$ are such that $P$ and $G$ are elliptic on $\WF'(Q)$, let $\sw$ be an arbitrary spacetime order, and let $M,N\in\RR$. Then if $GPu\in H_{\mathrm{par}}^{s-2,\sw}$, we have $Qu\in H^{s,\sw}$ with the estimate
	\begin{equation}\label{eq:elliptic.estimate}
		\|Qu\|_{H^{s,\sw}}\leq C_{M,N} (\|GPu\|_{H^{s-2,\sw}}+\|u\|_{H^{M,N}}).
	\end{equation}
\end{prop}

\begin{remark}\label{rem:elliptic} For non-experts in microlocal analysis, we mention that this estimate is the microlocal analogue of the standard elliptic estimate in classical PDE theory: if $q, g$ are two $C_c^\infty$ functions with $\supp q \subset \{ g > 0 \}$ and if $P$ is a differential operator of order $2$ with smooth coefficients that is elliptic on the support of $g$, then we have  for any $M \in \RR$ the estimate  (in standard Sobolev spaces)
\begin{equation}\label{eq:ellipticPDE.estimate}
		\|q u\|_{H^{s}}\leq C_M (\|g Pu\|_{H^{s-2}}+\|u\|_{H^{M}}).
	\end{equation}
	The estimate is of course only interesting when $M$ is smaller than $s$. 
We think of the symbols of $Q$ and $G$ in \eqref{eq:elliptic.estimate} as cutoff functions, analogous to $q$ and $g$ in \eqref{eq:ellipticPDE.estimate}, but on phase space rather than just on spacetime. 	
	\end{remark}

The propagation of singularities (regularity) estimate, Proposition~\ref{prop:sing}, reads as follows. Note the loss of one order of regularity in both the spatial index $\sw$ and the differential order $s$, reflecting the fact that the characteristic variety $\Sigma(P)$ meets both spacetime-infinity and fibre-infinity. 

\begin{prop}[Propagation of regularity]
	\label{prop:sing.P}
	Let $Q,Q',G\in\Psip{0}{0}$ be operators of order $(0,0)$ with $G$ elliptic on $\WF'(Q)$. Let $\sw$ be a variable spacetime order that is non-increasing in the direction of the bicharacteristic flow of $P$. 
	
	Furthermore, suppose that for every $\alpha\in\WF'(Q)\cap\Sigma(P)  $ there exists $\alpha'$ such that $Q'$ is elliptic at $\alpha'$ and there is a forward bicharacteristic curve $\gamma$ of $P$ from $\alpha'$ to $\alpha$ such that $G$ is elliptic on $\gamma$.
	
	Then if $GP u\in H_{\mathrm{par}}^{s-1,\sw+1}$ and $Q'u\in H_{\mathrm{par}}^{s,\sw}$, we have $Q u\in H_{\mathrm{par}}^{s,\sw}$ with the estimate
\begin{equation}\label{eq:sing.P}	
	\|Qu\|_{H_{\mathrm{par}}^{s,\sw}}\leq C(\|Q' u \|_{H_{\mathrm{par}}^{s,\sw}}+\|GP u\|_{H_{\mathrm{par}}^{s-1,\sw+1}}+\|u\|_{H_{\mathrm{par}}^{M,N}})
	\end{equation}
	for any $M,N\in\RR$.
\end{prop}

\begin{remark}\label{rem:pos} Figure \ref{fig:prop} illustrates the setup of Proposition~\ref{prop:sing.P}. In words, the Proposition states that regularity of the function $u$ propagates from the microsupport of $Q'$, that is $\WF'(Q')$, to the microsupport of $Q$, provided that the regularity order is not greater at $\WF'(Q)$ than at the corresponding points of $\WF'(Q')$ (we cannot inexplicably gain regularity!) and provided that $Pu$ is sufficiently regular in a microlocal neighbourhood of all the bicharacteristics that traverse between $\WF'(Q')$ and $\WF'(Q)$. 
\end{remark}

\begin{figure}[h]
	\label{fig:prop}
	\centering
	\begin{tikzpicture}[scale=0.8]
		\draw (8,0) arc (0:180:8 and 4);
		\draw (8,0) arc (360:225:2);
		\draw (-8,0) arc (180:315:2);
		\draw (-6+1.414,-1.414) arc (135:45:6.48);
		\filldraw (6,0) circle (2pt);
		\filldraw (-6,0) circle (2pt);
		\draw[-{To[scale=2]}] (-6,0) .. controls (-4,1) and (-2,1.9) .. (0,2);
		\draw (0,2) .. controls (2,1.9) and (4,1) .. (6,0);
		\draw (-6,0) circle (30pt);
		\draw (-6,-1.4) node {$\mathrm{supp}(q')$};
		\draw (6,0) circle (15pt);
		\draw (6,-0.9) node {$\mathrm{supp}(q)$};
		\draw (0,-0.1) node {$\mathrm{supp}(g)$};
		\draw (0,1.5) node {$\gamma$};
		\draw (-5.8,-0.2) node {$\alpha'$};
		\draw (6.2,-0.2) node {$\alpha$};
	\end{tikzpicture}
	\caption{The hypotheses of Proposition \ref{prop:sing.P} is that for each $\alpha\in \mathrm{supp}(q)\cap \Sigma(P)$ there is a bicharacteristic segment $\gamma$, contained within $\mathrm{supp}(g)$, connecting $\alpha'\in \mathrm{supp}(q')$ to $\alpha$. Notice that we only need this condition for $\alpha\in \Sigma(P)$, otherwise the stronger elliptic estimate is available. }
\end{figure}
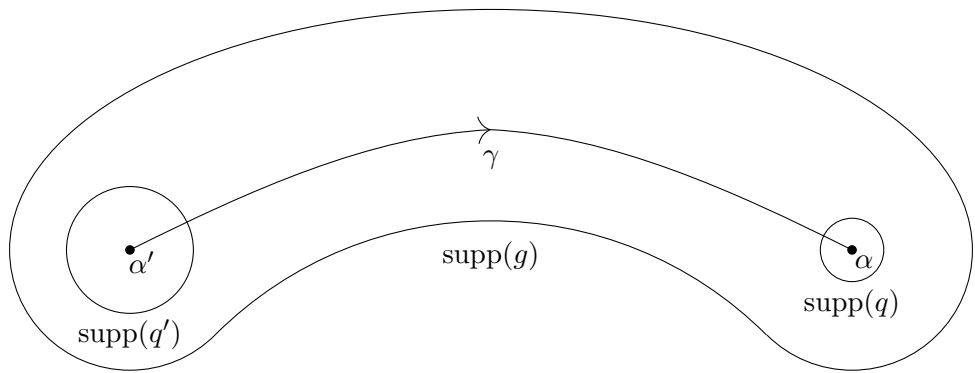

Notice that the estimate above only gives a trivial estimate if $\WF'(Q)$ meets the radial set. This is because the bicharacteristic flow is stationary on the radial sets, so we would need $Q'$ also elliptic at the radial set, which means the conclusion would be no stronger than the assumption. 

For estimates valid near the radial points, we refer to Propositions \ref{prop:above} and \ref{prop:below} in the Appendix, based on estimates due to Melrose \cite{RBMSpec} and Vasy \cite{VD2013}, adapted here to the parabolic calculus. 
Specializing to the case of the time dependent Schr\"{o}dinger operator $P=\Delta_g+D_t+V\in\Psip{2}{0}$, and iterating the results to give an arbitrary gain of regularity compared to the background regularity assumption, gives the following results.
We state them for constant orders for simplicity, as that is all that our arguments require.

\begin{prop}[Below threshold radial point estimate]
	\label{prop:belowschrod}
	Suppose $r<-\frac{1}{2}$. Assume there exists a neighbourhood $U$ of $\mathcal{R}_\pm$ and $Q', G\in\Psi_{\mathrm{par}}^{0,0} $ such that for every $\alpha\in \Sigma_P \cap U \setminus \mathcal{R}_\pm$ the bicharacteristic $\gamma$ through $\alpha$ enters $\Ell(Q')$ whilst remaining in $\Ell(G)$. Then there exists $Q\in\Psi_{\mathrm{par}}^{0,0}$ elliptic on $\mathcal{R}_\pm$ such that if $u \in H_\mathrm{par}^{M,N}$,  $Q'u\in H_\mathrm{par}^{s,r}$ and $GP u\in H_\mathrm{par}^{s-1,r+1}$, then $Qu\in H_\mathrm{par}^{s,r}$ and for all $M,N\in\RR$, there exists $C>0$ such that
	\begin{equation}\label{eq:radial.schrod.below}
		\|Qu\|_{H_\mathrm{par}^{s,r}}\leq C(\|Q'u\|_{H_\mathrm{par}^{s,r}} +\|GP u\|_{H_\mathrm{par}^{s-1,r+1}}+\|u\|_{H_\mathrm{par}^{M,N}}).
	\end{equation}
\end{prop}

\begin{prop}[Above threshold radial point estimate]
	\label{prop:aboveschrod}
	Suppose $r>r'>-\frac{1}{2}$ and $s>s'$. Assume that $G\in\Psip{0}{0}$ is elliptic at $\SR_\pm$. Then there exists $Q\in\Psip{0}{0}$ elliptic at $\SR_\pm$ such that, if $u \in H_\mathrm{par}^{M,N}$, $Gu\in H_\mathrm{par}^{s',r'}$ and $GP u\in H_\mathrm{par}^{s-1,r+1}$, then $Qu\in H_\mathrm{par}^{s,r}$ and for all $M,N\in\RR$, there exists $C>0$ such that
	\begin{equation}\label{eq:radial.schrod.above}
		\|Qu\|_{H_\mathrm{par}^{s,r}}\leq C(\|GP u\|_{H_\mathrm{par}^{s-1,r+1}}+\|Gu\|_{H_\mathrm{par}^{s',r'}}+\|u\|_{H_\mathrm{par}^{M,N}}).
	\end{equation}
\end{prop}

\begin{remark} We see that the below threshold estimate \eqref{eq:radial.schrod.below} looks the same as \eqref{eq:sing.P} but with the additional assumption that $r$ is below the threshold value of $-1/2$. On the other hand, the above threshold estimate is a bit different: we do \emph{not} need to assume that we have microlocal regularity at the same order $(s,r)$ at some other place $\WF'(Q)$, but instead, we \emph{do} need to assume a priori that we have regularity at some order $(s', r')$ where $r'$ is already above threshold. The proposition then tells us we can bootstrap this to $(s,r)$-regularity, provided that $Pu$ is suitably regular. This difference is crucial as it means that we have a starting place for proving $(s,r)$-regularity: that is, we can deduce $(s,r)$ regularity, say for a solution to $Pu = 0$,  without having to already know it somewhere else. This explains why our function spaces introduced below in \eqref{eq:Y space}, \eqref{eq:X space} impose above threshold regularity at one of the radial sets. On the other hand, to propagate regularity all the way to the other radial set the regularity needs to be below threshold at the other radial set so that Proposition~\ref{prop:belowschrod} can be applied. 
\end{remark}

\subsection{Global Fredholm estimate --- variable order case}
\label{subsec:global.fredholm}
In this subsection, we combine the estimates in the preceding subsection
into a single global estimate that will suffice to establish the
Fredholm property for $P=\Delta_g+D_t+V$ as a map between two
suitable variable order Sobolev spaces. We choose real two constants $l, m$ with $l < -1/2 < m$ --- that is, $l$ is below, and $m$ above, threshold --- and 
fix a weight function $\sw_+\in S^{0,0}_{\mathrm{par}}$ with the properties 
\begin{enumerate}[(i)]
	\item $\sw_+(z,\zeta)\in [l, m]$;
	\item $\sw_+=l$ in a neighbourhood $\SU_+$ of $\SR_+$ and $\sw_+=m$ in a neighbourhood $\SU_-$ of $\SR_-$;
	\item $\sw_+$ is nonincreasing along the bicharacteristics of $P$
\end{enumerate}
and take 
\begin{equation}\label{eq:sw+-}
	\sw_-=-1-\sw_+.
\end{equation}
In some cases, it is convenient to assume additionally that 
\begin{enumerate}[(iv)]
	\item $l > -3/2$ and $m \leq l+1$. 
\end{enumerate}

\begin{remark} Proposition~\ref{prop:radial set geom} shows that assumption (iii) is compatible with (i) and (ii). \end{remark}

We then define the variable order Sobolev spaces
\begin{equation}
	\SY^{s,\sw_{\pm}}=H_\mathrm{par}^{s,\sw_{\pm}}\label{eq:Y
		space}
\end{equation}
and
\begin{equation}
	\SX^{s,\sw_{\pm}}=\{u\in \SY^{s,\sw_\pm}:Pu\in
	H_\mathrm{par}^{s-1,\sw_\pm+1}\}\label{eq:X space}
\end{equation}

We then have the following global Fredholm estimate for $P$.
\begin{prop}
	\label{prop:fredholmest}
	For the above choice of weight functions $\sw_{\pm}$ satisfying (i) -- (iii) above,  arbitrary $s\in\RR$, $M<s$, $N<l$, and all $u\in\SX^{s,\sw_{\pm}}$ we have
	\begin{equation}
		\label{eq:fredholmest}
		\|u\|_{s,\sw_{\pm}}\leq C(\|Pu\|_{s-1,\sw_{\pm}+1}+\|u\|_{M,N})
	\end{equation}
	
\end{prop}
\begin{proof}
	We shall prove \eqref{eq:fredholmest} in the case of the weight function $\sw_+$. The proof for $\sw_-$ is essentially identical, with the roles of $\mathcal{R}_{\pm}$ swapped.
	
	We begin by choosing $Q_1,Q_2,Q_3,Q_4,G_1,G_2,G_3,G_4\in\Psip{0}{0}$ such that 
	\begin{enumerate}[(i)]
		\item $G_j$ is elliptic on $\WF'(Q_j)$,
		\item $P$ is elliptic on $\WF'(Q_1)$,
		\item $\WF'(Q_3)\subset\WF'(G_3)\subset \SU_+$,
		\item $\WF'(Q_4) \subset \WF'(G_4) \subset\SU_-$,
		\item $\WF'(Q_2)$ is disjoint from  $\SR_\pm$,
		\item Every $\alpha$ in a punctured neighbourhood of $\SR_+$ on $\Sigma_P$ lies on forward bicharacteristic $\gamma$ from a point $\alpha'\in \Ell(Q_2)$, with the bicharacteristic lying entirely in $\Ell(G_3)$,
		\item Every $\alpha\in \WF'(Q_2)\cap \Sigma_P$ lies on a forward bicharacteristic $\gamma$ from a point $\alpha'\in\Ell(Q_4)$, with the bicharacteristic lying entirely in $\Ell(G_2)$,
		\item $Q_1+Q_2+Q_3+Q_4=\mathrm{Id}$.
	\end{enumerate}
	We can now apply Proposition \ref{prop:aboveschrod} using the operators $Q_4$ and $G_4$. We can replace the spatial order $r$ with the variable weight $\sw_+$ in this estimate, as $\sw_+$ is constant (equal to $m$) in $\WF'(G_4)$. This yields the estimate
	\begin{equation}
		\label{eq:fred4}
		\|Q_4 u \|_{s,\sw_+} \leq C(\|G_4Pu\|_{s-2,\sw_++1}+\|G_4u\|_{s',r'}+\|u\|_{M,N})
	\end{equation} 
	for any $M,N,s,s'\in\RR$ and $-1/2<r'<m$.
	Similarly, we apply Proposition \ref{prop:belowschrod} to the operators $Q_3,Q_2$ to give
	\begin{equation}
		\label{eq:fred3}
		\|Q_3u\|_{s,\sw_+}\leq C(\|Q_2 u\|_{s,\sw_+}+\|G_3Pu\|_{s-1,\sw_++1}+\|u\|_{M,N}).
	\end{equation}
	Away from the radial sets, we can control $\|Q_1u\|$ and $\|Q_2u\|$ using the microlocal elliptic estimate of Proposition  \ref{prop:elliptic.estimate} and the real principal type propagation result of Proposition \ref{prop:sing} respectively. For the latter, we use $\WF'(Q_4)$ as a source of regularity, given the dynamical condition (vii).
	Consequently we have an estimate
	\begin{equation}
		\label{eq:fred2}
		\|Q_2u\|_{s,\sw_+}\leq C(\|Q_4u\|_{s,\sw_+}+\|G_2Pu\|_{s-1,\sw_++1}+\|u\|_{M,N}).
	\end{equation}
In the elliptic region, we weaken \eqref{eq:elliptic.estimate} to 	
	\begin{equation}
		\label{eq:fred1}
		\|Q_1u\|_{s,\sw_+}\leq C(\|G_1Pu\|_{s-1,\sw_+ +1}+\|u\|_{M,N})
	\end{equation}
so that the norm of $G_1 Pu$ agrees with the norms for $G_i Pu$ with $i = 2 \dots 4$.

Without loss of generality, we can assume that the constants $C$ in estimates \eqref{eq:fred4} --- \eqref{eq:fred1} are equal, and exceed $1$. Then, we 
 estimate $$\|u\|_{s,\sw_+}\leq \|Q_1u\|_{s,\sw_+}+ 2C\|Q_2u\|_{s,\sw_+}+\|Q_3u\|_{s,\sw_+}+4C^2 \|Q_4u\|_{s,\sw_+}$$ and combine the estimates in \eqref{eq:fred1}, \eqref{eq:fred2},  \eqref{eq:fred3} and \eqref{eq:fred4}. This combination allows us to absorb the $Q_2u$ and $Q_4u$ terms on the RHS by those on the LHS. 
	This gives (with a new constant $C$) 
	\begin{equation}
		\label{eq:fred5}
		\|u\|_{s,\sw_+}\leq C(\|Pu\|_{s-1,\sw_++1}+\|G_4u\|_{s',r'}+\|u\|_{M,N}).
	\end{equation}
	For $r'=(-1/2, m)$ and appropriate choices of $s'\in (M,s)$ and $\eta\in(0,1)$, Sobolev interpolation and Young's inequality then give
	\begin{align}
		\|G_4u\|_{s',r'}&\leq \|G_4u\|_{s,m}\|^{1-\eta}\|G_4 u\|^\eta_{M,N}\\
		&\leq \frac{1}{2}\|u\|_{s,m}+C\|u\|_{M,N}
	\end{align}
	for a suitable constant $C$. As $\sw_+=m$ on $\WF'(G_4)$, we can replace the constant order $m$ with the weight $\sw_+$ and absorb this term into the left-hand side of \eqref{eq:fred5}, allowing us to conclude \eqref{eq:fredholmest}.
\end{proof}

We now show, following \cite[Theorem 21.7]{Ho4} and \cite[Section 4.3]{Vasy2013},  that the estimate of Proposition  \ref{prop:fredholmest} implies that $P$ is a Fredholm map.

\begin{prop}
	\label{prop:fredholm}
	For $s\in \RR$, the map $P:\mathcal{X}^{s,\sw_{\pm}}\to\mathcal{Y}^{s-1,\sw_{\pm}+1}$ for either sign choice is a Fredholm map.
\end{prop}

\begin{proof}
	The argument is essentially identical for the two sign choices, and so we can take the positive sign without loss of generality. 
	On $\ker(P)\subset \mathcal{X}^{s,\sw_+}$, the estimate \eqref{eq:fredholmest} simplifies to 
	\begin{equation}\label{eq:kernel.est}
		\|u\|_{s,\sw_+}\leq C \|u\|_{M,N}.
	\end{equation}
	From compactness of the embedding $H^{s,\sw_+}_\mathrm{par}\subset H^{M,N}_\mathrm{par}$, it follows that the identity map restricted to $\ker(P)\subset H^{M,N}_\mathrm{par}$ is compact, and so $\ker(P)$ is finite-dimensional.
	
	Next we show that the range of $P$ is closed. To this end, we take a sequence of $u_j\in \mathcal{X}^{s,\sw_+}$ with $u_j\in \ker(P)^\perp$ and $Pu_j$ converging to some $f$ in $H^{s-1,\sw_++1}_\mathrm{par}$. Then first we observe that  $\|u_j\|_{M,N}$ is uniformly bounded. If this were not the case, then we could pass to a subsequence with $\|u_j\|_{M,N}\to\infty$ and then making the rescaling $\hat u_j=u_j/\|u_j\|_{M,N}$, an application of  \eqref{eq:fredholmest} to $\hat{u}_j$ together with the compactness of the embedding  $H^{s,\sw_+}_\mathrm{par}\subset H^{M,N}_\mathrm{par}$ allows us to deduce convergence in $H_\mathrm{par}^{s,\sw_+}$ of a subsequence $\hat{u}_j$ to a limit $v\in\ker(P)$. As $u_j\in \ker(P)^\perp$, it follows that $v=0$, which is a contradiction as we have $\|\hat{u}_j\|_{M,N}=1$ by construction.
	The boundedness of $\|u_j\|_{M,N}$ just demonstrated immediately implies boundedness of $\|u_j\|_{s,\sw_++1}$ from \eqref{eq:fredholmest}. Once more exploiting the compactness of the embedding  $H^{s,\sw_+}_\mathrm{par}\subset H^{M,N}_\mathrm{par}$, it follows that a subsequence $u_j$ is convergent in $H_\mathrm{par}^{M,N}$. Since $Pu_j$ is convergent in $H_\mathrm{par}^{s-1,\sw_++1}$, \eqref{eq:fredholmest} implies that this subsequence is convergent to some $u$ in  $\mathcal{X}^{s,\sw_+}$ with $Pu=f$ hence proving that the range of $P$ is closed.
	
	It remains to show that the cokernel of $P$ is finite-dimensional. The formal self-adjointness of $P$ gives an identification of $\coker(P)$ with the set of $v\in (H_{\mathrm{par}}^{s-1,\sw_++1})^*=H_{\mathrm{par}}^{1-s,\sw_-}$ such that $Pv=0$. Now the same argument used to establish finite dimensionality of $\ker(P)$ can be used (with $s$ replaced by $1-s$ and with the opposite sign choice for our spatial weight), provided we take $M <\min(s,1-s)$.
\end{proof}

\subsection{Module regularity estimates away from radial sets}
\label{subsec:module.positive.commutator}

The estimates in Proposition \ref{prop:elliptic.estimate} and Section \ref{subsec:positive.commutator} have analogues in the setting of Sobolev spaces with module regularity, as introduced in Definition~\ref{def:module.def.characteristic}. We state these results for the particular operator $P=\Delta_g+D_t+V$.

First we prove an analogue of Proposition \ref{prop:elliptic.estimate.P}.

\begin{prop}
	\label{prop:module.elliptic.estimate.old}
	Suppose $Q,G\in\Psip{0}{0}$ are such that $P$ and $G$ are elliptic on $\WF'(Q)$, let $\sw$ be an arbitrary variable order, and let $M,N\in\RR$. Then if $GP u\in H_\pm^{s-2,\sw;\kappa,k}$, we have $Qu\in H_\pm^{s,\sw;\kappa,k}$ with the estimate
	\begin{equation}
		\|Qu\|_{H_\pm^{s,\sw;\kappa,k}}\leq C(\|GP u\|_{H_\pm^{s-2,\sw;\kappa,k}}+\|u\|_{H^{M,N}}).
	\end{equation}

\end{prop}

\begin{proof}
	An elliptic parametrix construction allows us to write 
	\begin{equation}
		Q=CGP +R
	\end{equation}
	where $C\in\Psip{-2}{0}$ satisfies $\WF'(C)\subseteq \WF'(Q)$.
	Then for a collection $A_1,A_2,\ldots,A_{n_1}$ of elements of $\SM_\pm$ and a collection $B_1,B_2,\ldots,B_{n_2}$ of elements of $\SN$, we can compute
	\begin{align}
		\label{eq:ell.mod.commest}
		A_1\ldots A_{n_1}B_1\ldots B_{n_2}Qu&= A_1\ldots A_{n_1}B_1\ldots B_{n_2}(CGP +R)u\\
		&=CA_1\ldots A_{n_1}B_1\ldots B_{n_2} GP u + \tilde{R}u\\
		&+\sum_{j=1}^{n_1} A_1\ldots A_{j-1}[C,A_j]A_{j+1}\ldots A_{n_1}B_1\ldots B_{n_2}GP u\\
		&+\sum_{j=1}^{n_1} A_1\ldots A_{n_1}B_1\ldots B_{j-1}[C,B_j]B_{j+1}\ldots B_{n_2}GP u.
	\end{align}
	We can move the commutators to the left of the final two terms by incurring terms involving a double commutator and one fewer module generator in the product that does not lie in a commutator. Iterating this process, we obtain 
	\begin{align}\label{eq:commutators}
		A_1\ldots A_{n_1}B_1\ldots B_{n_2}Qu
		&=CA_1\ldots A_{n_1}B_1\ldots B_{n_2} GL u + \tilde{R}u\\
		&+\sum_{S} C_S(\prod_{A\notin S} A)GP u
	\end{align}
	where $S$ ranges over all nonempty subsets of $\{A_1,\ldots,A_{n_1},B_1,\ldots,B_{n_2}\} $, and $C_S$ is a multi-commutator involving only $C$ and operators from $S$.
	In particular, this means that the operator $C_S$ lies in $\Psip{-2}{0}$ and so the $H_{\mathrm{par}}^{s,\sw}$ norm of the  multi-commutator terms as well as that of the first RHS term in \eqref{eq:commutators} is controlled by  $H_{\mathrm{par}}^{s-2,\sw;\kappa,k}$.
	
	After fixing $M,N\in\RR$, we conclude
	\begin{equation}
		\|A_1\ldots A_{n_1}B_1\ldots B_{n_2}Qu\|_{H_{\mathrm{par}}^{s,\sw}}\leq C(\|GP u\|_{H_\pm^{s-2,\sw;\kappa,k}}+\|u\|_{H_{\mathrm{par}}^{M,N}}).
	\end{equation}
	Summing over all choices of $A_j$ and $B_j$ from our generating set completes the proof.
\end{proof}

We also have a module regularity version of Proposition \ref{prop:sing.P}. As we will only apply this result away from the radial set of $P$, we include this as an additional convenient assumption. 
\begin{prop}
	\label{prop:module.propagation}
	Let $Q,Q',G\in\Psip{0}{0}$ be operators of order $(0,0)$ with $G$ elliptic on $\WF'(Q)$, and such that $\WF'(Q'), \WF'(G)$ are disjoint from the radial set $\SR$.  Let $\sw$ be a variable spacetime order that is non-increasing in the direction of the bicharacteristic flow of $P$.

	Furthermore, suppose that for every $\alpha\in\WF'(Q)\cap\Sigma_P  $ there exists $\alpha'$ such that $Q'$ is elliptic at $\alpha'$ and there is a forward bicharacteristic curve $\gamma$ of $P $ from $\alpha'$ to $\alpha$ such that $G$ is elliptic on $\gamma$.
	
	Then if $GP u\in H^{s-1,\sw+1;\kappa,k}$ and $Q'u\in H^{s,\sw;\kappa,k}$, we have $Q u\in H^{s,\sw;\kappa,k}$ with the estimate
	
	\[\|Qu\|_{H_\pm^{s,\sw;\kappa,k}}\leq C(\|Q' u \|_{H_\pm^{s,\sw;\kappa,k}}+\|GP u\|_{H_\pm^{s-1,\sw+1;\kappa,k}}+\|u\|_{H_{\mathrm{par}}^{M,N}})\]
	for any $M,N\in\RR$.
\end{prop}
\begin{proof}
	From Proposition \ref{prop:R.char}, for any $B\in\Psip{0}{0}$ with $\WF'(B)\cap (\SR_+\cup \SR_-)=\emptyset$, we can use a microlocal partition of unity to write $B=B_1+\ldots+B_m$ where each $\WF'(B_j)$ is contained in $\Ell(A_j)$ for some $A_j\in \SN$. As such we have that the norms $\|B_jv\|_{H^{s,r;\kappa,k}_{\pm}}$ and $\|B_jv\|_{H_{\mathrm{par}}^{s+\kappa+k,r+\kappa+k}}$ are equivalent. Summing in $j$ we obtain eqivalence between $\|Bv\|_{H^{s,r;\kappa,k}_{\pm}}$ and $\|Bv\|_{H_{\mathrm{par}}^{s+\kappa+k,r+\kappa+k}}$. We can then directly apply Proposition \ref{prop:sing} to complete the proof, noting that the operators $Q,Q',G$ in these two  propositions enjoy this same microsupport condition.
\end{proof}

\subsection{Module regularity estimates near the radial sets}
\label{subsec:module.positive.commutator.radial}

We now adapt the results of Section \ref{subsec:positive.commutator.radial} to the module regularity spaces $H_{\SM_\pm}^{s,r;\kappa,k}$ and the particular operator $P =\Delta_g+D_t+V$.

\begin{prop}
	\label{prop:mod.belowschrod}
	Suppose $r<-1/2$. Assume that there exists a  neighbourhood $U$ of $\SR_+$  and $Q', G\in\Psip{0}{0}$ such that for every $\alpha\in\Sigma_P \cap U\setminus \SR_+$ the bicharacteristic $\gamma$ through $\alpha$ enters $\Ell(Q')$ whilst remaining in $\Ell(G)$. Then there exists $Q\in\Psip{0}{0}$ elliptic on $\SR_+$ such that if $u \in H_\mathrm{par}^{M,N}$, $Q'u\in H_+^{s,r;\kappa,k}$, and $GP u\in H^{s-1,r+1;\kappa,k}_+$, then $Qu\in H^{s,r;\kappa,k}_+$. Moreover, there exists $C>0$ such that 
	\begin{equation}\label{eq:mod.radial.schrod.below}
		\|Qu\|_{H_+^{s,r;\kappa,k}}\leq C(\|Q'u\|_{H_+^{s,r;\kappa,k}} +\|GP u\|_{H_+^{s-1,r+1;\kappa,k}}+\|u\|_{H_\mathrm{par}^{M,N}})
	\end{equation}
\end{prop}

\begin{prop}
	\label{prop:mod.aboveschrod} 
	Suppose $r>r'>-\frac{1}{2}$ and $s>s'$. Assume that $G\in\Psip{0}{0}$ is elliptic at $\SR_\pm$. Then there exists $Q\in\Psip{0}{0}$ elliptic at $\SR_\pm$ such that, if $u \in H_\mathrm{par}^{M,N}$, $Gu\in H_+^{s',r';\kappa,k}$ and $GP u\in H_+^{s-1,r+1;\kappa,k}$, then $Qu\in H_+^{s,r;\kappa,k}$ and there exists $C>0$ such that
	\begin{equation}\label{eq:mod.radial.schrod.above}
		\|Qu\|_{H_+^{s,r;\kappa,k}}\leq C(\|GP u\|_{H_+^{s-1,r+1;\kappa,k}}+\|Gu\|_{H_+^{s',r';\kappa,k}}+\|u\|_{H_{\mathrm{par}}^{M,N}}).
	\end{equation}
\end{prop}

The statements of Proposition \ref{prop:mod.belowschrod} and Proposition \ref{prop:mod.aboveschrod} also hold with $H_+,\SR_{\pm}$ replaced with $H_-,\SR_{\mp}$ with the obvious modifications to their proof.

\begin{proof}
	The proof of Proposition \ref{prop:mod.belowschrod} and Proposition \ref{prop:mod.aboveschrod} proceeds along similar lines to the proofs of Proposition \ref{prop:belowschrod} and Proposition \ref{prop:aboveschrod}, by iterative use of a positive commutator estimate.
	The commutator $i([A,P ]+(P -P ^*)A)$, as in \eqref{eq:commutatorAL}, where $A\in\Psip{2s-1}{2r+1}$ has principal symbol $a$ defined in \eqref{eq:symboldef}, is replaced by 
	\begin{equation}
		\label{eq:newcomm}
		i[A_\alpha^* AA_\alpha ,P ] + 2 \operatorname{Im} V A_\alpha^* AA_\alpha
	\end{equation}
	where $\alpha=(\alpha',\alpha'')\in \NN^{N+1}\times \NN^{N'+1}$ and
	\begin{equation}
		A_{\alpha}:=\mathbf{A}^{\alpha'}\mathbf{B}^{\alpha''}=\prod_{j=0}^N \mathbf{A}_j^{\alpha_j}\prod_{k=0}^{N'} \mathbf{B}_k^{\alpha_k'}
	\end{equation}
	is a product of the generators of $\SM_+$ that lies in $\SM_+^{\kappa} \SN^{k}$.

	We treat the addition of the $A_\alpha$ factors in an inductive manner, and suppose that the conclusions of Proposition \ref{prop:mod.belowschrod} and Proposition \ref{prop:mod.aboveschrod} hold for all $(\kappa',k')<(\kappa,k)$, that is for all pairs $(\kappa',k')\neq (\kappa,k)$ with $\kappa'\leq \kappa$ and $k'\leq k$. The case $(\kappa, k) = (0,0)$ is of course provided by Propositions \ref{prop:belowschrod} and \ref{prop:aboveschrod}. 
	
	Using \eqref{eq:gen.commutators} we obtain
	\begin{equation}
		\label{eq:module.comm}\begin{aligned}
		i[A_\alpha^* AA_{\alpha} ,P ]&=A_\alpha ^*(A(\sum_j \alpha_jC_{jj})\rho_{\mathrm{base}}A_\alpha +A_\alpha^* \rho_{\mathrm{base}}(\sum_j \alpha_jC_{jj}^*)A)A_\alpha\\
		&+ \sum_{|\beta|=|\alpha|\,\beta\neq \alpha} A_\alpha^* AC_{\alpha\beta}\rho_{\mathrm{base}}A_\beta+\sum_{|\beta|=|\alpha|,\beta\neq \alpha} A_\beta^*\rho_{\mathrm{base}} C_{\alpha\beta}^*AA_{\alpha}\\
		&+ A_\alpha^* i[A,P ] A_\alpha\\
		&+ \sum_\alpha A_\alpha^* AE_\alpha P + \sum_\alpha P {E_\alpha}^* A A_\alpha
	\end{aligned}\end{equation}
	where
	\begin{enumerate}
		\item $E_\alpha\in\SM_+^{\kappa'}\SN^{k'}$ with $(\kappa',k')<(\kappa,k)$
		\item $\sigma_{\mathrm{base},1,1}(C_{\alpha\beta})|_{\SR_+})=0$
		\item $\Re(\sigma_{\mathrm{base},1,1}(C_{jj}))|_{\SR_+} \geq 0$
	\end{enumerate}
	The first term has nonnegative symbol on $\SR_+$ from Proposition \ref{prop:positivity}, the second term has sign determined by that of $i[A,P ]$, which has symbol \eqref{eq:comm}, the terms in the third line are characteristic on $\SR_+$ by Proposition \ref{prop:positivity} and finally the remaining terms are regarded as error terms.
	The identity \eqref{eq:module.comm} is analogous to \cite[Eq. (3.23)]{NEH}.
	
	We now assume that we are in the below threshold case, that is $r<-1/2$.
	
	In order to concisely write down the contribution of the first two lines of \eqref{eq:module.comm} to the commutator estimates, we introduce a matrix of operators in $\Psip{2s}{2r}$, with rows and columns indexed by multi-indices $\alpha$ with $|\alpha'|=\kappa$ and $|\alpha''|=k$. We introduce the notation for the indexing set 	\begin{align}\label{eq:def.index.set}
		S_{\kappa,k}:&=\{\alpha=(\alpha',\alpha'')\in \NN^{N+1}\times \NN^{N'+1}:|\alpha'|=\kappa, \  |\alpha''|=k \}.
	\end{align}
	For $\alpha,\beta\in S_{\kappa,k}$, the aforementioned matrix of operators is given by
	\begin{equation}
		C_{\alpha\beta}'=\begin{cases}
			(A(\sum_j \alpha_jC_{jj})\rhob+\rhob(\sum_j \alpha_jC_{jj}^*)A	&(\alpha=\beta)\\
			AC_{\alpha\beta}\rhob+\rhob C_{\alpha\beta}^*A &(\alpha\neq \beta)
		\end{cases}.
	\end{equation}
	Now let $u'\in H_{+}^{s,l;\kappa,k}$ and take $v_\alpha=A_\alpha u'$. 
	
	We compute formally, referring  to Section~\ref{sec:fredholm} for the regularization arguments needed to justify various steps in the computation. In matrix notation, we have obtained the identity 
	\begin{equation}\begin{aligned}
		\sum_{\alpha\in S_{\kappa,k}}\ang{i[A_\alpha^*AA_\alpha,P ]u',u'}&= \ang{C'v,v}
		+\ang{(i[A,P ] \otimes \mathbb{I})v, v}\\
		&
		+ 2\Re(\sum_{\alpha\in S_{\kappa,k}}\ang{v_\alpha,AE_\alpha P u'})
	\end{aligned}\label{eq:module.comm.2}\end{equation}
	where $\mathbb{I}$ indicates the $|S_{\kappa,k}| \times |S_{\kappa,k}|$ identity matrix. Using \eqref{eq:naive.below.threshold}, we obtain 
	\begin{equation}\begin{aligned}
		\sum_{\alpha\in S_{\kappa,k}}\ang{i[A_\alpha^*AA_\alpha,P ]u',u'}&= \ang{(C' + (B_1^*B_1 - B_2^*B_2 + F + R) \otimes \mathbb{I})v, v}\\
		&
		+ 2\Re(\sum_{\alpha\in S_{\kappa,k}}\ang{v_\alpha,AE_\alpha P u'}).
	\end{aligned}\label{eq:module.comm.22}\end{equation}

	From the nonnegativity conditions on the $C_{jj},C_{\alpha\beta}$, and the strict positivity of  the symbol of $B_1$, we see that the matrix $C' + B_1^* B_1 \otimes \mathbb{I}$ is diagonal with strictly positive entries on $\SR_+$. 	As such, we may write
	\begin{equation}
		C'+ B_1^* B_1 \otimes \mathbb{I}=B^*B+\tilde{R}
	\end{equation}
	where the symbol of $B$ is a positive matrix and $\tilde{R}$ is a matrix of operators in $\Psip{2s-1}{2r-1}$. This allows us to write \eqref{eq:module.comm.22} as (dropping the $\mathbb{I}$ tensor factor for brevity)	\begin{equation}\label{eq:commest}\begin{aligned}	
		\sum_{\alpha\in S_{\kappa,k}}\ang{u',i[A_\alpha^*AA_\alpha,P ]u'}&= \|Bv\|^2-\|B_2  v\|^2 -\ang{Fv, v} + \ang{(-R + \tilde{R})v,v}  \\
		&	+ 2\Re(\sum_{\alpha\in S_{\kappa,k}}\ang{v_\alpha,AE_\alpha P u'})
	\end{aligned}\end{equation}
	We estimate 	the $\|Bv\|^2$ term by using identity \eqref{eq:commest} and bounding all the other terms that appear there.

	We first estimate the commutator term $\ang{u',i[A_\alpha^*AA_\alpha,P ],u'}$. To do this, we use the identity 
$$
\ang{i[A_\alpha^*AA_\alpha,P ]u',u'} = -2 \Im \ang{A A_\alpha u', A_\alpha P u'} - \ang{u', (P-P^*) A_\alpha A A_\alpha u'}.
$$
The second term is trivial to estimate, since the symbol of $a$ has disjoint support from that of $V - \overline{V} = P - P^*$, so this operator is order $(-\infty, -\infty)$. This term is therefore bounded by $\| u \|_{H^{M,N}_\mathrm{par}}^2$ for any $M$ and $N$. The first term is estimated using 	
	\begin{equation}
		2|\ang{AA_\alpha P u',A_\alpha u'}|=2|\ang{\rhob^{1/2}A^{1/2}A_\alpha u',\rhob^{-1/2}A^{1/2}A_\alpha P u'}|.
	\end{equation}
	Summing over $\alpha$ and applying a weighted Young inequality gives the upper bound
	for the commutator term of 
	\begin{equation}\label{eq:commest2}\begin{gathered}
		\ep\|\rhob^{1/2}A^{1/2}v\|_{H_\mathrm{par}^{1/2,0}}^2+\ep^{-1}\sum_{\alpha\in S_{\kappa,k}}\|\rhob^{-1/2}A^{1/2}A_\alpha P u'\|_{H_\mathrm{par}^{-1/2,0}}^2 + \| v \|_{H^{M,N}_\mathrm{par}}^2 .
\end{gathered}	\end{equation}
We choose $Q'' \in \Psip{0}{0}$ to be microlocally the identity on $\WF'(A)$. 
	The terms $F$, $R$ and $\tilde R$ terms in \eqref{eq:commest} are estimated  as in \eqref{key2.5} and \eqref{key2.75} giving
	\begin{equation}
	|\ang{Fv, v}| + |\ang{(-R + \tilde{R})v,v}| \leq C  
		 C \Big(\|GP v\|_{H_\mathrm{par}^{s-1,r+1}}+ \|Q''v\|_{H_\mathrm{par}^{s-1/2, r'}})+\|v\|_{H_\mathrm{par}^{M,N}}).
	\end{equation}
	The term $\|B_2v\|^2$ is estimated as in Proposition \ref{prop:below}, using the standard propagation estimate of Proposition \ref{prop:sing}. This leads to a $\| Q' v \|_{H_\mathrm{par}^{s,r}}^2$ term on the RHS. The term $\|B_1v\|^2$ can be discarded as it has the same sign as $\|B v\|^2$.

	It remains to consider the term in the last line of \eqref{eq:commest}. The weighted Young inequality gives
	\begin{align}\label{eq:commest3}
		2\sum_{\alpha\in S_{\kappa,k}}|\ang{v_\alpha,AE_\alpha P  u'}|&=2\sum_{\alpha\in S_{\kappa,k}}|\ang{\rhob ^{1/2}A^{1/2}v_\alpha,\rhob^{-1/2}A^{1/2}E_\alpha P  u'}|\\
		&\leq \ep \|\rhob^{1/2}A^{1/2}v\|^2_{H_\mathrm{par}^{1/2,0}}+\ep^{-1}\sum_{\alpha\in S_{\kappa,k}}\|\rhob^{-1/2}A^{1/2}E_\alpha P u'\|^2_{H_\mathrm{par}^{-1/2,0}}
	\end{align}
	
	We combine \eqref{eq:commest},\eqref{eq:commest2}, \eqref{eq:commest3} to bound $Bv$, in particular we have
	\begin{align}
		\|Bv\|^2 &\leq C(\|Q' v\|^2_{H_\mathrm{par}^{s,r}}+\|GP v\|^2_{H_\mathrm{par}^{s-2,r}}+\|Q''v\|_{H_\mathrm{par}^{s-1/2,r'}}^2+\|v\|_{H_\mathrm{par}^{M,N}}^2\\
		&+2\ep\|\rhob^{1/2}A^{1/2}v\|_{H_\mathrm{par}^{1/2,0}}^2\\
		&+\ep^{-1}\left(\sum_{\alpha \in S_{\kappa,k}}\|\rhob^{-1/2}A^{1/2}A_\alpha P u'\|_{H_\mathrm{par}^{-1/2,0}}^2+\sum_{\alpha \in S_{\kappa,k}}\|\rhob^{-1/2}A^{1/2}E_\alpha P u'\|_{H_\mathrm{par}^{-1/2,0}}^2\right).
	\end{align}

We now choose $Q \in \Psip{0}{0}$ to be microlocally equal to the identity near $\SR_+$ and such that $\WF'(Q)$ is contained in the elliptic set of $B$. Then for any $M,N$ there is $C$ such that 
$$
\|Qv\|_{H_\mathrm{par}^{s,r}}\leq \| B v\|_{H_\mathrm{par}^{s,r}} + \|u' \|_{H_\mathrm{par}^{M,N}}.
$$
From the definition \eqref{eq:mod.reg.sob.norm} of the norm in
	Sobolev spaces with module regularity, an estimate for
	$\sum_{\alpha\in S_{\kappa,k}}\|A_\alpha Qu'\|$ in fact gives
	an estimate on  $\|Qu'\|_{H_+^{s,r;\kappa,k}}$. On the other hand, 
	$A_\alpha Qu' = Qv + [A_\alpha, Q]u'$ and the last term is microsupported away from the radial set, so we can estimate the $H_\mathrm{par}^{s,r}$ norm of $[A_\alpha, Q]u'$ by 
\begin{equation}\label{eq:prop.sing.est}
	 \| Q' u' \|_{H_\mathrm{par}^{s,r; \kappa, k}} +  \| G P u' \|_{H_\mathrm{par}^{s-1,r+1; \kappa, k}}   + \| u' \|_{H_\mathrm{par}^{M,N}} ,
\end{equation}
	using Proposition~\ref{prop:module.propagation}. We argue similarly with the other terms. Hence we obtain the estimate expressed in terms of module regularity spaces: 
	Hence,  we obtain 
	\begin{equation} \label{eq:Qu'} \begin{aligned}
		\|Qu'\|_{H_+^{s,r;\kappa,k}}^2 &\leq C(\|Q' u'\|^2_{H_+^{s,r;\kappa,k}}+\|GP u'\|^2_{H_+^{s-2,r;\kappa,k}}+\|Q'' u'\|_{H_+^{s-1/2,r';\kappa,k}}^2+\|u'\|_{H_\mathrm{par}^{M,N}}^2\\
		&+2\ep\|\rhob^{1/2}A^{1/2}u'\|_{H_+^{1/2,0;\kappa,k}}^2\\
		+\ep^{-1}\bigg( &\sum_{\alpha\in S_{\kappa,k}} \|\rhob^{-1/2}A^{1/2}A_\alpha P u'\|_{H_\mathrm{par}^{-1/2,0}}^2+\sum_{\alpha \in S_{\kappa,k}}\|\rhob^{-1/2}A^{1/2}E_\alpha P u'\|_{H_\mathrm{par}^{-1/2,0}}^2\bigg)
	\end{aligned}\end{equation}
	The term in the second line is bounded by
	\begin{equation}
		2\ep\|Q''u'\|_{H_+^{s,r;\kappa,k}}^2 \leq 4 \ep \Big( \|Qu'\|_{H_+^{s,r;\kappa,k}}^2 + \|(Q''-Q)u'\|_{H_+^{s,r;\kappa,k}}^2 \Big)
	\end{equation}
	and the $Qu'$ term can be absorbed into the left-hand side for sufficiently small $\epsilon$, while the $(Q''-Q)u'$ term can be estimated as in \eqref{eq:prop.sing.est} as $Q''-Q$ is microsupported away from $\SR_+$. The terms in the final line are controlled by $\|GP u'\|_{H_+^{s-1,r+1;\kappa,k}}$ from the ellipticity of $G$ on $\WF'(A)$. This yields the estimate 
	\begin{multline}\label{key}
		\|Qu'\|_{H_+^{s,r;\kappa,k}}\leq C \Big(\|Q'u'\|_{H_+^{s,r;\kappa,k}} +\|GP u'\|_{H_+^{s-1,r+1;\kappa,k}}+\|Q''u'\|_{H_+^{s-1/2,r-1/2;\kappa,k}} \\ +\|u'\|_{H_\mathrm{par}^{M,N}}\Big).
	\end{multline}
Iterating the estimate as in Remark~\ref{rem:iteration}, the lower-order term is subsumed into the $\|u'\|_{H_\mathrm{par}^{M,N}}$ term and we obtain 
	\begin{equation}\label{key2}
		\|Qu'\|_{H_+^{s,r;\kappa,k}}\leq C\Big(\|Q'u'\|_{H_+^{s,r;\kappa,k}} +\|GP u'\|_{H_+^{s-1,r+1;\kappa,k}}+\|u'\|_{H_\mathrm{par}^{M,N}}).
	\end{equation}

We now consider $u$ satisfying the conditions of Proposition~\ref{prop:mod.belowschrod}. By the inductive assumption, we know that $Qu$ is in $H_+^{s,r;\kappa',k'}$ for all $(\kappa', k') < (\kappa, k)$. We now regularize $u$ by letting $u' = u'(\eta) = S_\eta u$ for each $\eta > 0$, where 
$$
S_\eta = q_L \Big( \frac{\rhob}{\rhob + \eta} \,  \frac{\rhof}{\rhof + \eta} \Big).
$$
It is easy to check that $S_\eta$ is in $\Psip{-1}{-1}$ for each $\eta > 0$, and $\Psip{0}{0}$ in a uniform sense, that is, with seminorms uniformly bounded as $\eta \to 0$. Moreover, $S_\eta$ tends to the identity operator in the strong operator topology of $\Psip{0}{0}$, and in the operator norm topology in $\Psip{-\epsilon}{-\epsilon}$ for any $\epsilon > 0$. Then $Qu'(\eta)$ is in $H_+^{s,r;\kappa,k}$ for each $\eta > 0$ and the above estimate \eqref{key2} is valid. Then we examine the behaviour of the terms as $\eta \to 0$. Let $\tilde Q$, $\tilde Q'$, $\tilde G$ satisfy the same conditions as $Q, Q'$ and $G$ but with $\WF'(\tilde Q)$ contained in the elliptic set of $Q$ and similarly for the other operators. Then the assumption that $Q'u \in H_+^{s,r;\kappa,k}$ implies that $\tilde Q' u'$ is uniformly in $H_+^{s,r;\kappa,k}$. Similarly, because $GPu$ is in $H_+^{s-1,r+1;\kappa,k}$, $\tilde GP u'$ is uniformly in $H_+^{s-1,r+1;\kappa,k}$. We deduce from \eqref{key2} (with operators $\tilde Q$, $\tilde Q'$ and $\tilde G$) that $\tilde Qu'$ is \emph{uniformly} in $H_+^{s,r;\kappa,k}$. It follows that $\tilde Qu'$ has a weak limit in $H_+^{s,r;\kappa,k}$, as well as converging strongly to $\tilde Qu$ in a weaker topology, say $H_+^{s-2,r-2;\kappa,k}$ using the inductive assumption on $u$ and the norm convergence of $S_\eta$ in $\Psip{-1}{-1}$. Now redefining $Q$ to be $\tilde Q$, it follows that $Qu$ is in $H_+^{s,r;\kappa,k}$ and satisfies the estimate \eqref{eq:mod.radial.schrod.below}.

	We now turn our attention to the above threshold case, that is Proposition \ref{prop:mod.aboveschrod}. 
	From Proposition \ref{prop:R.char}, the module $\SM_+$ is elliptic at $\SR_-$ and hence on $U_-$ for sufficiently small $U_-$. Consequently all functions in \eqref{eq:mod.radial.schrod.above} are microlocalised to regions where the norms $H_+^{s,r;\kappa,k}$ and $H_+^{s+k,r+k;\kappa,0}$ are equivalent for any $s,r\in\RR$ and so it suffices to treat the case $k=0$. 
	
	We can now run the same argument as in the below threshold case, using  \eqref{eq:naive.above.threshold} to handle the commutator $[A,P ]$. As we are now working in a neighbourhood of the source $\SR_-$, the second sign choice is applicable.
	The difference this makes to \eqref{eq:commest} is that  both the $B_1$ and $B_2$ terms in the second line will now be positive, and so the $B_2v$ term can be dropped without the need for an application of the propagation theorem Proposition \ref{prop:sing}.
	
	Note that since $k=0$, we need only consider $A_\alpha$ that are products of elements of $\SN$, and so Proposition \ref{prop:positivity} still applies to show that $\sigma_{\mathrm{base},1,1}(C_{jj})|_{\SR_-}= 0$ in this case. The rest of the proof proceeds in parallel with the below threshold case.

	The analogues of Proposition \ref{prop:mod.belowschrod} and Proposition \ref{prop:aboveschrod} in the module regularity spaces $H_-^{s,r;\kappa,k}$ and switched roles of the radial set components $\SR_\pm$ have an almost identical proof. The primary difference is that the module $\SM_-$ is $P $-negative by Proposition \ref{prop:positivity}, and so the matrices of operators $C'$ is now negative-definite. In the below-threshold argument, this leads to a change in the sign of the $\|Bv\|^2$ in \eqref{eq:commest}. However, we have also switched the roles of the source $\SR_-$ and sink $\SR_+$, giving corresponding changes to the signs of the second line of \eqref{eq:commest}, and so the proof goes through without further changes. The above-threshold argument is adapted similarly.
\end{proof}

\subsection{Global module regularity estimates}
We can combine our microlocal propagation estimates on module regularity spaces in the same way to obtain global (semi-)Fredholm estimates. 

\begin{prop}
	\label{mod.fredholm}
	(i) Constant spatial order. Fix constants $s\in\RR$ and $l\in(-3/2,-1/2)$. Then for any $k\geq 0$ and $\kappa \geq 1$, and any real numbers $M$ and $N$, we have an estimate
\begin{equation}
\label{eq:mod.fredholmest}
\|u\|_{\SX_\pm^{s,l;\kappa,k}}\leq C(\|Pu\|_{\SY_\pm^{s-1,l+1;\kappa,k}}+\|u\|_{M,N}).
\end{equation}

(ii) Variable spatial order. Let $\sw_\pm$ satisfy assumptions (i) -- (iii) and \eqref{eq:sw+-} of Section~\ref{subsec:global.fredholm}. Then for any $k\geq 0$ and $\kappa \geq 0$, and any real numbers $M$ and $N$, we have an estimate
\begin{equation}
\label{eq:mod.fredholmest.kappa.zero}
\|u\|_{s,\sw_\pm;\kappa,k}\leq C(\|Pu\|_{s-1,\sw_\pm +1;\kappa,k}+\|u\|_{M,N}).
\end{equation}
\end{prop}

\begin{proof} We combine our estimates in the same way as in the proof of Proposition~\ref{prop:fredholmest}, using Proposition \ref{prop:module.elliptic.estimate.old}, Proposition \ref{prop:module.propagation}, Proposition \ref{prop:mod.aboveschrod} and Proposition \ref{prop:mod.belowschrod} replacing Proposition \ref{prop:elliptic.estimate}, Proposition \ref{prop:sing}, Proposition \ref{prop:aboveschrod} and Proposition \ref{prop:belowschrod} respectively.
We note that in case (i), since $l$ is below threshold, we cannot apply Proposition \ref{prop:mod.aboveschrod} directly at the above-threshold radial set. However, at the above threshold radial set, the module $\SM_\pm$ is elliptic. So estimate is equivalent to the estimate obtained by increasing $l$ by $1$ and reducing $\kappa$ by $1$. This is the reason for the assumption that $\kappa \geq 1$: we must have at least one order of module regularity at the radial set $\SR_\mp$ to ensure that $u$ is above threshold there. 
\end{proof}

\section{Solvability of the time-dependent
  equation}\label{sec:solv}

\subsection{Invertibility on variable order spaces}
\label{sec:inv proof sec}
In this section, we prove Theorem \ref{thm:sob.invertible}, which we
restate and slightly extend as follows.
\begin{theorem}\label{thm:main thm restate}
Let $P$ be as in \eqref{eq:inhomogeneous equation} and \eqref{eq:Pconditions}, and assume 
$\sw_\pm$ are variable orders satisfying (i) -- (iii) in Section \ref{subsec:global.fredholm}.  Then for any $s \in
\mathbb{R}$, the mappings \eqref{eq:P.mapping} are invertible.

The inverse $\Rout$ to \eqref{eq:P.mapping} with the $+$ sign, and the
inverse $\Rin$ to \eqref{eq:P.mapping} with the $-$ sign, are defined
independently of the choices of $s$ and $\sw_\pm$, in the sense
that for $\Rout$, if $v \in \mathcal{Y}^{s - 1,
  \sw_++1} \cap \mathcal{Y}^{s' - 1, \sw_+' +1}$ for $s, s'$
and two different choices $\sw_+, \sw_+'$, both
satisfying assumptions (i) -- (iii) of Section \ref{subsec:global.fredholm}, and we denote 
by $u \in \mathcal{Y}^{s - 1,
  \sw_+}, u' \in \mathcal{Y}^{s' - 1,
  \sw_+'}$ the unique distributions with $P u = v = Pu'$, then
$u \equiv u'$.
\end{theorem}
\begin{proof}
	From the Fredholm property established in Proposition \ref{prop:fredholm} and the formal self-adjointness of $P$, it suffices to show that $\ker(P)=0$ for either sign choice in \eqref{eq:P.mapping}.
The argument is essentially identical for the two sign choices so without loss of generality we take a solution $u\in \mathcal{X}^{s,\sw_+}$ to $Pu=0$ and show that $u=0$.

First, since $\sw_+=m > -1/2$ in a neighbourhood of $\SR_-$, we have that $u$ is microlocally above threshold in a neighbourhood of $\SR_-$. An application of Proposition \ref{prop:aboveschrod} allows us to deduce that in fact $u$ is microlocally in $H_\mathrm{par}^{S,L}$ for all $S,L\in\RR$ in a neighbourhood of $\SR_-$. From Proposition \ref{prop:sing} and Proposition \ref{prop:elliptic.estimate}, it follows in fact that $u$ is microlocally in $H_\mathrm{par}^{S,L}$ for all $S,L\in \RR$ everywhere except possibly at $\SR_+$. In particular, $u$ is Schwartz in cones $\{(z,t)\in\RR^{n+1}:t \ll 0,|z|/|t| < C\}$ for arbitrary $C>0$.

We can also apply the propagation theorems for module regularity spaces. In particular, the microlocal triviality of $u$ near $\SR_-$ implies that $u\in H_+^{S,L;\mathcal{K},K}$ for all $S,\mathcal{K},K\geq 0$ microlocally near $\SR_-$, and Proposition \ref{prop:module.elliptic.estimate.old} and Proposition \ref{prop:mod.belowschrod} imply that $u\in H_+^{S,-1/2-\ep;\mathcal{K},K}$ for all $S,\mathcal{K},K\geq 0$.

Taking the spatial Fourier transform, we have 
\begin{equation}
	\label{eq:FT.def.a}
	u(z, t) =\SF^{-1}_{\zeta \to z} (e^{-it|\zeta|^2}a)
\end{equation}
for some $$a(\zeta, t) \in e^{it|\zeta|^2}H^{-1/2-\ep,-1/2-\ep}(\RR^{n+1})\subset \mathcal{S}'(\RR^{n+1})$$
where the Sobolev space is a standard weighted Sobolev space in the $(\zeta,t)$ variables (i.e. not parabolic).	

Since $D_z$ and $2tD_z - z$ and $P_0=D_t+D_z\cdot D_z$ lie in powers of the small module $\mathcal{N}$ and since 

\begin{equation}
	\label{eq:commute.data}
	P_0^\gamma D_z^\beta(2tD_z-z)^{\alpha} u=\SF^{-1}_{\zeta\to z}( e^{-it|\zeta|^2}D_t^\gamma \zeta^\beta D_\zeta^\alpha a),
\end{equation}

the condition $u\in H_+^{s,-1/2-\ep;\mathcal{K},K}$ implies that 
$$\zeta^\alpha D_\zeta^\beta D_t^\gamma a\in e^{it|\zeta|^2}H^{-1/2-\ep,-1/2-\ep} $$	
for all $\alpha,\beta,\gamma$. The Leibniz rule then implies that 
$$\zeta^\alpha t^\beta D_\zeta^\gamma D_t^\delta (\chi a)\in e^{it|\zeta|^2}H^{-1/2-\ep,-1/2-\ep}$$
for all $\alpha,\beta,\gamma,\delta$ and any smooth bump function $\chi(t)$. Hence
$\chi a\in\mathcal{S}$, and the same is true of $\chi u$ from \eqref{eq:FT.def.a}.

We now exploit the fact that our metric is Euclidean for $t\ll 0$. As $P=P_0$ for $t\ll 0$, we have $D_ta=0$ in the region $\RR_\zeta^n\times (-\infty,-T]$ for some $T$, and so the traces $a(\cdot,t)$ for $t\in (-\infty,t]$ are a fixed Schwartz function $a(\zeta)$
If we fix $z/t$ and take $t\to -\infty$, we can apply stationary phase to the convergent integral expression for the inverse Fourier transform
\[u(z,t)= (2\pi)^{-n} \int e^{i(\frac{z}{t}\cdot \zeta-|\zeta|^2)t}a(\zeta)\, d\zeta.\]
The phase is stationary at $\zeta=z/2t$ and from the condition
that $u$ is Schwartz in negative time cones, the leading term
in the stationary phase expansion must vanish. In particular
we must have $a(\zeta)=0$ for $|\zeta| \leq C/2$. As $C$ is arbitrary, we conclude
that $a$ is identically zero, and hence so is
$u(z,t)$ for $t \leq -T$. 

As $\chi u\in \mathcal{S}$ for any bump function $\chi(t)$, the norm $E(t):=\|u(\cdot,t)\|_{L^2(dg_t)}^2$ is a non-negative  differentiable function of $t$ that vanishes for $t\leq -T$. We can write
$$E(t):= \int_{\mathbb{R}^n} |u|^2 \rho \,dz  $$ for a smooth positive function $\rho(z,t)$, equal to $1$ outsize of a compact set and we can compute $E'(t)$ by differentiating under the integral because $\chi u\in \mathcal{S}$. Most terms cancel as in the standard proof of conservation of $L^2$-mass for $P_0$, and we are left with
\begin{align*}
	E'(t)&= \int |u|^2 \rho'(t) \, dz\leq CE(t)
\end{align*}  
for a global constant $C$. Hence $E=0$ identically by Gr\"{o}nwall, from which we conclude that $u=0$ identically. This establishes triviality of $\ker(P)$, and hence invertibility of \eqref{eq:P.mapping}.

Finally, we show that the values of $\Rin$ and $\Rout$ are
defined independently of $s$ and $\sw_\pm$ satisfying
the assumptions of Section \ref{subsec:global.fredholm}.  Focusing on
$\Rout$, choose any pair of pairs $s, \sw_+$ and
$s', \sw_+'$ satisfying assumptions (i) -- (iii) with constants $(l,m)$ and $(l',m')$ respectively, and let $\sw_+'' \in S^{0,0}_{\mathrm{par}}(\mathbb{R}^{n +1})$ be a
function satisfying assumptions (i) -- (iii) for some weights $(l'' , m'')$ such that $\sw_+'' \leq \min (\sw_+, \sw_+')$. 
Then $P:\mathcal{X}^{s,\sw_+''}\to
\mathcal{Y}^{s-1,\sw_+'' + 1}$ is also invertible by
the proof above, and since $\mathcal{X}^{s,\sw_+'},
\mathcal{X}^{s,\sw_+} \subset
\mathcal{X}^{s,\sw_+''}$, the uniqueness holds.

\end{proof}

\subsection{Invertibility on module regularity spaces}

We can use Theorem \ref{thm:sob.invertible} to obtain an invertibility theorem regarding $P$ as a map between Sobolev spaces with module regularity, as in \eqref{eq:Xmoddef} and \eqref{eq:mod.invertible.intro}. 

First, we record the following inclusion between module regularity
spaces and variable order spaces.
\begin{prop}
	\label{lem:mod.vs.variable}
Assume that $\sw_+$ satisfies assumptions (i) --- (iv) of Section~\ref{subsec:global.fredholm}. 	For $\kappa\geq 1$, we have the inclusion
	\begin{equation}
		\label{eq:mod.vs.variable}
		H_+^{s,l;\kappa,k}\subset H_\mathrm{par}^{s,\sw_+}
	\end{equation}
	for $l$ as in assumptions (i) and (ii). 
\end{prop}
\begin{proof}
	It suffices to establish \eqref{eq:mod.vs.variable} for $\kappa=1,k=0$. Let $u\in H_+^{s,l;1,0}$ and fix a neighbourhood $U\subset \overline{\phantom{}^{\mathrm{sc}}T^*\RR^{n+1}}$ of $\SR_+$ on which $\sw_+=l$, and form a finite cover of $\overline{\phantom{}^{\mathrm{sc}}T^*\RR^{n+1}}$ consisting of $U,U_1,\ldots U_m$, where each $U_j$ is disjoint from $\SR_+$ and such that each $U_j$ lies in $\Ell(A_j)$ for some $A_j\in \SM_+$. 
	
	We then quantise a partition of unity subordinate to this cover, and denote the microlocal cutoffs by $Q,Q_1,\ldots,Q_m\in \Psip{0}{0}$.
	Now since $u\in H_{\mathrm{par}}^{s,l}$, we have $Qu\in H_{\mathrm{par}}^{s,l}$. Since $\sw_+=l$ on $\WF'(Q)$, it follows that $Qu\in H_\mathrm{par}^{s,\sw_+}$.
	
	On the other hand, since $A_ju\in H^{s,l}_\mathrm{par}$, and $\WF'(Q_j)\subset \Ell(A_j)$, microlocal elliptic regularity implies $Q_ju\in H_\mathrm{par}^{s+1,l+1}\subset H_\mathrm{par}^{s,\sw_+}$ for each $j$, where the final containment is a consequence of $l+1\geq m = \max(\sw_+)$ using assumption (iv).
\end{proof}

\begin{theorem}
	\label{thm:mod.invertible}
	Fix $s\in\RR$ and $l\in(-3/2,-1/2)$. Let $\mathcal{X}_\pm^{s,l;\kappa,k}$ and $\mathcal{Y}_\pm^{s-1,l+1;\kappa,k}$ be as in \eqref{eq:Xmoddef} and \eqref{eq:mod.invertible.intro}. 
 Then for any $k\geq 0$ and $\kappa \geq 1$, the map
	\begin{equation}
		\label{eq:mod.invertible}
		P:\mathcal{X}_\pm^{s,l;\kappa,k}\to \mathcal{Y}_\pm^{s-1,l+1;\kappa,k}
	\end{equation}
	is a Hilbert space isomorphism.
	
	For any $k\geq 0$ and $\kappa \geq 0$, provided that $\sw_\pm$ satisfy assumptions (i) -- (iii) and \eqref{eq:sw+-}, 
	\begin{equation}
		\label{eq:mod.invertible kappa zero}
		P:\mathcal{X}_\pm^{s,\sw_\pm ; \kappa,k}\to
		\mathcal{Y}_\pm^{s-1,\sw_\pm + 1;\kappa,k}
	\end{equation}
	are Hilbert space isomorphisms.
\end{theorem}
\begin{proof}
We choose $\sw_\pm$ to satisfy assumptions (i) -- (iv) of Section~\ref{subsec:global.fredholm} with respect to $l$, which is possible since $l > -3/2$. Then Proposition \ref{lem:mod.vs.variable} gives inclusions  $H_+^{s,l;\kappa,k}\subset H_\mathrm{par}^{s,\sw_+}$ and $H_+^{s-1,l+1;\kappa,k}\subset H_\mathrm{par}^{s-1,\sw_++1}$.
	Hence by Theorem \ref{thm:sob.invertible} we have the following diagram
              \begin{equation}\label{eq:ogspin}
\begin{tikzcd}        
    	\SX_+^{s,l;\kappa,k}  \arrow[hook]{d} & 	\SY_+^{s-1,l+1;\kappa,k}  \arrow[hook]{d} \\
   \SX_\mathrm{par}^{s,\sw_+}   \arrow{r} & \SY_\mathrm{par}^{s-1,\sw_++1}
\end{tikzcd}
\end{equation}
	We now show the restriction of $P$ to $\mathcal{X}_+^{s,l;\kappa,k}$ yields an isomorphism \eqref{eq:mod.invertible} by showing it is a bounded bijection. Boundedness is immediate from the definition of these spaces, and injectivity is immediate from the injectivity of the second row of \eqref{eq:ogspin}. 
	
It remains to prove surjectivity of $P : \mathcal{X}_+^{s,l;\kappa,k} \to \mathcal{Y}_+^{s-1,l+1;\kappa,k}$. 
Let $f$ be an element of $\mathcal{Y}_+^{s-1,l+1;\kappa,k}$. We exploit the density of Schwartz functions $\mathcal{S}(\RR^{n+1})$ in $\mathcal{Y}_+^{s-1,l+1;\kappa,k}$, as shown in Proposition~\ref{prop:density}. So let $f_j$ be Schwartz functions converging to $f$ in $\mathcal{Y}_+^{s-1,l+1;\kappa,k}$. We define $u_j := P_+^{-1} f_j$. Then, according to the propagation estimates of Section~\ref{subsec:mic.prop.est}, $u_j$ is microlocally trivial away from the below-threshold radial set $\Rp$. (To see this, note that we can take the order $r$ or $\sw$ in Propositions~\ref{prop:elliptic.estimate.P}, \ref{prop:sing.P} and \ref{prop:aboveschrod} to be arbitrarily large outside any neighbourhood of the below-threshold radial set, here $\Rp$.) Moreover, we can interpret this as arbitrary module regularity away from $\Rp$, and then by Proposition~\ref{prop:mod.belowschrod}, this module regularity propagates into $\Rp$. Thus $u_j$ is in $\mathcal{X}_+^{s,l;\kappa',k'}$ for arbitrary $(\kappa', k')$. In particular, from \eqref{eq:mod.fredholmest} (taking $M \leq s$ and $N < -1/2$) and Theorem~\ref{thm:main thm restate}, we have
\begin{equation}
\| u_i - u_j \|_{\SX_+^{s, l; \kappa, k}} \to 0 \text{ as } i, j \to \infty.
\end{equation}
Thus, $u_j$ is a Cauchy sequence in $\mathcal{X}_+^{s,l;\kappa,k}$, and hence has a limit $u \in \mathcal{X}_+^{s,l;\kappa,k}$. Finally, since $P$ is continuous $\mathcal{X}_+^{s,l;\kappa,k} \to \mathcal{Y}_+^{s-1,l+1;\kappa,k}$, 
$$
Pu = P (\lim_{j \to \infty} u_j) = \lim_{j \to \infty} P u_j = \lim_{j \to \infty} f_j = f,
$$
showing that $P$ is surjective on module regularity spaces.


The second statement follows via similar reasoning.

\begin{remark} The corresponding proof of invertibility on module regularity spaces in \cite{NEH} has a gap. In \cite[Proof of Theorem 2.4]{NEH}, the analogue of estimate \eqref{eq:mod.fredholmest}, that is, \cite[Equation (3.31)]{NEH}, is asserted without first establishing a priori that $u$ is in the appropriate space $\SX_+^{s, \ell; \kappa, k}$. The gap may be filled by arguing as above, that is, using the density of Schwartz functions in this space and then considering a Cauchy sequency of Schwartz functions converging to $Pu$. The authors thank Yilin Ma for bringing this gap to our attention. 
\end{remark}

\end{proof}

      \subsection{The final state problem for Schwartz
        data}\label{sec:final state}  Let
      $f$ lie in the Schwartz space $\mathcal{S}(\mathbb{R}^n)$, and
      the  operator 
\begin{equation}
  \label{eq:poisson}
  \mathcal{P}_0(f) = (2\pi)^{-n} \int e^{-it|\zeta|^2} e^{i z \cdot \zeta} f(\zeta) d\zeta =
  \left( \mathcal{F}^{-1}_{\zeta \to z} (e^{-it|\zeta|^2} f(\zeta)) \right) (z, t).
\end{equation}
This gives the unique solution to $P_0 u = (D_t + \Delta_0) u = 0$ whose
incoming and outgoing data are $f$, meaning
\begin{equation}
  \label{eq:asymptotics free}
\lim_{t \to \pm \infty} (4\pi it)^{n/2}  e^{- i t|\zeta|^2}
\mathcal{P}_0f(2t \zeta, t) \equiv \lim_{\substack{t \to \pm \infty \\  z/2t \to  \zeta}} (4\pi it)^{n/2}  e^{- i |z|^2/4t}
\mathcal{P}_0f(z, t) = f(\zeta),
\end{equation}
as follows easily from the stationary phase lemma applied to the $\zeta$ integral in \eqref{eq:poisson}.  For this reason we
refer to $\mathcal{P}_0$ as the ``free Poisson'' operator. It
is also the operator which solves the free Schr\"odinger equation for
initial data $\hat f$.

We now define the Poisson operators for the perturbed operator $P = D_t + \Delta_g + V$. 
\begin{defn}\label{def:pert pois}
Let $P$ be as in the Introduction.  Define the Poisson operators $\Poim, \Poip$ 
by
\begin{equation}
  \label{eq:4}
  \begin{aligned}
  \Poim f &= \mathcal{P}_0 f - \Rout P \mathcal{P}_0 f   = \Big( \mathcal{P}_0 - \Rout (P - P_0) \mathcal{P}_0 \Big) f, \\
   \Poip f &= \mathcal{P}_0 f - \Rin P \mathcal{P}_0 f   = \Big( \mathcal{P}_0 - \Rin (P - P_0) \mathcal{P}_0 \Big) f,
\end{aligned} \end{equation}
where $\Rout$, resp. $\Rin$ are the outgoing, resp. incoming
propagators for $P$ (see Theorem~\ref{thm:main thm restate}).
\end{defn}

\begin{prop}\label{prop:final state schwartz}
The operator $\Poip$ solves the final state problem for $f \in
\mathcal{S}(\mathbb{R}^n)$, meaning $P \Poip f = 0$ and
$$
\lim_{\substack{t \to \infty , \  z/2t \to  \zeta}} (4\pi it)^{n/2}  e^{- i |z|^2/4t}
\mathcal{P}_+f(z, t) = f(\zeta)
$$
\end{prop}
\begin{proof}
  This is a consequence of Theorem \ref{prop:true scattering map}
  below, but can be seen directly for Schwartz data quite easily
  since.  Indeed, for $f \in \mathcal{S}(\mathbb{R}^n)$,
  $P \Poip f = (P - P_0) \mathcal{P}_0 f - (P - P_0)\mathcal{P}_0 f =
  0$. In the region $|z|/t < c, t > 0$, we have
  $u_- := \Rin (P - P_0) \mathcal{P}_0 f$ is Schwartz. This follows from the propositions in Section~\ref{subsec:mic.prop.est}. In detail, we know that $u_-$ is above threshold at $\SR_+$ as it is in the image of $\Rin$, so we can take $s$ and $r$ as large as we like in \eqref{eq:radial.schrod.above}, giving microlocal regularity of any order in a neighbourhood of $\SR_+$. Using Proposition~\ref{prop:sing.P} this then propagates to $\Sigma(P) \setminus \SR_-$, while microlocal regularity in the elliptic region is immediate from Proposition~\ref{prop:elliptic.estimate.P}.  
  Therefore, $\Poip f$ is a solution to the equation which agrees with 
  $\mathcal{P}_0 f$ to infinite order in $|z|/t < c, t > 0$, for arbitrary $c$, and thus
  also satisfies \eqref{eq:asymptotics free}.
\end{proof}


\section{Poisson operator and scattering map}
\label{sec:scat mat}
\subsection{Mapping properties of the free Poisson operator}
We will now discuss finer mapping
properties of the Poisson operator and scattering operator.

Note that, by the density of Schwartz functions in the space of tempered
distributions, it is easy to see that $\mathcal{P}_0$ extends from
$\mathcal{S}(\mathbb{R}^n)$ to a mapping from
$\mathcal{S}'(\mathbb{R}^{n})$ to $\ker(P_0) \cap
\mathcal{S}'(\mathbb{R}^{n + 1})$.

Throughout this section, we assume that $\sw_\pm$ satisfy (i) -- (iii)
at the beginning of Section~\ref{subsec:global.fredholm}, with $l = -1/2 - \epsilon$ and $m = -1/2 + \epsilon$ for some small $\epsilon > 0$, as well as \eqref{eq:sw+-}. In addition, we assume that  both $\sw_\pm$ are equal to $-1/2$ on $\Sigma(P)$ outside small neighbourhoods of the radial sets. We then define 
%
\begin{equation}\begin{gathered}
\swmin = \min(\sw_+, \sw_-),  \\
\swmax = \max(\sw_+, \sw_-)
\end{gathered}\end{equation}  
and note that $\swmin + \swmax = -1$ due to \eqref{eq:sw+-}.

We begin with an identity that we will find useful on several
occasions. To state it, we choose microlocal cutoffs $Q_-$ and $Q_+$
such that $Q_- + Q_+ = \Id$, and so that $Q_-$ is microlocally equal
to the identity in a neighbourhood of $\Radm$ and microlocally trivial
in a neighbourhood of $\Radp$ (and, consequently, vice versa for
$Q_+$).


\begin{lemma}\label{lem:commidentity} Let $Q_-, Q_+$ be a microlocal partition as described above. Then for any $u \in \mathcal{S}'(\RR^{n+1})$ satisfying $Pu = 0$, we have 
\begin{equation}\label{Geq:commidentity7}
u = (P_+^{-1} - P_-^{-1}) [P, Q_+] u.
\end{equation}
\end{lemma}

\begin{proof} 
We observe that $Q_+u$ is microlocally trivial near $\R_-$. We can therefore find a variable order $\mathsf{t}_+$ satisfying (i) -- (iii) of Section~\ref{subsec:global.fredholm} and a real $s$ such that $Q_+u \in H^{s, \mathsf{t}_+}_{\mathrm{par}}$. By Theorem~\ref{thm:main thm restate}, $P_+^{-1}$ is a left inverse to $P$ on this space, so we have 
$$
Q_+ u = P_+^{-1} P Q_+ u.
$$
Similarly, we have 
$$
Q_- u = P_-^{-1} P Q_- u.
$$
Since $Pu = 0$ we have $P Q_+ u = [P, Q_+] u$ and $P Q_- u = [P, Q_-] u$. As we have $Q_+ + Q_- = \Id$, we find $[P, Q_+] = - [P, Q_-]$ and so we obtain
$$
u = Q_+ u + Q_- u = P_+^{-1} [P, Q_+] u + P_-^{-1} [P, Q_-] u = (P_+^{-1} - P_-^{-1})[P, Q_+] u.
$$
\end{proof}

Due to our assumptions on $\sw_\pm$, we can impose the additional assumption that 
\begin{equation}\label{eq:mpavo}
\sw_\pm = -1/2 \text{ on }\WF'([P, Q_+]).
\end{equation}
 We will then say that  $Q_-, Q_+$ is a \emph{microlocal partition adapted to the variable orders $\sw_\pm$}.

The following lemma establishes a fundamental mapping property of
the free Poisson operator $\mathcal{P}_0$ on variable order spaces. 

\begin{lemma}\label{lem:freePoissoniso}
The mapping
  \begin{equation}
    \label{eq:poisson basic bound}
    \mathcal{P}_0 \colon  L^2(\mathbb{R}^n_\xi) \longrightarrow H^{1/2,
     \swmin}_{\mathrm{par}}(\mathbb{R}^{n + 1}) \cap \ker \left( D_t
      + \Delta_0 \right)
  \end{equation}
  is a bounded isomorphism.
  \end{lemma}
  \begin{proof}
    First note that
\begin{equation}
\mathcal{P}_0 \mathcal{P}_0^* = i (2\pi)^{-n} \big( {(P_0)}_+^{-1} - {(P_0)}_-^{-1} \big),
\label{eq:PP0}
\end{equation}
as can be verified by explicit computation. In fact, both represent the Fourier multiplier $(2\pi)^{1-n} \delta(\tau + |\xi|^2)$.  For any $s \in \mathbb{R}$,
$$
{(P_0)}_+^{-1} - {(P_0)}_-^{-1}  \colon H^{s-1,  \swmax + 1}_{\mathrm{par}}(\mathbb{R}^{n + 1})
\longrightarrow H^{s,  \sw_{+}}_{\mathrm{par}}(\mathbb{R}^{n + 1}) +
H^{s,  \sw_{-}}_{\mathrm{par}}(\mathbb{R}^{n + 1}) \subset H^{s,
  \swmin}_{\mathrm{par}}(\mathbb{R}^{n + 1}).
$$
We apply a $TT^*$ argument to this bounded mapping, 
for which we require the range to be contained in the dual of the
domain, i.e.\
$$
H^{s,
  \swmin}_{\mathrm{par}}(\mathbb{R}^{n + 1})  \subset (H^{s-1,  \swmax +
  1}_{\mathrm{par}}(\mathbb{R}^{n + 1}))^*.
$$
Choosing $s=1/2$, then since $-\swmax - 1 = \swmin$, we see that the desired containment holds, and thus $\mathcal{P}_0^*$ maps $H^{-1/2,
     \swmax + 1}_{\mathrm{par}}(\mathbb{R}^{n + 1})$ into $L^2(\RR^n)$. Dually, we conclude that $\mathcal{P}_0$ maps $L^2(\RR^n)$ into $H^{1/2,\swmin}_{\mathrm{par}}(\mathbb{R}^{n + 1})$.

The operator $\mathcal{P}_0$ is obviously injective, as the restriction to $t=0$ is the inverse Fourier transform.  So it remains only to show
that it is surjective.  Thus, let $u \in H^{1/2,
  \swmin}_{\mathrm{par}}(\mathbb{R}^{n + 1}) \cap \ker \left( D_t
      + \Delta_0 \right)$. We employ a microlocal partition adapted to the $\sw_\pm$ and combine \eqref{Geq:commidentity7} (for the free operator $P_0$) and \eqref{eq:PP0} to obtain 
$$
u = -i (2\pi)^{n} \mathcal{P}_0 \mathcal{P}_0^* [P_0, Q_+] u. 
$$
We notice that $[P_0, Q_+] u$ is in $H^{-1/2,
  1/2}_{\mathrm{par}}(\RR^{n+1})$ using \eqref{eq:mpavo}, and thus is
contained in $H^{-1/2, \swmax + 1}_{\mathrm{par}}(\RR^{n+1})$ also by \eqref{eq:mpavo}. Thus, $f :=
-i(2\pi)^n\mathcal{P}_0^* [P_0, Q_+] u$ is in $L^2$ using the mapping property of $\mathcal{P}_0^*$ just proved. It follows that $u = \mathcal{P}_0 f$ where $f \in L^2$, proving the surjectivity. 
%
%
%
%
%
  \end{proof}
  
Recall the small module $\mathcal{N}$ defined in Definition \ref{def:module.def.characteristic}. 
Consider the generators
\begin{equation}\label{eq:small gens}
  \begin{gathered}
\Id, \quad z_k D_{z_j} - z_j D_{z_k} , \quad 2 tD_{z_j} - z_j, \quad
D_{z_j}, \quad E \langle z, t
\rangle (D_t +
\Delta_0 ).  
  \end{gathered}
\end{equation}
with $E \in \Psi^{-1, 0}$ globally elliptic.

We have the commutation relations of these generators with the Poisson operator $\mathcal{P}_0$:
\begin{equation}
  \label{eq:2}
  \begin{split}
  \Id \mathcal{P}_0 f &=  \mathcal{P}_0 (\Id f) \\
(z_j D_{z_l} - z_l D_{z_j})\mathcal{P}_0 f &=  \mathcal{P}_0 ((\xi_j D_{\xi_l} - \xi_lD_{\xi_j})f) \\
(2 tD_{z_j} - z_j) \mathcal{P}_0 f &= \mathcal{P}_0 (D_{\xi_j} f)  \\
P_0  \mathcal{P}_0 &\equiv 0 \\
D_{z_j} \mathcal{P}_0f &=    \mathcal{P}_0 (\xi_j f)
  \end{split}
\end{equation}

The Poisson operator thus intertwines the action of these generators with the following operators on $\RR^n_{\xi}$:
\begin{equation}
\Id, \quad \xi_j D_{\xi_l} - \xi_lD_{\xi_j}, \quad D_{\xi_j}, \quad \xi_j.
\label{eq:Nhatgen}\end{equation}
It is trivial to check that these operators are in $\Psisc^{1,1}(\RR^n)$ and are closed under commutators. 
They therefore generate a module which we denote $\Nhat$. We let $\Nhatgen$ denote the finite set of generators in \eqref{eq:Nhatgen}. 
(We remark here that we replaced the generator $E_1$ of $\mathcal{N}$ with the $D_{z_j}$ in \eqref{eq:small gens} for convenience. The reason for doing so is that, if we take $E_1 = (1 + D_t^2 + \Delta_z^2)^{1/4}$, then this is intertwined with $(1 + 2 |\xi|^4)^{1/4}$. We find it more convenient to replace this with factors $\xi_j$ which lead to the same module $\Nhat$.)

This leads to the definition of spaces of incoming/outgoing data $\Hdata^k(\RR^n_\xi)$ that will be suitable domain spaces for the free Poisson operator viewed as mapping into module regularity spaces. 
\begin{defn}
For $k \in \NN$, we define the Hilbert space $\Hdata^k(\RR^n_\xi)$ by 
\begin{equation}
\mathcal{W}^k(\mathbb{R}^n_\xi) = \{ f \in L^2(\mathbb{R}^n_\xi, d\xi)
\mid A_1 \dots A_j  f \in L^2(\mathbb{R}^n_\xi, d\xi) \, \forall \,
A_i \in \Nhatgen, 1 \leq i \leq j \leq k \}. \label{eq:Hdata}
\end{equation}
The norm in this Hilbert space is defined by 
$$
\| f \|^2 = \sum \| A_1 \dots A_j f \|_2^2,
$$
where the sum is over all $j$-tuples $(A_1, \dots, A_j)$ of generators for $0 \leq j \leq k$. (when $j=0$ this is of course just the $L^2$ norm of $f$.)

For $k \in \NN$, we define the spaces of negative order by 
$$
\mathcal{W}^{-k}(\mathbb{R}^n_\xi) = \{f =  \sum_{\substack{A_1, \dots, A_j \in \Nhatgen \\ j \leq k}} A_1 \dots A_j f_{A_1,  \dots,  A_j} \mid f_{A_1,  \dots,  A_j} \in L^2(\mathbb{R}^n_\xi, d\xi) \}.
$$
The squared norm of $f$ in this Hilbert space is the infimum of 
$$
\sum_{A_1 \dots A_j} \| f_{A_1,  \dots,  A_j} \|_2^2
$$
over all representations of $f$ in this form,   where $(A_1, \dots, A_j)$ are distinct $j$-tuples of elements of $\Nhatgen$ with $1 \leq j \leq k$. 
\end{defn}

Standard considerations show that $\Hdata^{-k}(\RR^n)$ is the dual space of $\Hdata^k(\RR^n)$. Recalling that $H^{s,
     \sw; k}_{\mathcal{N}}(\mathbb{R}^{n + 1})$ denotes the module regularity space of order $k$ with respect to $\mathcal{N}$, we show 
  \begin{proposition}\label{prop:freePoissonbounds} 
For $k \in \NN$, the mapping
    \begin{equation}
    \label{eq:poisson better bound}
    \mathcal{P}_0 \colon  
    \mathcal{W}^k(\mathbb{R}^n_\xi) \longrightarrow H^{1/2,
     \sw_{min}; k}_{\mathcal{N}}(\mathbb{R}^{n + 1}) \cap \ker \left( D_t
      + \Delta_0 \right)
  \end{equation}
    is a bounded isomorphism. 
    
Moreover, if $Q \in \Psip{0}{0}(\RR^{n+1})$ is such that the intersection of its microsupport with the characteristic variety is contained in the set where $\sw_{min} = -1/2$ (which implies that it is microsupported away from the radial sets), then for all integers $k \in \ZZ$ (positive or negative) we have 
\begin{equation}\label{eq:QP0}
Q \mathcal{P}_0 \colon  
    \mathcal{W}^k(\mathbb{R}^n_\xi) \longrightarrow H^{k+1/2, k-1/2}_{\mathrm{par}}(\mathbb{R}^{n + 1}) 
  \end{equation}
  is bounded. 
  \end{proposition}

\begin{proof} The first statement follows immediately from
  Lemma~\ref{lem:freePoissoniso} and commutation identities
  \eqref{eq:2}. The second statement for $k\ge 0$ follows from the
  first and the observation that the module $\mathcal{N}$ is elliptic
  on the microsupport of $Q$, so $k$ orders of module regularity gains
  us $k$ in both the differential and spacetime orders of
  regularity. 
For  $k >0$, by definition of the space $\Hdata^{-k}(\RR^n)$ it suffices to consider $f$ of the form $f = \sum A_1 \dots A_j f'$ where $A_i \in \Nhatgen$, $j \leq k$ and $f' \in L^2$. Using the commutation properties, $\mathcal{P}_0 f$ is equal to a sum of up to $k$ module elements applied to $\mathcal{P}_0 f'$, which we know lies in the space $H^{1/2,\sw_{min}}_{\mathrm{par}}(\mathbb{R}^{n + 1})$. Since these module elements are order $(1,1)$, we find that $Q \mathcal{P}_0 f$ is in the space $H^{1/2 - k,-1/2 - k}_{\mathrm{par}}(\mathbb{R}^{n + 1})$. 
\end{proof}

\subsection{Perturbed Poisson operator}
It is essentially tautological that there is a scattering map
  $$
S_0 \colon \mathcal{W}^k(\mathbb{R}^n_\xi) \lra 
    \mathcal{W}^k(\mathbb{R}^n_\xi)
    $$
    extending the map taking $f \in
    \mathcal{S}(\mathbb{R}^n_\xi)$ to $S_0(f)(\xi') = \lim_{\xi' = z/(2t) , \ t \to  \infty}  e^{- i |z|^2/4t}
    \mathcal{P}_0f.$
   Indeed, we have just seen that this map is the identity; it
can be thought of as the map taking the incoming data of solutions produced by
$\mathcal{P}_0$ to their outgoing data, though these data are equal
for the free equation.  

We now study the forward and advanced propagators $\Poim, \Poip$ for
the perturbed operator $P$ from Definition \ref{def:pert pois}.  Directly
from the definition we see that they $\Poim$ and $\Poip$ extend to the
space of tempered distributions, in particular they are defined on $\Hdata^k$ for any
integer $k$.  We then have an analogue (in fact, a slight strengthening) of Proposition~\ref{prop:freePoissonbounds} for the perturbed Poisson operator. The following proposition is the same as the first part of Theorem~\ref{thm:Poissonbounds-intro}. 

  \begin{proposition}\label{prop:Poissonbounds} 
 For $k \in \NN$, the range of the Poisson operator $\Poip$ on $\mathcal{W}^k(\RR^n)$ is precisely 
\begin{equation}\label{eq:Pkpos}
\{ u \in \mathcal{X}_+^{1/2, \sw_+; k,0}(\RR^{n+1}) + \mathcal{X}_-^{1/2, \sw_-; k,0}(\RR^{n+1}) \mid Pu = 0 \}, 
\end{equation}
i.e. that is, those elements of $\mathcal{X}^{1/2, \sw_+} + \mathcal{X}^{1/2, \sw_-}$ in the kernel of $P$ having module regularity of order $k$. %

For $k \leq -1$, the range of $\Poip$ on $\mathcal{W}^k(\RR^n)$ is precisely 
\begin{equation}\label{eq:Pkneg}
\{ u \in H_{\mathrm{par}}^{k+1/2, k-1/2}(\RR^{n+1}) \mid Pu = 0 \}. 
\end{equation}

In either case, we characterise the range of $\Poip$ on
$\mathcal{W}^k(\RR^n)$ as those elements of the null space of $P$ that
are microlocally in $H_{\mathrm{par}}^{k+1/2, k-1/2}(\RR^{n+1})$ on
$\Sigma(P) \setminus \mathcal{R}$. That is, provided $Q \in
\Psip{0}{0}$ is microsupported away from the radial sets, the map
\begin{equation}\label{eq:QP}
Q \Poip \colon  
    \mathcal{W}^k(\mathbb{R}^n_\xi) \longrightarrow H^{k+1/2, k-1/2}_{\mathrm{par}}(\mathbb{R}^{n + 1}) 
  \end{equation}
  is bounded.

%
%
%
  
  \end{proposition}
 
\begin{remark} Analogous statements hold for $\Poim$. \end{remark}
  
\begin{proof} 
We first prove the statement \eqref{eq:QP}. This will be 
deduced from \eqref{eq:poisson better bound}, the identity \eqref{eq:4} relating the free and perturbed Poisson operator, and mapping properties of the resolvent.

Let $f \in \mathcal{W}^k(\mathbb{R}^n)$. By \eqref{eq:4} we have
\begin{equation}\label{eq:Poipf}
\Poip f = \Poi_0 f - \Rin P \Poi_0 f,
\end{equation}
and we have already shown the required regularity for the $\Poi_0 f$ term in \eqref{eq:QP0}. We now consider the other term on the RHS. 

Since $P - P_0$ is compactly supported in spacetime, one can choose a
$G  \in \Psi^{0,0}_{\mathrm{par}}$ which is supported near spacetime infinity and microsupported near $\mathcal{R}_+$ such
that $G(P - P_0) \equiv 0$.  Thus $ G P \mathcal{P}_0
f = G (P - P_0) \mathcal{P}_0
f \equiv 0$.  Since $u' = \Rin  P \mathcal{P}_0 f$ is above
threshold near $\mathcal{R}_+$, the estimate in Proposition
\ref{prop:aboveschrod} applies to $u'$, and for any $Q$ with $\WF' Q  \subset
\Ell(G)$ we have that for any $K, S, M, N \in \mathbb{R}$ and $r'$ above threshold, 
\begin{multline*}
\|Q \Rin P \mathcal{P}_0 f\|_{H_\mathrm{par}^{S,K}} \\
\leq C\Big(\|G P \Rin P\mathcal{P}_0 f\|_{H_\mathrm{par}^{S - 1,K+1}}
+  \|G  \Rin P\mathcal{P}_0 f\|_{H_\mathrm{par}^{s',r'}}+\|\Rout P \mathcal{P}_0 f\|_{H_\mathrm{par}^{M,N}}\Big) \\
= C\Big(  \|G  \Rin P\mathcal{P}_0 f\|_{H_\mathrm{par}^{s',r'}}+\|\Rin P \mathcal{P}_0 f\|_{H_\mathrm{par}^{M,N}}\Big) \phantom{33}
\end{multline*}
where the last line follows since $G P \Rin P\mathcal{P}_0 f = G
P\mathcal{P}_0 f \equiv 0$. We see that on the microsupport of $Q$,
the second term on the RHS in \eqref{eq:Poipf} is microlocally trivial, and therefore $\Poip f$ satisfies the required microlocal regularity at least close to $\Radp$. 
Because $u = \Poip f$ satisfies $Pu = 0$, we can apply propagation of regularity, that is Theorem~\ref{prop:sing.P}, to deduce this regularity everywhere on $\Sigma(P) \setminus \mathcal{R}$.

We next show that the range of $\Poip$ is included in \eqref{eq:Pkpos}, when $k \geq 0$. Let $u = \Poip f$. We choose a microlocal partition adapted to $\sw_+$ and apply \eqref{Geq:commidentity7}. Notice that $[P, Q_+]$ has order $(1,-1)$ so the term $[P, Q_+] u$ is in $H_{\mathrm{par}}^{k-1/2, k+1/2}(\RR^{n+1})$ using \eqref{eq:QP}. This clearly belongs to both $\Y_+^{-1/2, \sw_+ +1; k, 0}$ and $\Y_+^{-1/2, \sw_- +1; k, 0}$ since both $\sw_\pm$ are equal to $-1/2$ on the microsupport of $[P, Q_+]$. Applying Theorem~\ref{thm:mod.invertible}, we find that $u$ lies in \eqref{eq:Pkpos}.  

To show that the range of $\Poip$ for $k < 0$ lies in \eqref{eq:Pkneg}, 
we start the same way: we use \eqref{Geq:commidentity7} again, and the fact that the term $[P, Q_+] u$ is in $H_{\mathrm{par}}^{k-1/2, k+1/2}(\RR^{n+1})$ using \eqref{eq:QP}. In this case we apply the resolvent mapping property on variable order spaces, Theorem~\ref{thm:main thm restate}, with a judicious choice of $\sw_\pm$. Namely, for the operator $P_+^{-1}$, we choose $\sw_+$ chosen to be equal to $k-1/2$ on all bicharacteristic segments between $\WF'([P, Q_+])$ and $\Radp$, and for $R_-$, we make a similar choice for $\sw_-$ (with $\Radm$ replacing $\Radp$). The sum of the two spaces $H_{\mathrm{par}}^{k+1/2, \sw_+}(\RR^{n+1}) + H_{\mathrm{par}}^{k+1/2, \sw_-}(\RR^{n+1})$ is then equal to \eqref{eq:Pkneg}. 

It remains to show that the Poisson operator $\Poip$ maps surjectively to the spaces in \eqref{eq:Pkpos} and \eqref{eq:Pkneg}. This is postponed until after Proposition~\ref{prop:PPidentity}. 
%
%
%
%
%
%
%
\end{proof}

\begin{remark}
There is an apparent problem with \eqref{eq:Poipf}: it looks at first sight as though $\Poip f$ is less regular (in the differential sense) than $\mathcal{P}_0 f$ since applying $P$ loses two orders of regularity and $\Rin$ gains back only one order. We circumvent this difficulty by using the compact support of $P - P_0$ and propagation of regularity. If $P - P_0$ were not compactly supported --- even if it decayed quite rapidly at spacetime infinity --- this argument could not be used. Instead, we would need to use an approximate Poisson operator adapted to $P$, similarly to what is done for the Helmholtz equation in \cite{MZ96}. The point is that the free Poisson operator has the `wrong phase function', adapted to $P_0$ not $P$, and only if they agree (at least at the principal symbol level) near spacetime infinity 
can we effectively treat the Poisson operator $\Poip$ as a perturbation of the free Poisson operator. 
\end{remark}

\begin{remark} It is interesting that one can bootstrap from small module regularity to full module regularity. This arises because all the modules are elliptic on $\Sigma(P) \setminus \mathcal{R}$, so small module regularity and large module regularity are equivalent there.
\end{remark}

\begin{corollary}\label{cor:Schwartz}
$\mathcal{P}^*$ maps $\mathcal{S}(\RR^{n+1})$ to $\mathcal{S}(\RR^n)$. 
\end{corollary}

\begin{proof}
Dualizing \eqref{eq:Pkneg}, we find that  $\mathcal{P}^*$ maps $H^{k-1/2, k+1/2}_{\mathrm{par}}(\RR^{n+1})$ to 
$\Hdata^{k}(\RR^n)$ for $k \geq 1$. Taking the intersection over all such $k$ yields the corollary.
\end{proof}

We now prove an analogue of Proposition 3.4 of \cite{NLSM}. 

\begin{proposition}\label{prop:limits}
Let $v \in H^{-1,\sw_{max} + 1; k}_{par, \mathcal{N}}(\RR^{n+1})$ with $k \geq 2$ and $\epsilon > 0$. Then $u_+ := \Rout v$ is such that the limits 
\begin{equation}\label{eq:tlimit}
\mathcal{L}_+ u_+(\zeta) := \lim_{t \to +\infty} (4\pi it)^{n/2} e^{-it|\zeta|^2} u_+(2t \zeta, t) 
\end{equation} 
and
\begin{equation}\label{eq:zerolimit}
\mathcal{L}_- u_+(\zeta) := \lim_{t \to -\infty} (4\pi it)^{n/2} e^{-it|\zeta|^2} u_+(2t \zeta, t) 
\end{equation}
exist in $\ang{\zeta}^{1/2 + \epsilon}\Hdata^{k-2}(\RR^n_\zeta)$, with the limit \eqref{eq:zerolimit} identically zero. 

Moreover, we have estimates 
\begin{equation}\begin{gathered}
\| \mathcal{L}_+ u_+ \|_{\ang{\cdot}^{1/2 + \epsilon}\Hdata^{k-2}} \leq C \| v \|_{H^{1/2, \sw_{max} + 1; 1, k}_{par, +}},  \\
\| t^{n/2} e^{-it|\zeta|^2/4} u_+(t \zeta, t) - \mathcal{L}_+ u_+ \|_{\ang{\cdot}^{1/2 + \epsilon}\Hdata^{k-2}} = O(t^{-\epsilon'}), \quad t \to \infty
\end{gathered}\end{equation}
for $\epsilon'$ sufficiently small. A similar statement is true for $u_- := \Rin v$, with a zero limit $\mathcal{L}_+ u_-$ as $t \to +\infty$ and a (potentially) nonzero limit $\mathcal{L}_- u_-$ as $t \to -\infty$. 
\end{proposition}

\begin{proof}
We prove the statement only for $u_+$ as the proof for $u_-$ is
essentially the same with the incoming and outgoing radial sets
switched.  Define
$$
\tilde u(\zeta, t) = (4\pi i t)^{n/2} e^{-it|\zeta|^2} u_+(2t \zeta, t).
$$
We will compute the partial derivative of $\tilde u(\zeta, t)$, i.e.\
with $\zeta$ fixed, which we denote by $D_t \rvert_{\zeta}$ to avoid
confusion with the partial derivative with respect to $t$ with $z$ fixed.

Then, using $\zeta = z/(2t)$ and $D_t = P - D_z \cdot D_z$, we can write
\begin{equation}\begin{gathered}
D_t \rvert_{\zeta} \tilde u(\zeta, t) = 
 (4\pi it)^{n/2} e^{-it|\zeta|^2}  
 \Big( - \frac{in}{2t} - |\zeta|^2 + (\frac{z}{t} \cdot D_z + D_t) \Big) u_+(2t \zeta, t)  \\
 = (4\pi it)^{n/2} e^{-it|\zeta|^2} \Big( P - t^{-2}  (t D_z -
 \frac{z}{2} ) \cdot (t D_z - \frac{z}{2}) \Big) u_+(2t \zeta, t) \\
  = (4\pi it)^{n/2} e^{-it|\zeta|^2}\left( v(2t\zeta,t)  - \left( t^{-2}  (t D_z -
 \frac{z}{2} ) \cdot (t D_z - \frac{z}{2}) \right) u_+(2t \zeta, t) \right)
 \end{gathered}\label{Dtzeta}\end{equation}
 We recognize the factor $t D_{z_i} - z_i/2$ as an element of the
 module $\mathcal{N}$. So we are now in a similar position to the
 proof in \cite{NEH}. By Theorem \ref{thm:mod.invertible}, $u_+
 \in H^{0, \sw_+; 0, 2}$, which allows us to conclude that the
 second term in the parenthesis has
 $$
\left((t D_z -
 \frac{z}{2} ) \cdot (t D_z - \frac{z}{2}) \right) u_+(2t \zeta, t)  \in H^{0, \sw_+; 0, 0}
 $$
 Moreover, from the assumption on $v$, using: (1) 
 $\Psi^{1,0}_{\mathrm{par}} \subset \mathcal{N}$, (2) that $\sw_{max} = -1/2 + \epsilon$ near the radial sets, and (3) that $\mathcal{N}$ is
 elliptic away from the radial sets, we have
 $$
v \in H^{-1,\sw_{max} + 1; 2}_{par, \mathcal{N}}(\RR^{n+1})
\subset H^{0,\sw_{max} + 1; 1}_{par, \mathcal{N}}(\RR^{n+1})
\subset H^{0, 1/2 + \epsilon}_{\mathrm{par}} = \ang{(t,z)}^{-1/2 - \epsilon}
L^2(dtdz)
 $$
 Thus, on $0 < T < t$
\begin{multline}
D_t \tilde u (\zeta, t)  \rvert_{\zeta} \in   \ang{(t,z)}^{-1/2 - \epsilon}t^{n/2}
L^2(dtdz) + \ang{(t,z)}^{1/2 + \epsilon}t^{n/2} \ang{t}^{-2} L^2(dt dz) \\ \subset t^{-1/2 - \epsilon} \ang{\zeta}^{1/2 + \epsilon} L^2(dt d\zeta).
\label{eq:Dtuliesin}\end{multline}
 This is in $$t^{-\epsilon'} L^1([T, \infty)_t; \ang{\zeta}^{1/2 +
  \epsilon}L^2(\RR^n_\zeta))$$ for  $0 < \epsilon' < \epsilon$. We can
thus integrate the $t$-derivative, for fixed $\zeta$, of $\tilde u$,
viewed as a function of $t$ with values in $\ang{\zeta}^{1/2 +
  \epsilon}L^2(\RR^n_\zeta)$, out to infinity, showing that the limit
exists. Moreover, the convergence is at a rate of $O(t^{-\epsilon'})$
as we see by integrating $D_t \rvert_{\zeta} \tilde u$ back from $t = \infty$. 

Now to prove the result for $k > 2$, we observe that applying module element $2t D_{z_i} - z_i$ to $u_+$ is equivalent to applying $D_{\zeta_i}$ to $\tilde u$. In the same way, applying module element $z_i D_{z_j} - z_j D_{z_i}$ to $u_+$ is equivalent to applying $\zeta_i D_{\zeta_j} - \zeta_j D_{\zeta_i}$ to $\tilde u$. Moreover, since $2t D_{z_i} - z_i$ and $D_{z_i}$ are both module elements, it follows that multiplication by $z_i = (2t D_{z_i} - z_i) - 2t \cdot D_{z_i}$ maps $u$ to $\ang{t}H^{0, -1/2 - \epsilon}_{\mathrm{par}}$. Since multiplication by $\ang{t}$ commutes with both $2t D_{z_i} - z_i$ and $D_{z_i}$, we can iterate this argument, showing that for  $|\alpha| \leq k$, and $t$ large, multiplication by $z^\alpha$ maps to $\ang{t}^{|\alpha|} H^{0, -1/2 - \epsilon}_{\mathrm{par}}$, and hence, multiplication by $\zeta^\alpha$ maps to $H^{0, -1/2 - \epsilon}_{\mathrm{par}}$ for $t \geq 1$. This means that we can apply compositions of up to $k$ generators of $\Nhat$ to $\tilde u$, improving  \eqref{eq:Dtuliesin} to 
\begin{equation}
D_t \rvert_{\zeta}  \tilde u \in  t^{-1/2 - \epsilon} \ang{\zeta}^{1/2 + \epsilon} \Hdata^{k-2}(dt d\zeta).
\label{eq:Dtuliesink}\end{equation}
Repeating the argument above shows that the limit \eqref{eq:tlimit} exists in the $\ang{\zeta}^{1/2 + \epsilon}\Hdata^{k-2}$ topology, and thus the limit lies in $\ang{\zeta}^{1/2 + \epsilon}\Hdata^{k-2}(\RR^n_\zeta)$. 

Exactly the same argument shows that the limit of \eqref{eq:zerolimit} exists as $t \to -\infty$. However, because $u$ is obtained by applying the \emph{outgoing} propagator $\Rout$ to $v$, $u$ is above threshold (that is, in $H^{0, -1/2 + \epsilon}_{\mathrm{par}}$) microlocally away from $\Radp$. So in any region of the form $t \leq -1$, $|\zeta| \leq R$, we have $u' \in t^{1/2 - \epsilon} L^2(dt dz)$, which amounts to $\tilde u$ being in $t^{1/2 - \epsilon} L^2(dt d\zeta)$ for $(t, \zeta) \in (-\infty, -1] \times B(0, R)$. This is incompatible with the existence of the limit $\mathcal{L}_- u$ in \eqref{eq:zerolimit} unless the limit function vanishes in $B(0, R)$. Since $R$ is arbitrary this proves that $\mathcal{L}_- u = 0$.

 \end{proof}
 
 \begin{corollary}
 The incoming and outgoing Poisson operators $\Poim, \Poip$ satisfy 
 \begin{equation}
  \mathcal{L}_- \mathcal{P}_- = \mathcal{L}_+ \mathcal{P}_+  = \Id.
  \end{equation}
  \end{corollary}
  
  \begin{proof} This is an immediate consequence of \eqref{eq:asymptotics free} and the Proposition~\ref{prop:limits}. 
  \end{proof}

We will also need the following operator identity. 

\begin{proposition}\label{prop:PPidentity}
We have the identity
\begin{equation}\label{eq:PPidentity}
\mathcal{P}_+ \mathcal{P}_+^* = \mathcal{P}_- \mathcal{P}_-^* = i (2\pi)^{-n} \big( \Rout - \Rin \big).
\end{equation}
\end{proposition}

\begin{proof}
We first note that we have already shown that each of these three operators is bounded from $H^{-1/2, \sw_{max} + 1}_{\mathrm{par}}(\RR^{n+1})$ to $H^{1/2, \sw_{min}}_{\mathrm{par}}(\RR^{n+1})$. So to prove the equality, we need only consider the action on a dense subspace, such as $\mathcal{S}(\RR^{n+1})$. 

We thus consider $v \in \mathcal{S}(\RR^{n+1})$ and let $u_+ = \Rout v$, $u_- = \Rin v$ and $u = u_+ - u_-$, which therefore solves $Pu = 0$. 
Then $u_\pm$ are in $H^{1/2, \sw_{min}; N}_{par, \pm}$ for every $N$, i.e. they have infinite module regularity. Following the proof of  Proposition~\ref{prop:limits}, therefore, $u$ has an expansion as $t \to \pm \infty$, which we write in the form 
\begin{equation}\begin{gathered}
u = (4\pi it)^{-n/2} e^{i|z|^2/4t} \Big( f_\pm(\zeta) + O_{\Hdata^N_{\zeta}}(t^{- \epsilon'}) \Big) , \quad \zeta = \frac{z}{2t} \in B(0, R), \quad t \to \pm \infty,  \\
\in |z|^{1/2 + \epsilon} \ang{\zeta}^{-N} L^2(dz \, dt), \quad |\zeta| \geq R, 
\end{gathered}\label{eq:asympt}\end{equation}
for $t$ large and $N \in \NN$ and some $\epsilon' > 0$. Here the $f_\pm(\zeta)$ are in $\ang{\zeta}^{1/2 + \epsilon}\Hdata^k(\RR^n_\zeta)$ for arbitrary $k$, hence Schwartz functions, and do not depend on the choice of $R$. The second estimate follows from the fact, shown in the proof of Proposition~\ref{prop:limits}, that module regularity of order $k$ allows us to gain a factor $(\ang{t} /|z|)^k$. 

We now prove a pairing identity for two functions $u_1$ and $u_2$ such that $Pu_i$ in $\mathcal{S}(\RR^{n+1})$. Each function has such an expansion as $t \to \pm \infty$, with data $f_i^+$ as $t \to +\infty$ and $f_i^-$ as $t \to -\infty$. 
Then we claim that the following identity holds: 
\begin{equation}
\int_{\RR^{n+1}} ( u_1 \overline{P u_2} - P u_1 \overline{u_2} ) dg(t) dt = -i (2\pi)^{-n} \int_{\RR^n_\zeta} \big( f_1^+(\zeta) \overline{f_2^+(\zeta)} - f_1^-(\zeta) \overline{f_2^-(\zeta)} \big) \, d\zeta.
\label{eq:pairing}\end{equation}

To prove \eqref{eq:pairing} we write the LHS as 
$$
\lim_{R \to \infty} \int \chi\big( \frac{t}{R} \big) \chi\big( \frac{|z|}{R^2} \big) ( u_1 \overline{P u_2} - P u_1 \overline{u_2} ) \, dg(t) \, dt
$$
where $\chi \in C_c^\infty(\RR)$ is identically equal to $1$ near zero. 
Since $P$ is formally self-adjoint, we can shift derivatives from one factor of $P$ to the other and the only non-cancelling terms will be those where a derivative hits one of the $\chi$ factors. When derivatives hit the second $\chi$ factor, the integrand is supported where $|t| \leq C |z|^{1/2}$ and $|z| \sim R^2$. Using the second line of \eqref{eq:asympt} we see that the limit as $R \to \infty$ of these terms is zero. 

So now consider when a $t$-derivative hits the first $\chi$ factor. These occur when $|t| \sim R$, where $g(t)$ is the flat metric for $R$ sufficiently large. So we can write these terms as 
$$
\lim_{R \to \infty} -i \int \chi'\big( \frac{t}{R} \big) \chi\big( \frac{|z|}{R^2} \big)  u_1 \overline{ u_2}  \, dz \, \frac{dt}{R}.
$$
We substitute the first line of \eqref{eq:asympt} for $u_i$ and notice that only the leading order asymptotic of each contributes to the limit. Moreover, the second $\chi$ factor is $1$ on a ball $B(0, cR)$ in the $\zeta$ variable.  We further write $dz = (2t)^{-n} d\zeta$, change integration variable to $\zeta$ and we obtain \eqref{eq:pairing}, since 
$$
\int_0^\infty \chi'(s) ds = 1, \quad \int_{-\infty}^0 \chi'(s) ds = -1. 
$$

We now apply the pairing formula \eqref{eq:pairing} with $u_1 = \Poim a$, for some $a \in \mathcal{S}(\RR^n)$, and with $u_2 = u_-$ as defined above, i.e.\ $u_2 = \Rin v$. Then we find that $P u_1 = 0$, $f_1^- = a$ and $f_2^+ = 0$, so we obtain 
$$
\ang{\Poim a, v}_{L^2(\RR^{n+1})} = i (2\pi)^{-n}\ang{a, f_2^-}_{L^2(\RR^n)},
$$
from which follows $\Poim^* v = i (2\pi)^{-n} f_2^-$. We may also express $u = \Poim f_2^-$ as $u$ is the unique solution to $Pu = 0$ with incoming data $f_2^-$. We conclude that 
$$
\Poim \Poim^* v = i (2\pi)^{-n} \Poim f_2^- = i (2\pi)^{-n} u = i (2\pi)^{-n} \big( \Rout - \Rin \big) v  ,
$$
proving the proposition. 
\end{proof}

\begin{proof}[Completion of the proof of Proposition~\ref{prop:Poissonbounds}]
We need to show that $\Poip$ with domain $\mathcal{W}^k(\mathbb{R}^n)$
surjects onto \eqref{eq:Pkpos} when $k \ge 0$ and \eqref{eq:Pkneg}
when $k \le -1$. To do this, we let $u$ be an element of \eqref{eq:Pkpos} and use \eqref{Geq:commidentity7} together with \eqref{eq:PPidentity} to write
\begin{equation}
u = -i (2\pi)^n \Poip \Poip^* [P, Q_+] u.
\end{equation}
Thus, it clearly suffices to show that $\Poip^*[P, Q_+] u$ is in $\mathcal{W}^k$. Observe that, since $[P, Q_+]$ has order $(1, -1)$, and is microsupported away from the radial sets, $[P, Q_+] u$ is in $H_{\mathrm{par}}^{k-1/2, k+1/2}(\RR^{n+1})$. We choose a microlocal cutoff $Q$ that is microsupported away from the radial sets, and microlocally the identity on $\WF'([P, Q_+])$, and write 
\begin{equation}
u = -i (2\pi)^n \Big( \Poip \Poip^* Q^* [P, Q_+] u + \Poip \Poip^* (\Id - Q^*) [P, Q_+] u \Big).
\end{equation}
Using the dual of \eqref{eq:QP} (with $k$ replaced by $-k$), we find that $\Poip^* Q^* [P, Q_+] u$ is in $\mathcal{W}^k$ as required. On the other hand, by the microlocal support assumptions, $(\Id - Q^*) [P, Q_+] u$ is in $\mathcal{S}(\RR^{n+1})$, and by Corollary~\ref{cor:Schwartz}, we have $\Poip^* (\Id - Q^*) [P, Q_+] u$ is in $\mathcal{S}(\RR^n)$, which is even better. 

Surjectivity for \eqref{eq:Pkneg} is proved in exactly the same way. 
\end{proof}

\begin{proof}[Proof of Theorem~\ref{thm:Poissonbounds-intro}] The combination of Propositions~\ref{prop:Poissonbounds} and \ref{prop:limits} establishes\break Theorem~\ref{thm:Poissonbounds-intro}.
\end{proof}

\subsection{Scattering map}

The following theorem is a slight elaboration of Theorem~\ref{thm:sc-intro}. 

\begin{theorem} \label{prop:true scattering map}
  The scattering map, defined initially for $f \in
  \Hdata^k(\mathbb{R}^n)$ with $k \ge 2$ by
  $$
S(f) = \lim_{\substack{t \to \infty , \  z/2t \to  \zeta}} (4\pi it)^{n/2}  e^{- i |z|^2/4t}
\mathcal{P}_+f(z, t) \in \langle \zeta \rangle^\epsilon  \Hdata^{k - 2}(\mathbb{R}^n)
  $$
  in fact satisfies that
  \begin{equation}
    \label{eq:S mapping true}
    S \colon \Hdata^k(\mathbb{R}^n) \longrightarrow \Hdata^k(\mathbb{R}^n)
  \end{equation}
  is bounded, and extends naturally to a continuous mapping for all
  $k \in \mathbb{Z}$.
\end{theorem}


\begin{proof} 
Let $f \in \Hdata^k(\RR^n)$, and let $u = \Poim f$. As in the previous proof, we use \eqref{Geq:commidentity7} together with \eqref{eq:PPidentity} to express 
\begin{equation}\begin{gathered}
u =  -i (2\pi)^n \Poim \Poim^* [P, Q_+] u = -i (2\pi)^n \Poip \Poip^* [P, Q_+] u .
\end{gathered}\end{equation}

It follows that the outgoing data for $u$ is $\mathcal{L}_+ u = -i (2\pi)^n \Poip^* [P, Q_+] u$. That is, the scattering map $S$ has the form (similarly to \cite[Proposition 5.1]{Vasy1998}) 
\begin{equation}\label{eq:S}
S = -i (2\pi)^n \Poip^* [P, Q_+] \Poim.
\end{equation}

Next, Corollary~\ref{cor:Schwartz} shows that, up to an operator mapping $\Hdata^k(\RR^n)$ to $\mathcal{S}(\RR^n)$ for any $k \in \NN$, this is equal to 
\begin{equation}
-i (2\pi)^n \Poip^*  Q \, [P, Q_+] \, Q' \Poim
\end{equation}
where $Q, Q'$ are operators of order $(0,0)$ that are microlocally the identity on $\WF'([P, A])$ and microlocally trivial in a neighbourhood of the radial sets. 

The mapping property then follows from
Proposition~\ref{prop:Poissonbounds}. In fact, $Q' \Poim$ maps
$\Hdata^k$ to $H^{k+1/2, k-1/2}_{\mathrm{par}}$; the operator $[P, Q_+]$ is order
$(1, -1)$ so maps to $H^{k-1/2, k+1/2}_{\mathrm{par}}$; and then the adjoint
operator $\Poip^* Q$ maps to $\Hdata^k$. Thus $S$ (defined this way)
extends to a map on all $\Hdata^k$. This completes the proof. 

\end{proof}

  


\section{Appendix: Global propagation of regularity and Fredholm estimates}
\label{sec:fredholm}
In this appendix we prove the general propagation of regularity results on the full phase space $\overline{T}^*_{\mathrm{par}}\mathbb{R}^{n
	+ 1}$.  

We treat more general
parabolic differential operators $\newP \in \Psipcl m l$ and establish two microlocal estimates controlling $u$ in
terms of itself and $\newP u$. 
The first of these estimates, Proposition~\ref{prop:sing}, is microlocalised to the subset of the characteristic variety $\Sigma(\newP )$ where the renormalised Hamiltonian vector field $H^{m,l}$ of $L$ is nonvanishing, and amounts to the standard propagation of regularity theorem in the parabolic setting.
The second estimate is microlocalised to a neighbourhood of the radial set $\SR$ where $H^{m,l}$ vanishes, and in this region we employ radial set estimates as introduced by Melrose \cite{RBMSpec}.

\subsection{Positive commutator estimates away from radial sets}
\label{subsec:positive.commutator}

In the subset of $\Char(\newP )$ where the renormalised Hamiltonian vector field $H^{m,l}$ is nonvanishing, we have positive commutator estimates analogous to H\"{o}rmander's propagation theorem for real principal type operators.

\begin{prop}
	\label{prop:sing}
	Let $\newP \in\Psipcl{k}{l}$ be an operator of real principal type and let $Q,Q',G\in\Psip{0}{0}$ with $G$ elliptic on $\WF'(Q)$.
	
	Assume that $\mathsf{m}$ is a variable spacetime order that is nonincreasing in the direction of the bicharacteristic flow of $L$. Furthermore, suppose that for every $\alpha\in\WF'(Q)\cap\Sigma_\newP  $ there exists $\alpha'$ such that $Q'$ is elliptic at $\alpha'$ and there is a forward bicharacteristic curve $\gamma$ of $\newP $ from $\alpha'$ to $\alpha$ such that $G$ is elliptic on $\gamma$.
	
	Then if $G\newP u\in H_{\mathrm{par}}^{s-k+1,\mathsf{m} -l +1}$ and $Q'u\in H_{\mathrm{par}}^{s,\mathsf{m}}$, we have $Q u\in H_{\mathrm{par}}^{s,\mathsf{m}}$ with the estimate
	
	\[\|Qu\|_{H_{\mathrm{par}}^{s,\mathsf{m}}}\leq C(\|Q' u \|_{H_{\mathrm{par}}^{s,\mathsf{m}}}+\|G\newP u\|_{H_{\mathrm{par}}^{s-k+1,\mathsf{m}-l+1}}+\|u\|_{H_{\mathrm{par}}^{M,N}})\]
	for any $M,N\in\RR$.
\end{prop}

The proof of Proposition \ref{prop:sing} is essentially identical to that of  \cite[Theorem~5.4]{grenoble}. The only difference is that the boundary defining function for the fiber compactification $\rho_\mathrm{fib}$ is in our setting given by the quasi-homogeneous $(1 + R^4)^{-1/4}$, as in \eqref{eq:fiber weight function}, and so the Sobolev spaces in the theorem become the parabolic Sobolev spaces considered in this paper.

\begin{remark}
	When $\WF'(Q),\WF'(Q')$ and $\WF'(G)$ are disjoint from the
	corner of the compactified phase space $\overline{T}^*_{\mathrm{par}}\mathbb{R}^{n + 1}$, Proposition \ref{prop:sing} can be
	obtained as a direct consequence of standard propagation
	theorems valid on the boundary faces $\partial_{\mathrm{base}} \overline{T}^*_{\mathrm{par}}\mathbb{R}^{n + 1}$ and
	$\partial_{\mathrm{fib}} \overline{T}^*_{\mathrm{par}}\mathbb{R}^{n + 1}$. The original propagation theorem due to
	H\"{o}r\-mander \cite{Hormander:Existence}, as first used by
	Melrose \cite{RBMSpec} in the scattering setting, is valid on the interior of 
	$\partial_{\mathrm{base}} \overline{T}^*_{\mathrm{par}}\mathbb{R}^{n + 1}$ where the anisotropic nature of
	$\Psi_{\mathrm{par}}$ plays no role. On the other hand, the
	propagation theorem is proven for $\Psi_{\mathrm{par}}$ and
	general  anisotropic pseudodifferential calculi on the interior of the boundary face $\partial_{\mathrm{fib}} \overline{T}^*_{\mathrm{par}}\mathbb{R}^{n + 1}$ in \cite{lascar}.
\end{remark}

\subsection{Positive commutator estimates near the radial sets}
\label{subsec:positive.commutator.radial}

In this section we write down the microlocal propagation estimates for a general operator $\newP \in \Psipcl{m}{l}(\RR^{n+1})$ with real principal symbol near its radial set $\SR_\newP$ (see Definition~\ref{def:radial set}).

The proofs of these results are essentially identical to the positive commutator estimates in the standard scattering calculus, and we follow the presentation of \cite{grenoble},\cite{hintz},\cite{RBMSpec}.

We are interested in the study of propagation estimates near a radial set $\SR_\newP$ extending into the corner of the compactified phase space $\overline{T}^*_{\mathrm{par}}\mathbb{R}^{n + 1}$. 
We shall consider the case of a radial set $\SR_\newP \subset \partial_{\mathrm{base}} \overline{T}^*_{\mathrm{par}}\mathbb{R}^{n + 1}$ that meets the other boundary face $\partial_{\mathrm{base}} \overline{T}^*_{\mathrm{fib}}\mathbb{R}^{n + 1}$ transversally.

Let $\newP \in\Psi^{m,l}_\mathrm{par, cl}$ have real principal symbol $p$, and let 
$\hat p$ denote $\rhob^m \rhof^l p$ (where $\rhob$ and $\rhof$ are defined by \eqref{rho-defn} and \eqref{rhofib}), which by the assumption of classicality of $\newP$ is a smooth function on $\overline{T}^*_{\mathrm{par}}\mathbb{R}^{n + 1}$. 

We recall from Section~\ref{sec:scatcalc} that the vector field 
	\begin{equation}\label{eq:renormalised}
		H^{m,l}:= \rhof^{1-m} \rhob^{1-l} H_\newP 
	\end{equation}
	extends to a smooth vector field tangent to the boundary faces of $\overline{T^*\RR^n}$. By definition, $\SR_\newP$ is the subset of $\Sigma(\newP)$ where $H^{m,l}$ vanishes. 
We assume that $\SR_\newP$ is a smooth submanifold of $\Sigma(\newP)$ of codimension $k$ that meets $\partial_{\mathrm{fib}} \overline{T}^*_{\mathrm{fib}}\mathbb{R}^{n + 1}$ transversally. As we have seen, this assumption holds for our specific operator $P$ with $k = n$. As a submanifold of $\partial_{\mathrm{base}} \overline{T}^*_{\mathrm{fib}}\mathbb{R}^{n + 1}$, $\SR_\newP$ can be characterized by 
\begin{equation}\label{eq:quad.non.degen}
\SR_\newP = \{ 	\rho_{\mathcal{R}_L}=0, \  \hat p =0 \},
\end{equation}
where $\rho_{\mathcal{R}_L}$ is a quadratic defining function for $\SR_\newP$ as a submanifold of $\Sigma(\newP)$, that is, $\rhor=\sum_{j=1}^k \rho_{\mathcal{R}_L,j}^2$, with the  $\rho_{\mathcal{R}_L,j}$ a collection of $k$ smooth functions vanishing on $\mathcal{R}$ with linearly independent differentials. 

Our main assumption on $\SR_\newP$ will be that it is either a \emph{source} or a \emph{sink} of the Hamilton vector field. To explain what this means concisely, we adopt the convention in the remainder of this Section that in all future occurrences of $\pm$ and $\mp$, the top sign choice corresponds to the sink situation and the bottom corresponds to the source situation. 
By a source/sink we mean that $H^{m,l} \rho_{\mathcal{R}_L}$ is nonnegative/nonpositive in a neighbourhood of the radial set, and that $H^{m,l} \rhob$ is `strictly' nonnegative/nonpositive in the sense that 
\begin{equation}
	\label{eq:beta}
	H^{m,l} \rhob=\mp\beta \rhob
\end{equation}
where $\beta\in\SC^\infty(\overline{T^*\RR^{n+1}})$ is strictly positive on $\SR_\newP$. 
As a consequence of this first condition, taking $\phi\in\mathcal{C}_c^\infty([0,\infty))$, equal to $1$ near $0$ and decreasing, then 
\begin{equation}
	\label{eq:phi1}
	\phi_1:=\sqrt{\pm H^{m,l}\big(\phi( \rho_{\mathcal{R}} + \rhob^2)\big)}
\end{equation}
is non-negative, smooth, and vanishes near the radial set. 
Furthermore, we require that 
\begin{equation}
	\label{eq:beta1}
	H^{m,l} \rhof=\mp 2\beta\beta_1 \rhof
\end{equation}
with $\beta_1\in\SC^\infty(\overline{T^*\RR^{n+1}})$ vanishing on $\mathcal{R}$. 
We introduce $\tilde{p},q\in\mathcal{C}^\infty(\partial (\overline{T^*\RR^{n+1}}))$ by setting
\begin{equation}
	\label{eq:subprincipal}
	\tilde{p}=\sigma_{\mathrm{par},m-1,l-1}(\frac{1}{2i}(\newP -\newP ^*))=\pm \beta q\rhof^{1-m}\rhob^{1-l}
\end{equation}
Finally we choose $\phi_0\in\mathcal{C}_c^\infty(\RR)$ identically $1$ near, and supported sufficently close to $0$. We shall use $\phi_0\circ \hat{p}$ to localize near the characteristic set $\Sigma_\newP =\hat p^{-1}(0)$.

In the specific case of the operator $P = D_t + \Delta_g + V$, we have seen via explicit calculation in Section~\ref{subsec:radial} that its rescaled Hamilton vector field is a sink near $\SR_+$ and a source near $\SR_-$. Moreover, those calculations shows that  $\beta = 2$ on the radial set, while it is clear that $\beta_1$  and $q$ are both zero near the spacetime boundary. Thus all these conditions are fulfilled for the operator $P$. 

In order to prove microlocal estimates near the radial set, we need to come up with an operator $A$ so that its commutator with $L$, or more exactly the operator on the LHS of \eqref{eq:commutatorAL}, has a positive symbol at the radial set. To this end, we now define \begin{equation}
	\label{eq:symboldef}
	a=\phi(\rho_{\mathcal{R}} + \rhob^2)^2\phi_0(\hat{p})^2\rhob^{-l'}\rhof^{-m'},
\end{equation}
which we may assume is supported in a given small neighbourhood $U$ of $\SR_\newP$, and compute the principal symbol
\begin{equation}\label{eq:commutatorAL}
\sigma_{m+m'-1, l+l'-1}([A,\newP ]+(\newP -\newP ^*)A) = -(H_{\newP}a+2\tilde{p}a),
\end{equation}
 where $A\in\Psip{m'}{l'}$ is the symmetric operator with principal symbol $a$ given by 
 \begin{equation}\label{eq:AA}
 A = (A^{1/2})^2 , \quad A^{1/2} =  \frac{q_L(\sqrt{a}) + q_L(\sqrt{a})^*}{2} .
 \end{equation}
We have 
\begin{align*}
	H_\newP a &= \rho_\mathrm{base}^{1-l}\rho_\mathrm{fib}^{1-m}H^{m,l}a\\
	&= \rho_\mathrm{base}^{1-l-l'}\rho_\mathrm{fib}^{1-m-m'}(2\phi\phi_0^2H^{m,l}\rho_{\mathcal{R}}+2\phi^2\phi_0\phi_0'H^{m,l}\hat p)\\
	&+ \rho_\mathrm{base}^{1-l}\rho_\mathrm{fib}^{1-m}\phi^2\phi_0^2(-l'\rho_\mathrm{base}^{-l'-1}\rho_\mathrm{fib}^{-m'}H^{m,l}\rho_{\mathrm{base}}-m'\rho_{\mathrm{base}}^{-l'}\rho_{\mathrm{fib}}^{-m'-1}H^{m,l}\rho_{\mathrm{fib}})\\
	&= \rho_{\mathrm{base}}^{-l-l'+1}\rho_{\mathrm{fib}}^{-m-m'+1}(\pm 2\phi\phi_0^2\phi_1^2+2\phi^2\phi_0\phi_0'H^{m,l}\hat{p}\pm l'\phi^2\phi_0^2\beta\pm 2m'\phi^2\phi_0^2\beta\beta_1).
\end{align*} 

Hence we can express \eqref{eq:commutatorAL} as
\begin{multline}\label{eq:comm}
	-(H_\newP a+2\tilde{p}a) =\rho_{\mathrm{base}}^{-l-l'+1}\rho_{\mathrm{fib}}^{-m-m'+1} \\
	\times \Big(\mp 2\phi\phi_0^2\phi_1^2-2\phi^2\phi_0\phi_0'H^{m,l}\hat{p}\mp l'\phi^2\phi_0^2\beta\mp 2m'\phi^2\phi_0^2\beta\beta_1\mp 2\beta q \phi^2\phi_0^2 \Big).
\end{multline}
Recall that $\phi_0$ cuts off near $\Sigma_\newP $, and $\phi$ cuts off near $\mathcal{R}$ where $\phi_1'=0$. Hence the first term in \eqref{eq:comm} is supported in a punctured neighbourhood of  the radial set, and the second term involving $\phi_0'$ is supported away from the characteristic set. The latter is easily treated by using microlocal elliptic estimates.

The sum of the final three terms in \eqref{eq:comm} has sign determined by that of
\begin{equation}
\mp(l'+2m'\beta_1+ 2q ).
\label{eq:quantity}\end{equation}
In particular, if $l'+2m'\beta_1+2q > 0$ on $\mathcal{R}$, then this sign matches that of the first term in \eqref{eq:comm}. (Notice that in the case of the specific operator $P$, this quantity is just $\mp l'$.)

We require that the quantity \eqref{eq:quantity} has a definite sign in order to run the positive commutator argument, drawing different conclusions in the two sign cases. 
Suppose that we want to estimate $u$ (or a microlocalized version of $u$) in the $H^{s,r}_{par}$ norm. This requires that we choose $m'$ and $l'$ (the orders of $A$) to satisfy 
\begin{equation}\label{eq:m'l'}
2s= m+m'-1, \quad 2r=l+l'-1
\end{equation}
and recalling that $\beta_1$ vanishes on $\SR$, we require (eliminating $l'$ from \eqref{eq:quantity} that $r+q-\frac{l-1}{2}$ has definite sign on $\mathcal{R}$. 
We obtain estimates for both signs, but the estimates have slightly different characters. 
If $r+q-\frac{l-1}{2}$ is positive, then we obtain microlocal regularity if we assume a priori that $u$ is microlocally in $H^{s,r'}_{\mathrm{par}}$ for some $r' \in [r-1/2, r)$ for which we still have $r'+q-\frac{l-1}{2}>0$. If this quantity is negative, we obtain instead propagation of regularity `towards' the radial set from a punctured neighbourhood of the radial set.

In the case of the specific operator $P$, we have $q=0$ and $0$ so the condition becomes that $r - (-1/2)$ has definite sign. This shows that $r=-1/2$ is a threshold value of the spacetime order, where different behaviours occur above and below this value. 

Returning to the general operator $\newP$, in the case $r+q>\frac{l-1}{2}$, we use  \eqref{eq:comm} and \emph{formally} compute $$\langle i([A,\newP ]+(\newP -\newP ^*)A)u,u\rangle$$ (that is, ignoring regularity conditions for pairing distributions, and integrating by parts) to obtain 
\begin{equation}\label{eq:naive.above.threshold}
	\mp 2\Im \langle Au,\newP u \rangle=\|B_1 u\|^2 + \|B_2 u\|^2 + \langle Fu,u\rangle + \langle Ru,u \rangle .
\end{equation}
Here $B_j=q_L(b_j), F=q_L(f)$, where, using \eqref{eq:m'l'} to replace $m'$ and $l'$ by $s$ and $r$, 
\begin{equation}
	\label{eq:b1def}
	b_1=\phi\phi_0\sqrt{\beta(2r-l+1+2\beta_1(2s-m+1)+2q)}\rhob^{-r}\rhof^{-s}
\end{equation}
\begin{equation}
	\label{eq:b2def}
	b_2=\sqrt{2\phi}\phi_0\phi_1\rhob^{-r}\rhof^{-s}
\end{equation}
\begin{equation}
	f=\pm 2\phi^2\phi_0\phi_0' (H^{m,l}\hat p)\rhob^{-2r}\rhof^{-2s}
\end{equation}
and $R\in \Psip{2s-1}{2r-1}$.

In the special case $\newP u=0$ for example, \eqref{eq:naive.above.threshold} then yields the estimate 
\begin{equation}
	\label{eq:main.ineq}
	\|B_1u\|^2 \leq |\langle Fu,u \rangle|+|\langle Ru,u \rangle|
\end{equation}
Let $Q=\Lambda B_1 \in \Psip{0}{0}$ where $\Lambda=\mathrm{Op}(\rhob^{r}\rhof^{s})$ and take $Q''\in \Psip{0}{0} $ elliptic on $\WF'(A)$. We make the assumption that $Q''u\in H^{s-1/2,r'}$ as foreshadowed above. 
As $\WF'(F) \subset \WF'(A) \subset U$ is disjoint from the characteristic set of $\newP $, we may choose $\tilde Q\in\Psip{0}{0}$ such that $\WF'(\tilde Q) \subset \Ell(Q'')$ and $\tilde Q$ is microlocally equal to the identity on $\WF'(F)$. We can then estimate, for arbitrary $M,N$,   
\begin{multline}
	\label{key2.5}
	|\langle Fu, u\rangle|\leq C \big( |\langle Fu, \tilde Qu\rangle| + \|u\|^2_{M,N}\big) \\
	\leq C \Big( \|Fu\|_{H^{-s+1,-r+1}}^2+\|\tilde Q u\|_{H^{s-1,r-1}}^2+\|u\|^2_{M,N}\Big)\\
	\leq C  \Big( \|GLu\|_{H^{s-m+1,r-l+1}}^2+\|Q'' u\|_{H^{s-1,r-1}}^2+\|u\|^2_{M,N}\Big)
\end{multline} 
using the fact that we can invert elliptic operators microlocally --- see Proposition~\ref{prop:elliptic param}. So, for example, we invert $GL$ microlocally on $\WF'(F)$ to write $F = A GL + R'$ with $A \in \Psip{2s-m}{2r-l}$ and $R' \in \Psip{-\infty}{ -\infty}$. The last term in \eqref{eq:main.ineq} can be estimated similarly:
\begin{equation} \label{key2.75}\begin{aligned}
	|\langle Ru, u\rangle |&\leq C \Big( \|Ru\|_{H^{1/2-s,-r'}}^2 + \|\tilde Q u\|^2_{H^{s-1/2,r'}}+ \|u\|^2_{H^{M,N}} \Big)\\
	&\leq C\Big( \|Q''u\|^2_{H^{s-1/2,r'}}+ \|u\|^2_{H^{M,N}} \Big)
\end{aligned}\end{equation}
where $M,N\in\RR$ are again arbitrary.
Inserting these estimates into \eqref{eq:main.ineq}, we obtain
\begin{equation}\label{key3}
	\|Qu\|_{H^{s,r}}\leq C(\|G\newP u\|_{H^{s-m+1,r-l+1}}+\|Q''u\|_{s-1/2,r'}+\|u\|_{H^{M,N}}).
\end{equation}

In the second case with $r+q < \frac{l-1}{2}$, the above calculation is similar, but \eqref{eq:naive.above.threshold} is replaced with
\begin{equation}\label{eq:naive.below.threshold}
	\mp 2\Im \langle Au,\newP u \rangle=-\|B_1 u\|^2 + \|B_2 u\|^2 + \langle Fu,u\rangle + \langle Ru,u \rangle 
\end{equation}
where \eqref{eq:b1def} is replaced by 
\begin{equation}\label{eq:b1.below}
	b_1=\phi\phi_0\sqrt{\beta(l-2r-1-2\beta_1(2s-m+1)-2q)}\rhob^{-r}\rhof^{-s}
\end{equation}

The changed sign of $B_2$ relative to $B_1$ means that we additionally need microlocal control of $u$ on $\WF'(B_2)$, which lies in a punctured neighbourhood of the radial set $\SR$. This can be achieved by using the standard propagation estimate of Proposition \ref{prop:sing} away from the radial set, and leads to an additional term $\|Q'u\|_{H_\mathrm{par}^{s,r}}$ in the estimate, provided $Q'$ and $G$ satisfy the bicharacteristic condition in Proposition \ref{prop:sing}. 

One can also relax the assumption $\newP u=0$ to $G \newP u \in H^{s-m+1, r-l+1}_{\mathrm{par}}$, which only leads to the additional consideration of the term $\langle Au, \newP u\rangle$ in \eqref{eq:naive.above.threshold} and \eqref{eq:naive.below.threshold}.
We  absorb the contribution of this term into the positivity of $b_1$, by replacing the symbol $b_1$ with $\tilde{b}_1^2=b_1^2-\delta a \rhob^{l'-2r}\rhof^{m'-2s}>0$ for sufficiently small $\delta$. Then we have 
\begin{equation}
	\label{eq:main.ineq2}
	\|\tilde{B}_1u\|^2 \leq -\delta  \|\Lambda A^{1/2}u\|^2 +|\langle Fu,u \rangle|+|\langle Ru,u \rangle|+2|\Im \langle A^{1/2}u,A^{1/2}\newP u\rangle|
\end{equation}
where $\Lambda=\mathrm{Op}(\rhob^{l'/2-r}\rhof^{m'/2-s})$ and $A^{1/2}$ is given by \eqref{eq:AA}. 
Taking $\tilde\Lambda$ an elliptic parametrix to $\Lambda$, we have
\begin{align}
	|\langle A^{1/2}u, A^{1/2}\newP u\rangle|&\leq |\langle \Lambda A^{1/2}u, \tilde\Lambda A^{1/2}\newP u\rangle|\\
	&\leq \frac{\delta}{2}\|\Lambda A^{1/2}u\|^2+\frac{1}{2\delta}\|\tilde\Lambda A^{1/2}\newP u\|^2
\end{align}
and the first of these terms is absorbed by the first term on the right hand side of \eqref{eq:main.ineq2}, whilst the latter is bounded (recalling $G$ is elliptic on $U \supset \WF'(A)$) by
$$
\frac{C}{2\delta}\|G \newP u\|_{H^{s-m+1,r-l+1}} + C \|u\|_{H_\mathrm{par}^{M,N}}.
$$

We now state and prove the propagation result in the two cases. Our assumptions are as above for both results; that is, we assume that $L\in\Psipcl{m}{l}(\RR^{n+1})$ is a classical pseudodifferential operator in the parabolic scattering calculus that has real principal symbol, such that its radial set $\SR_\newP$ is a codimension $k$ submanifold of $\Sigma(\newP)$ contained in $\partial_{\mathrm{base}} \overline{T}^*_{\mathrm{fib}}\mathbb{R}^{n + 1}$ that meets $\partial_{\mathrm{base}} \overline{T}^*_{\mathrm{fib}}\mathbb{R}^{n + 1}$ transversally. We assume that $\SR_\newP$ is either a source or a sink for the rescaled Hamilton vector field in the sense described above, and for 
either the top (sink) or bottom (source) sign choices in \eqref{eq:phi1}, \eqref{eq:beta}, \eqref{eq:beta1}, we assume that $\phi_1$,$\beta_1$ are smooth and vanish near $\SR$ and on $\SR$ respectively, and that $\beta$ is smooth and positive on $\SR$.

\begin{prop}
	\label{prop:below}
	Suppose $L\in\Psip{m}{l}(\RR^{n+1})$ is as above. 	For the $q$ defined in \eqref{eq:subprincipal}, suppose that  $r+q<\frac{l-1}{2}$ on $\SR_\newP$. Assume that there exists a neighbourhood $U$ of $\SR_\newP$ and $Q', Q'', G\in\Psi_{\mathrm{par}}^{0,0}(\RR^{n+1})$, with $U \subset \Ell(Q'')$ and such that for every $\alpha\in p^{-1}(0)\cap U \setminus \SR_\newP$ the bicharacteristic $\gamma$ through $\alpha$ enters $\Ell(Q')$ whilst remaining in $\Ell(G)$. Then there exists $Q\in\Psi_{\mathrm{par}}^{0,0}(\RR^{n+1})$ elliptic on $\SR_\newP$ such that if $u \in H^{M,N}$, $Q'u\in H^{s,r}$, $Q''u\in H^{s-1/2,r-1/2}$ and $G\newP u\in H^{s-m+1,r-l+1}$, then $Qu\in H^{s,r}$ and  there is $C>0$ such that
	\begin{equation}\label{key superfluous label 2}
		\|Qu\|_{H_\mathrm{par}^{s,r}}\leq C(\|Q'u\|_{H_\mathrm{par}^{s,r}} +\|G\newP u\|_{H_\mathrm{par}^{s-m+1,r-l+1}}+\|Q''u\|_{H_\mathrm{par}^{s-1/2,r-1/2}}+\|u\|_{H_\mathrm{par}^{M,N}}).
	\end{equation}
\end{prop}

\begin{prop}
	\label{prop:above}
	Suppose $L \in\Psipcl{m}{l}(\RR^{n+1})$ is as above. 	For the $q$ defined in \eqref{eq:subprincipal}, suppose $r+q>\frac{l-1}{2}$ on $\SR_\newP$ and moreover $r' +q > \frac{l-1}{2}$ on $\SR_\newP$ for some $r'\in [r-1/2,r)$. Assume that  $Q'', G\in\Psi_{\mathrm{par}}^{0,0}(\RR^{n+1})$ are elliptic at $\SR_\newP$. 
	Then there exists $Q\in\Psi_{\mathrm{par}}^{0,0}(\RR^{n+1})$, elliptic at $\SR_\newP$, such that  if $u \in H^{M,N}$, $Gu\in H_\mathrm{par}^{s-1/2,r'}$ and $G\newP u\in H_\mathrm{par}^{s-m+1,r-l+1}$, then $Qu\in H_\mathrm{par}^{s,r}$ and there is $C>0$ such that
	\begin{equation}\label{key 3}
		\|Qu\|_{H_\mathrm{par}^{s,r}}\leq C(\|G\newP u\|_{H_\mathrm{par}^{s-m+1,r-l+1}}+\|Q''u\|_{H_\mathrm{par}^{s-1/2,r'}}+\|u\|_{H_\mathrm{par}^{M,N}}).
	\end{equation}
\end{prop}

\begin{remark} In the statement of Proposition~\ref{prop:above}, we could take $Q'' = G$. However, in order to treat the proofs of these two results jointly, it helps to state the result as above. 
\end{remark}

\begin{proof}
	The proof of these two results largely amounts to regularising the commutator estimates outlined above in order to legitimately obtain the equality  \eqref{eq:naive.above.threshold}. 
	
	Note that a priori, the conditions of Proposition \ref{prop:above} only imply that $\newP u$ and $Au$ have orders $(s-m+1,r-l+1)$ and $(s-1/2-m', r'-l')$ in $\WF'(A)$, summing to $(-1/2,r'-r)$ (using \eqref{eq:m'l'}), and so $\langle Au,\newP u\rangle$ is not a priori well-defined. This requires some regularization procedure added to the formal calculations above. 
	
	To deal with this issue, we replace the symbol $a$ in \eqref{eq:symboldef} with
	\begin{equation}\label{eq:regularised}
		a_\ep=\varphi_\ep(\rhob^{-1})^2 \tilde\varphi_\ep(\rhof^{-1})^2 a:=(1+\epsilon\rhob^{-1})^{r'-r}(1+\epsilon\rhof^{-1})^{-1/2}a
	\end{equation}
	for $\epsilon\geq 0$. Thus for each fixed $\epsilon$, the order of $A_\epsilon$ obtained from $a_\epsilon$ analogously to \eqref{eq:AA} has been shifted by $(-1/2, r'-r)$ relative to $A$. 
	
	These regularising functions have the property that 	
	\begin{equation}
		\label{eq:reg.bound.1}
		\varphi_\ep'/\varphi_\ep\leq \frac{r-r'}{2}\cdot \min\{1,\rhob\}
	\end{equation}
	and
	\begin{equation}
		\label{eq:reg.bound.2}\tilde\varphi_\ep'/\tilde\varphi_\ep\leq \frac{1}{4}\cdot \min\{1,\rhof\}.
	\end{equation}
	Due to the regularisation, the pairing $\langle A_\ep u,\newP U\rangle $ is now well defined for $\ep>0$, however the formal integration by parts, i.e.\ the identity
	\begin{equation}\label{eq:commidentity}
		\ang{\newP u,A_\ep u}-\ang{A_\ep u,\newP u}=\ang{(A_\ep \newP -\newP ^*A_\ep)u, u}
	\end{equation}
	still remains to be justified for fixed $\ep>0$. To do this, we use the functions $\varphi_\ep$,  $\tilde\varphi_\ep$ just defined and raise them to a sufficiently high fixed power $K$: let 
 $\Gamma_t=q_L(\varphi_t^K(\rhob^{-1})\tilde \varphi_t^K(\rhof^{-1}))$, and compute
	\begin{align*}
		\ang{\newP u,A_\ep u}-\ang{A_\ep u,\newP u}&= \lim_{t\rightarrow 0} \Big( \ang{\Gamma_t \newP u,A_\ep u}-\ang{\Gamma_t A_\ep u,\newP u} \Big) \\
		&= \lim_{t\rightarrow 0} \ang{(A_\ep\Gamma_t \newP -\newP ^*\Gamma_tA_\ep)u, u}\\
	\end{align*}
	We have 
	\begin{equation}\label{eq:reg2}
		A_\ep\Gamma_t \newP -\newP ^*\Gamma_tA_\ep=\Gamma_t(A_\ep \newP  -\newP ^* A_\ep)+[A_\ep,\Gamma_t]\newP -[\newP ^*,\Gamma_t]A_\ep.
	\end{equation}
	As $\Gamma_t$ is uniformly bounded in $\Psip{0}{0}$ (in the sense of having symbol with uniformly bounded seminorms), and converges to $\Id$ as $t \to 0$ in $\Psip{\mu}{\mu}$ for any $\mu > 0$, we have strong convergence $A_\ep\Gamma_t \newP -\newP ^*\Gamma_tA_\ep\rightarrow A_\ep \newP -\newP ^*A_\ep$
	which implies \eqref{eq:commidentity} is valid for each $\ep>0$.
	
	As before, we compute the symbol $-(H_p a_\ep+2\tilde p a_\ep)$ of the commutator expression $i([A_\ep,\newP ]+(\newP -\newP ^*)A_\ep)$ where  $A_\ep$ is symmetric with principal symbol $a_\ep$.
	
	The calculation proceeds as before from \eqref{eq:comm}, however in each term there is the regularising factor $\varphi_\ep^2\tilde\varphi_\ep^2$, and there are two additional terms from $H^{m,l}$ falling on the regularisers. These two terms are:
	\begin{equation}
		\mp 2 \rhob^{1-l-l'}\rhof^{1-m-m'}\phi^2\phi_0^2\varphi_\ep^2\tilde\varphi_\ep^2 \cdot (\varphi_\ep'/\varphi_\ep)\cdot \rhob^{-1}\beta 
	\end{equation}
	and
	\begin{equation}
		\mp 4 \rhob^{1-l-l'}\rhof^{1-m-m'}\phi^2\phi_0^2\varphi_\ep^2\tilde\varphi_\ep^2 \cdot (\tilde\varphi_\ep'/\tilde\varphi_\ep)\cdot \rhof^{-1}\beta \beta_1
	\end{equation}
	In the case $r+q<\frac{l-1}{2}$, these new terms in fact have the same sign as the `main' term $\|B_1u\|^2$, and can thus be dropped from the commutator estimate.
	
	In the case $r+q > \frac{l-1}{2}$, these new terms have the opposite sign of $\|B_1u\|^2$, but can be absorbed into the $b_1$ term as in \eqref{eq:main.ineq2}. To see this, we use \eqref{eq:reg.bound.1} and \eqref{eq:reg.bound.2} to obtain the estimate 
	\begin{equation}
		\frac{2\varphi_\ep '}{\varphi_\ep}+\frac{4\beta_1 \tilde{\varphi}_\ep'}{\tilde{\varphi}_\ep}\leq r-r'+\beta_1.
	\end{equation}
	Since the expression underneath the square root in \eqref{eq:b1def} remains positive near $\SR_\newP$ if we replace the $r$ with an $r'$, it follows that, provided we have $r'+q > \frac{l-1}{2}$, the two additional terms can be absorbed into the positive expression $b_1^2$.
	
	The remaining terms in
        \eqref{eq:naive.above.threshold}, \eqref{eq:naive.below.threshold}
        are now generally $\ep$-dependent, which we denote with a
        subscript. Replacing $\WF'(A)$ and $\WF'(F)$
         with their uniform versions $\WF'(A_\ep),\WF'(F_\ep)$
        (in the sense of \cite[Sect.\ 4.4]{grenoble}), then all estimates go through uniformly in $\ep$ to give
	\begin{equation}
		\|Q_\ep u\|_{H^{s,r}}\leq C(\|G\newP u\|_{H^{s-m+1,r-l+1}}+\|Q''u\|_{s-1/2,r'} +\|u\|_{H^{M,N}})
	\end{equation}
	for fixed $C$ and all $\ep > 0$ in the case $r+q>r'+q > \frac{l-1}{2}$.
	
	From weak compactness of the unit ball in $H^{s,r}$, as $\ep\rightarrow 0$ we thus obtain a limit $\lim_{\ep_j\rightarrow 0} Q_\ep u$ in $H^{s,r}$. However, we know that $Q_\epsilon u$ converges strongly (and a fortiori weakly) to $Qu$ in $H^{s-\epsilon', r - \epsilon'}$ for any $\epsilon' > 0$. 
	It follows that limit obtained from weak compactness is $Qu$, and hence $Qu$ lies in $H^{s,r}$ satisfying the required estimate 
	\begin{equation}
		\|Q u\|_{H^{s,r}}\leq C(\|G\newP u\|_{H^{s-m+1,r-l+1}}+\|Q''u\|_{s-1/2,r'} +\|u\|_{H^{M,N}}).
	\end{equation}
	
	In the case $r+q < \frac{l-1}{2}$, the exact same argument goes through, with the additional term $\|B_2u\|^2$ in \eqref{eq:naive.below.threshold} controlled using Proposition \ref{prop:sing} and giving rise to the additional term $\|Q'u\|_{H_\mathrm{par}^{s,r}}$ in Proposition \ref{prop:below}.
	
\end{proof}

\begin{remark}\label{rem:iteration} The conclusions of Propositions~\ref{prop:above} and \ref{prop:below} lend themselves to iteration. In fact, if we know that $Q'u$ (microsupported away from the radial set, say) and $G Lu$ have regularity $(s^*, r^*)$ and $(s^* -m+1, r^*-l+1)$ respectively, then we can apply \eqref{key superfluous label 2} over and over, starting at the assumed regularity $(M,N)$,  to obtain 
regularity $(s^*, r^*)$ for $Qu$, provided that $r^* + q < (l-1)/2$. Similarly, 
if $G Lu$ has very high regularity, say of order $(s^* -m+1, r^*-l+1)$ where $s^*$ and $r^*$ are very large, then we simply apply \eqref{key 3} over and over, gaining up to $1/2$ in both spacetime and fibre regularity until we reach $s^*$ and $r^*$, provided we know a priori that we have regularity $(s_0, r_0)$ with $r_0 + q > (l-1)/2$.  The results obtained by such iteration are stated for the specific operator $P$ in Propositions~\ref{prop:aboveschrod} and \ref{prop:belowschrod}. Notice that the $Q''$ term is not needed in the below-threshold case as it can be iteratively lowered in regularity until it is subsumed into the $\|u\|_{H^{M,N}}$ term. 
\end{remark}

\begin{remark}
The example of the free Schr\"odinger operator $P_0$ shows that these propositions cannot be improved much. Consider a solution $u$ to $P_0u = 0$ given by $\mathcal{P}_0 f$ for $f$ Schwartz, as in \eqref{eq:poisson}. The function $u$ is microlocally trivial outside the radial set $\SR$, and is in $H^{s, r}_{\mathrm{par}}$ for every $r < -1/2$, but not for $r \geq -1/2$. It shows that the a priori condition that $u \in H^{s', r'}_{\mathrm{par}}$ for some $r' > -1/2$ in Propositions~\ref{prop:above} and \ref{prop:aboveschrod} cannot be removed in order to gain additional regularity at the radial set, and also shows that regularity gain in Propositions~\ref{prop:below} and \ref{prop:belowschrod} cannot be pushed above the threshold level of $-1/2$. 
\end{remark}

\bibliographystyle{plain}
\bibliography{tds}

\end{document}